\documentclass[11pt]{article}
\usepackage[hscale=0.8,vscale=0.8]{geometry}
\usepackage{enumitem}
\usepackage{subcaption}
\usepackage{booktabs}
\usepackage{multirow}
\usepackage{amsmath, amssymb}
\usepackage{amsthm}
\usepackage[hidelinks]{hyperref}
\usepackage{cleveref}
\usepackage{graphics,graphicx}
\usepackage{bm}
\usepackage{color}
\usepackage{cite}
\usepackage{siunitx}
\usepackage{todonotes}
\newtheorem{theorem}{Theorem}[section]
\newtheorem{lemma}{Lemma}[section]
\newtheorem{remark}{Remark}[section]
\newtheorem{corollary}{Corollary}[section]

\numberwithin{equation}{section}

\newcommand{\myR}{\mathcal{R}}
\newcommand{\myQ}{\mathcal{Q}}

\newcommand{\myE}{\mathcal{E}}
\renewcommand{\div}{\operatorname{div}}

\DeclareMathOperator{\Tr}{Tr}

\newcommand{\bu}{\bm{u}}
\newcommand{\bv}{\bm{v}}
\newcommand{\bx}{\bm{x}}
\newcommand{\br}{\bm{r}}
\newcommand{\bq}{\bm{q}}
\newcommand{\bnu}{\bm{\nu}}

\newcommand{\bphi}{\bm{\phi}}
\newcommand{\bpsi}{\bm{\psi}}

\def\rh{\hat{\br}}
\def\RT{\text{RT}}
\def\hE{\hat{\omega}}
\def\he{\hat{\gamma}}
\def\hv{\hat\bv}
\def\DF{D\!F}

\begin{document}

\title{Flux-mortar mixed finite element methods with multipoint flux approximation}

\author{Wietse M. Boon\footnotemark[1]
  \and
  Dennis Gl\"aser\footnotemark[2]
  \and
  Rainer Helmig\footnotemark[2]
  \and
  Ivan Yotov\footnotemark[3]}
\renewcommand{\thefootnote}{\fnsymbol{footnote}}

\footnotetext[1]{Politecnico di Milano, Italy; \texttt{wietsemarijn.boon@polimi.it}}
\footnotetext[2]{University of Stuttgart, Germany; \texttt{dennis.glaeser@iws.uni-stuttgart.de, rainer.helmig@iws.uni-stuttgart.de}}
\footnotetext[3]{University of Pittsburgh, USA; \texttt{yotov@math.pitt.edu}}

\renewcommand{\thefootnote}{\arabic{footnote}}

\maketitle

\begin{abstract}
  The flux-mortar mixed finite element method was recently developed in \cite{boon2020flux} for a general class of domain decomposition saddle point problems on non-matching grids. In this work we develop the method for Darcy flow using the multipoint flux approximation as the subdomain discretization. The subdomain problems involve solving positive definite cell-centered pressure systems. The normal flux on the subdomain interfaces is the mortar coupling variable, which plays the role of a Lagrange multiplier to impose weakly continuity of pressure. We present well-posedness and error analysis based on reformulating the method as a mixed finite element method with a quadrature rule. We develop a  non-overlapping domain decomposition algorithm for the solution of the resulting algebraic system that reduces it to an interface problem for the flux-mortar, as well as an efficient interface preconditioner. A series of numerical experiments is presented illustrating the performance of the method on general grids, including applications to flow in complex porous media.
\end{abstract}


\section{Introduction}
\label{sec:introduction}

The flux-mortar mixed finite element method \cite{boon2020flux} is a domain decomposition method that allows for arbitrarily non-matching grids between the subdomains. In the context of Darcy flow, by choosing the normal flux as the interface mortar coupling variable, the mass conservation properties of the subdomain discretizations are preserved. The mortar variable is used as a Lagrange multiplier to impose weakly continuity of pressure. The method can be implemented via a non-overlapping domain decomposition algorithm by reducing the global system to the solution of a symmetric and positive definite interface problem for the mortar variable. When using a Krylov space iterative method, such as the conjugate gradient, for the solution of the interface problem, each iteration requires solving subdomain problems with normal flux boundary conditions on the interior interfaces, which can be done in parallel. 

The flux-mortar mixed finite element method is dual to the pressure-mortar mixed finite element method \cite{ACWY,APWY}, where the interface pressure is the coupling variable used to impose weakly continuity of normal flux. Both approaches were originally proposed in \cite{GW} in the case of matching subdomain grids. While the pressure-mortar mixed finite element method has been extensively studied, including multiphase and multiphysics flows in porous media \cite{PWY}, nonlinear elliptic \cite{ArshadParkShin-elliptic} and parabolic \cite{ArshadParkShin-parabolic} problems, mixed elasticity \cite{KhatYot} and poroelasticity \cite{JKY}, Stokes--Darcy flows \cite{GVY,VWY}, flow in fractured porous media \cite{AhmFumBud}, Stokes--Biot couplings \cite{AKYZ}, multiscale mortar multipoint flux mixed finite element discretizations \cite{wheeler2012multiscale}, mortar mimetic finite difference methods \cite{BLSWY}, and coupling with DG methods \cite{GSWY}, the flux-mortar mixed finite element method has only recently received an increased attention. It has been applied in the context of fracture flows \cite{boon2018robust} and coupled Stokes-Darcy flows \cite{Boon2019StokesDarcy,boon2020flux}. The flux-mortar mixed finite element method is related to the subgrid upscaling method proposed in \cite{arbogast2004analysis}. Moreover, the method has similarities with the multiscale hybrid-mixed (MHM) method with local mixed solves \cite{duran2019multiscale,MHM-two-phase} and we refer the interested reader to \cite[Sec. 1]{boon2020flux} for an exposition of these relations. 

In this paper we develop the flux-mortar mixed finite element method for Darcy flow using the multipoint flux approximation (MPFA) method as the subdomain discretization. The proposed method thereby combines the widely used MPFA method with the mass conservative flux-mortar domain decomposition approach. The MPFA method \cite{aavatsmark2002introduction,aavatsmark2007convergence,aavatsmark1998discretization,edwards1998finite,edwards2002unstructured} is a finite volume method that can handle general polytopal grids and full permeability tensors, which may be discontinuous from element to element. As is common with finite volume discretizations, the MPFA method preserves mass locally.
On simplicial, quadrilateral, and hexahedral elements, the MPFA method has been related to the multipoint flux mixed finite element (MFMFE) method \cite{WheYot,IngWheYot,WheXueYot}, which uses a Brezzi--Douglas--Marini BDM$_1$-type space for the velocity and a vertex quadrature rule for the velocity bilinear form. On general polytopal grids, the MPFA method has been formulated and analyzed as a mimetic finite difference (MFD) method in \cite{LSY}.

We refer to our method as the flux-mortar MFMFE method, since we utilize the relation of the MPFA method to the MFMFE method in our analysis. We note that our analysis allows for different polynomial degrees and grids for the subdomain and mortar discretizations. In particular, the mortar grid can be chosen to be coarser than the subdomain grids, resulting in a multiscale approximation. This feature is dual to the multiscale mortar mixed finite element method with pressure-mortar developed in \cite{APWY,PWY}. In our case, the solution is approximated locally on the fine scale, while continuity of pressure is imposed on the coarse scale. From computational point of view, the coarse mortar grid results in a smaller interface problem. Furthermore, even though this is beyond the scope of the paper, similarly to the multiscale flux basis in \cite{GanYot}, one can precompute a multiscale pressure basis by solving local Neumann subdomain problems for each flux-mortar degree of freedom. Thus, the computational cost of the method can be made comparable to existing multiscale mixed finite element methods \cite{ChungEfeLee,Aarnes-etal,arbogast2004analysis,CVD-MPFA,duran2019multiscale}. We emphasize that our method provides extra flexibility due to the non-matching subdomain and mortar grids and different polynomial degrees for subdomain and mortar discretizations. 

The main contributions of this work are as follows. First, we carry out the \emph{a priori} analysis of the flux-mortar MFMFE method, including its stability and the error estimates. We consider simplicial grids as well as smooth quadrilateral and hexahedral grids. In both cases, the multipoint approximation of the flux leads to an additional quadrature error term, which we bound appropriately. 
Second, we present a non-overlapping domain decomposition algorithm for the solution of the resulting algebraic system that reduces it to an interface problem for the flux-mortar, and develop an efficient preconditioner for the interface problem. The interface operator requires solving Neumann subdomain problems at each iteration with a flux boundary condition on the interior interfaces. The preconditioner involves solving Dirichlet subdomain problems with specified pressure on the interfaces. Both sets of subdomain solves can be done in parallel, resulting in a scalable algorithm for distributed memory parallel computers. The numerical results show that the number of iterations of the preconditioned Krylov method for the interface problem exhibits a very weak dependence on the discretization parameter.  
Third, we present numerical experiments that verify the expected convergence of the method as well as showcase the applicability of the method for involved porous media flow problems and general grids. We test the method for faulted geologies in two and three dimensions with low and high permeable faults, where the subdomain and mortar grids are suitably chosen. We illustrate the multiscale capability of the method for one of the Society of Petroleum Engineers SPE10 benchmark problems. Finally, we consider flow in a heterogeneous porous medium, where the different subdomain and mortar grids are locally chosen to resolve the variability of the permeability.

The article is organized as follows. \Cref{sec: Model problem} introduces the model problem and the notation conventions. The flux-mortar MFMFE method is proposed in \Cref{sec: Numerical method}. \Cref{sec: Analysis simplices} presents the \emph{a priori} analysis of the method on simplicial grids. The extension to quadrilateral and hexahedral grids is discussed in \Cref{sec: Analysis quads}. In \Cref{sec: DD} we present the domain decomposition algorithm and develop the preconditioner for the interface problem. Finally, \Cref{sec: numerical results} shows the performance of the method through the use of four numerical test cases and \Cref{sec: Conclusions} contains the conclusions.

\section{Model problem}
\label{sec: Model problem}

Let $\Omega \subset \mathbb{R}^n$, $n = 2,3$ be a bounded polygonal
domain. $\Omega$ is decomposed into disjoint polygonal subdomains $\Omega_i$
indexed with $i \in I_\Omega = \{1, 2, \ldots, n_\Omega \}$. Let $\bm{\nu}_i$
denote the outward unit vector normal to the boundary $\partial
\Omega_i$. The $(n - 1)$-dimensional interface between two subdomains
$\Omega_i$ and $\Omega_j$ is denoted by $\Gamma_{ij} :=
\partial \Omega_i \cap \partial \Omega_j$. Each interface
$\Gamma_{ij}$ is assumed to be Lipschitz and endowed with a unique,
unit normal vector $\bm{\nu}$ such that
%
$
	\bm{\nu} := \bm{\nu}_i = -\bm{\nu}_j 
	\text{ on } \Gamma_{ij}, \ i < j.
        $
%
Let $\Gamma := \bigcup_{i < j} \Gamma_{ij}$ and $\Gamma_i := \Gamma
\cap \partial \Omega_i$.  We categorize $\Omega_i$ as an interior
subdomain if $\partial \Omega_i \subseteq \Gamma$, i.e. if none of its
boundaries coincide with the boundary of the domain $\Omega$.
Let $I_{int} := \{ i \in I_\Omega :\ \partial \Omega_i \subseteq \Gamma \}$.

We will use the following notation. A subscript $i$ on a variable denotes its restriction to $\Omega_i$, i.e. $w_i := w|_{\Omega_i}$.
For $G$ a domain in
$\mathbb{R}^n$, $n = 2,3$, or a manifold in $\mathbb{R}^{n-1}$,
the Sobolev spaces on $G$ are denoted by $W^{k, p}(G)$. Let
$H^k(G) := W^{k, 2}(G)$ and $L^2(G) := H^0(G)$. The $L^2(G)$-inner product is
denoted by $(\cdot,\cdot)_G$. For $G \subset \mathbb{R}^n$, let
\begin{equation*}
H(\div, G) = \{\bm{v} \in (L^2(G))^n:
\nabla\cdot \bm{v} \in L^2(G)\}.
\end{equation*}
We use the following shorthand notation to denote the norms of these spaces:
	\begin{align*}
		\| f \|_{k, G} &:= \| f \|_{H^k(G)}, &
		\| f \|_{G} &:= \| f \|_{0, G}, &
		\| \bm{v} \|_{\div, G}^2 & := \| \bm{v} \|_{H(\div, G)}^2 
                = \| \bm{v} \|_{G}^2 + \| \nabla \cdot \bm{v} \|_{G}^2.
	\end{align*}
%
We use the binary relation $a \gtrsim b$ to imply that a
constant $C > 0$ exists, independent of the mesh size $h$, such that
$a \gtrsim C b$. The relationship $\lesssim$ is defined analogously.

The model problem for single-phase flow in porous media is
\begin{equation}\label{model}
  \bu = - K \nabla p, \quad \nabla\cdot \bu = f \ \mbox{ in } \Omega,
  \quad p = 0 \ \mbox{ on } \partial\Omega,
  \end{equation}
where $\bm{u}$ is the Darcy velocity, $p$ is
the pressure, $K$ is a uniformly bounded symmetric positive-definite
tensor field representing the conductivity, and $f\in L^2(\Omega)$ is a source function. We assume
that there exist $0 < k_{min} \le k_{max} < \infty$ such that $\forall \bx \in \Omega$,
\begin{equation}\label{K-spd}
  k_{min} \xi^T\xi \le \xi^T K (\bx)\xi \le k_{max} \xi^T\xi, \quad \forall
  \xi \in \mathbb{R}^n.
\end{equation}
The variational formulation of \eqref{model} is: Find
$(\bm{u}, p) \in V \times W := H(\div, \Omega) \times L^2(\Omega)$ such that
\begin{subequations}\label{weak-model}
	\begin{align}
		(K^{-1} \bm{u}, \bm{v})_\Omega
		- (p, \nabla \cdot \bm{v})_\Omega
		&= 0,
		& \forall \bm{v} &\in V, \label{weak-model-1}\\
		(\nabla \cdot \bm{u}, w)_\Omega
		&= (f, w)_\Omega,
		& \forall w &\in W. \label{weak-model-2}
	\end{align}
\end{subequations}
It is well known that \eqref{weak-model} has a unique solution \cite{boffi2013mixed}. Letting
\begin{equation*}
a_i(\bm{u}_i, \bm{v}_{i}) := (K^{-1} \bm{u}_i, \bm{v}_{i})_{\Omega_i}, \ \
a(\bm{u}, \bm{v}) := \sum_i  a_i(\bm{u}_i, \bm{v}_{i}), \ \
b_i(\bm{v}_{i}, p_i) := (\nabla \cdot \bm{v}_{i}, p_i)_{\Omega_i}, \ \
b(\bm{v}, p) := \sum_i b_i(\bm{v}_{i}, p_i),
\end{equation*}
the system \eqref{weak-model} can be written as
\begin{subequations} \label{eq: general form}
	\begin{align}
		\sum_i a_i(\bm{u}_i, \bm{v}_{i})
		- \sum_i b_i(\bm{v}_{i}, p_i)
		&= 0,
		& \forall \bm{v} &\in V, \label{eq: general form 1}\\
		\sum_i b_i(\bm{u}_i, w_i)
		&= (f, w)_\Omega,
		& \forall w &\in W. \label{eq: general form 2}
	\end{align}
\end{subequations}

For given $\Omega_i$, the local velocity and pressure function spaces
are defined as
$V_i := H(\div, \Omega_i)$ and
$W_i := L^2(\Omega_i)$, respectively. The global space $V$ possesses
continuity of the normal trace on $\Gamma$. In particular, it holds that
\begin{align} \label{eq: characterization V}
V = \left\{ \bm{v} \in \bigoplus_i V_i:\ 
	\bm{\nu}\cdot \bm{v}_i = \bm{\nu}\cdot \bm{v}_j \ \text{on each } \Gamma_{ij}
	\right\}.
\end{align}
The normal flux $\bm{\nu}\cdot \bm{u}$ on $\Gamma$ will be modeled by a Lagrange multiplier $\lambda \in \Lambda$, with
\begin{align*}
\Lambda := L^2(\Gamma).
\end{align*}
We note that $\Lambda$ has more regularity than the normal trace of $V$ on $\Gamma$, which is utilized in the numerical scheme. For $\lambda \in \Lambda$, we use a subscript to indicate its relative
orientation with respect to the adjacent subdomains:
\begin{align*}
 	\lambda_i := \lambda, \ \
 	\lambda_j := -\lambda \ \
 	\text{ on } \Gamma_{ij},\ i<j.
\end{align*}
In particular, $\lambda_i$ models $\bnu_i\cdot \bu$ and $\lambda_j$ models
$\bnu_j\cdot \bu$ on $\Gamma_{ij}$.

Next, we associate appropriate norms to the function spaces. The spaces $W$
and $\Lambda$ are equipped with the standard $L^2(\Omega)$ and
$L^2(\Gamma)$ norms, respectively, and the space $V$ is equipped with a broken $H(\div)$
norm:
\begin{align*}
	\| \bm{v} \|_V &:= \sum_i \| \bm{v}_{i} \|_{\div, \Omega_i}, &
	\| w \|_W &:= \| w \|_{\Omega}, &
	\| \mu \|_\Lambda &:= \| \mu \|_{\Gamma}.
\end{align*}

\section{Numerical method}
\label{sec: Numerical method}

We next describe the flux-mortar multipoint flux mixed finite element method for 
\eqref{eq: general form}, based on \cite{boon2020flux}. We first present the subdomain discretization, followed by the discretization of the interface variables, and end with the definition of the flux-mortar method.

\subsection{Subdomain discretization}
\label{sub: Subdomain discretization}

For a subdomain $\Omega_i$, let $\Omega_{h,i}$ be a shape-regular polytopal
tessellation with typical mesh size $h$. The grids $\Omega_{h,i}$ and $\Omega_{h,j}$ may be
non-matching along the interface $\Gamma_{ij}$. 
For the subdomain discretizations we employ the MPFA method \cite{aavatsmark2002introduction,aavatsmark1998discretization,edwards1998finite,edwards2002unstructured}. More specifically, we will use the MPFA-O method, but omit the suffix for the sake of readability. After describing the MPFA method in \Cref{ssub: MPFA}, we introduce the MFMFE method in \Cref{ssub: MFMFE} which provides the setting for \emph{a priori} analysis. 

\subsubsection{Multipoint flux approximation finite volume method} 
\label{ssub: MPFA}

In order to describe the MPFA Finite Volume method, we briefly introduce some geometric notation. Let $\bx_i$ be the barycenter of cell $\omega_i$. Let $\gamma_{ij}$ be the facet between cells $\omega_i$ and $\omega_j$ with barycenter $\bx_{ij}$ and unit normal $\bm{\nu}_{ij}$. The vertices of the mesh are denoted $\br_k$ with coordinates $\br_k$. A dual grid is created by connecting each $\bx_i$ with all $\bx_{ij}$ and, if $n = 3$, with the barycenters of the edges. Each facet $\gamma_{ij}$ is thereby subdivided into sub-facets $\gamma_{ij,k}$ in which the index $k$ indicates that $\br_k$ is a vertex of $\gamma_{ij,k}$.

\begin{figure}
    \centering
    \includegraphics[width=0.5\textwidth]{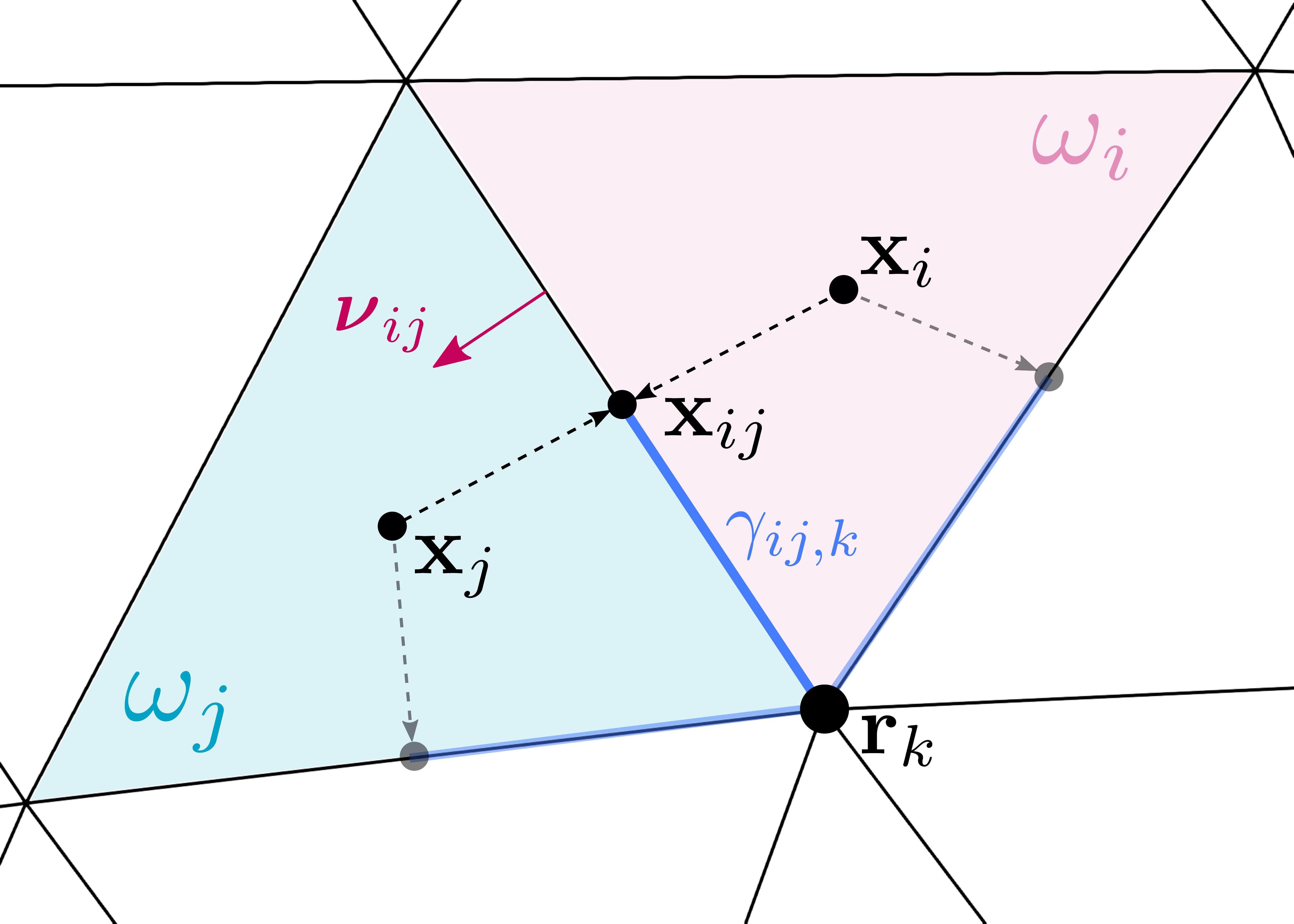}
    \caption{Illustration of the definitions used in the description of the MPFA method.}
    \label{fig:mpfa notation}
\end{figure}

The MPFA method is constructed as follows. First, we endow each $\gamma_{ij,k}$ with a sub-facet pressure $p_{ij,k}$. With these sub-facet pressures, we compute a discrete gradient $\nabla_{i,k}^h p \in \mathbb{R}^n$ that satisfies
\begin{align} \label{eq: discrete gradient}
	(\nabla_{i,k}^h p) \cdot (\bx_{ij} - \bx_i) = p_{ij,k} - p_i,
\end{align}
for all indices $j$ such that $\omega_j$ is a neighbor of $\omega_i$ that shares vertex $\br_k$. In order for $\nabla_{i,k}^h p$ to be well-defined, we assume that each cell-vertex pair $(i, k)$ has exactly $n$ sub-facets $\gamma_{ij, k}$. This is true for polygonal grids in 2D and for simplicial and hexahedral grids in 3D. 

Second, we formulate the flux continuity condition across each sub-facet $\gamma_{ij,k}$. Let $K_i$ be the conductivity in cell $\omega_i$, then this condition is given by
\begin{align} \label{eq: subface continuity}
    \bm{\nu}_{ij} \cdot (K_i \nabla_{i,k}^h p - K_j \nabla_{jk}^h p) = 0.
\end{align}

Third, we fix the vertex index $k$ and collect \eqref{eq: subface continuity} for all neighboring sub-facets $\gamma_{ij,k}$. Combined with \eqref{eq: discrete gradient}, this allows us to define the sub-facet pressures as a linear combination of the cell-center pressures. In turn, a substitution in \eqref{eq: discrete gradient} effectively eliminates the sub-facet pressures and defines the discrete gradient $\nabla_{i,k}^h$ as a linear operator acting on the cell-center pressures.

Finally, we let the discrete gradient define the flux across all facets $\gamma_{ij}$. Reusing our convention $\bm{\nu}_i = \bm{\nu}_{ij} = - \bm{\nu}_j$ for $i < j$, the MPFA method then solves 
\begin{align}
	\sum_{j, k} \int_{\gamma_{ij,k}} (- \bm{\nu}_i \cdot (K_i \nabla_{i,k}^h p)) 
	&= \int_{\omega_i} f, 
	&\forall \omega_i &\in \Omega_h.
\end{align}

\subsubsection{Multipoint flux mixed finite element method} 
\label{ssub: MFMFE}

On simplicial, quadrilateral, and hexahedral elements, the MPFA method has been related to the MFMFE method \cite{WheYot,IngWheYot,WheXueYot}, which is a MFE method with a BDM$_1$-type space for the velocity and a vertex quadrature rule for the velocity bilinear form $a(\cdot,\cdot)$. On general polytopal grids, the MPFA method has been formulated as a mimetic finite difference (MFD) method in \cite{LSY}.

We focus first on simplicial grids and the MFMFE formulation of the MPFA method \cite{WheYot}. In \Cref{sec: Analysis quads} we discuss smooth quadrilateral and hexahedral grids and comment on the extension to general quadrilateral and hexahedral, as well as general polytopes. 

For each $i$, let $V_{h, i} \times W_{h, i} \subset V_i \times W_i$ be
the BDM$_1$ pair of spaces on simplices \cite{BDM}, defined as follows. Let $\hE$ be the reference triangle or tetrahedron. For each element $\omega \in \Omega_h$ there exists an affine bijection mapping
$F_\omega: \hE \rightarrow \omega$.
Denote the Jacobian of $F$ by $\DF_\omega$ and let $J_\omega = |\mbox{det}(\DF_\omega)|$.
The BDM$_1$ spaces are defined on the reference element $\hE$ as
\begin{equation*}
\hat V(\hE) = P_1(\hE)^n, \quad \hat W(\hE) = P_0(\hE),
\end{equation*}
where $P_k$ denotes the space of polynomials of degree at most $k$. Note that $\hat{\nabla} \cdot \hat V(\hE) = \hat W(\hE)$
and that for
all $\hv \in \hat V(\hE)$ and for any facet $\he$ of $\hE$,
$\bnu_{\he} \cdot \hv \in P_1(\he)$. There are two degrees of freedom per facet in two dimensions and three in three dimensions, which can be chosen to be the values of $\bnu_{\he} \cdot \hv$ at the vertices of $\he$.

The BDM$_1$ spaces on any element $\omega \in \Omega_h$ are defined via the
transformations
\begin{align*}
 \bv \leftrightarrow \hv: \bv &=
 \frac{1}{J_\omega} \DF_\omega \hv \circ F_\omega^{-1}, &
 w \leftrightarrow \hat w: w &= \hat w \circ F_\omega^{-1},
\end{align*}
where the Piola transformation is used for the velocity space. The BDM$_1$ spaces on $\Omega_{h,i}$ are given by
\begin{equation}\label{eq:BDM}
\begin{aligned}
V_{h,i} &= \{\bv \in V_i \colon &
\bv|_\omega &\leftrightarrow \hv, &
\hv \in \hat V(\hE), 
\forall \omega \in \Omega_{h,i} \},\\[1.0ex]
W_{h,i} &= \{ w \in W_i \colon &
w|_\omega &\leftrightarrow \hat w, &
\hat w \in \hat W(\hE), 
\forall \omega \in \Omega_{h,i} \}.
\end{aligned}
\end{equation}
The BDM$_1$ pair is stable for the subproblem defined on $\Omega_i$, such that \cite{BDM}
\begin{subequations}\label{mixed-spaces}
\begin{align}
  & & 
  \nabla \cdot V_{h, i} &= W_{h, i}, \label{eq: div V = W} \\
  \forall \, w_{h,i} &\in W_{h,i}, 
  \exists \, 0 \ne \bm{v}_{h,i} \in V_{h,i}: &
  (\nabla\cdot\bv_{h,i},w_{h,i})_{\Omega_i}
  &\gtrsim \|\bv_{h,i}\|_{\div, \Omega_i} \|w_{h,i}\|_{\Omega_i}. \label{local-inf-sup}
\end{align}
\end{subequations}
%


The MFMFE method employs a vertex quadrature rule for the velocity bilinear form $a_i(\cdot,\cdot)$. For any element-wise continuous vector 
functions $\bphi$ and $\bpsi$ on $\Omega_{h,i}$, we denote by
\begin{equation*}
(\bphi,\bpsi)_{Q,\Omega_i} = \sum_{\omega \in \Omega_{h,i}}(\bphi,\bpsi)_{Q,\omega }
\end{equation*}
the application of the element-wise vertex quadrature rule for computing
$(\bphi,\bpsi)_{\Omega_i}$.  The integration on any element $\omega$ is performed by
mapping to the reference element $\hE$. Let $\tilde\bphi$ and
$\tilde\bpsi$ be the mapped functions on $\hE$, using the standard
change of variables. Since
$(\bphi,\bpsi)_\omega = (\tilde\bphi,\tilde\bpsi J_\omega)_{\hE}$, we define
\begin{equation*}
(\bphi,\bpsi)_{Q,\omega} = \frac{|\hE|}{s}\sum_{i=1}^s
\tilde\bphi(\rh_i)\cdot\tilde\bpsi(\rh_i) J_\omega(\rh_i) = \frac{|\hE|}{s}\sum_{i=1}^s
\bphi(\br_i)\cdot\bpsi(\br_i)J_\omega(\rh_i),
\end{equation*}
where $s$ is the number of vertices of $\omega$ and $\br_i$ and $\rh_i$,
$i = 1,\ldots,s$, are the vertices of $\omega$ and $\hE$, respectively.
Using this quadrature rule, the velocity bilinear form in the MFMFE is defined as
\begin{equation*}
a^h_i(\bu_i,\bv_i) := (K^{-1}\bu_i,\bv_i)_{Q,\Omega_i}, \quad
a^h(\bu,\bv) := \sum_i a^h_i(\bu_i,\bv_i).
\end{equation*}
The quadrature rule localizes the interaction of the velocity degrees of freedom around mesh vertices. This allows for local elimination of the velocity, resulting in a cell-centered finite volume system for the pressure, which is closely related to the MPFA system \cite{aavatsmark2007convergence} from \Cref{ssub: MPFA}. 
However, the quadrature rule introduces a non-conforming term in the numerical error, defined as
\begin{equation} \label{quad-error sigma}
  \sigma(\bphi,\bpsi): = a(\bphi,\bpsi) - a^h(\bphi,\bpsi).
\end{equation}
A bound on this term is shown in \Cref{sec: Analysis simplices}.

\subsection{Interface discretization and coupling}
\label{ssub: interface discretization}

With the subdomain discretizations defined above, we turn to the coupling at the interfaces.
Let $V_{h, i}^0$ denote the subspace of $V_{h, i}$ with zero normal trace on $\Gamma_i$ and let $V_{h, i}^\Gamma$ denote the trace space of $V_{h, i}$ on $\Gamma_i$:
\begin{align} 
  V_{h, i}^0 &:= \{\bm{v}_{h, i}^0 \in V_{h, i} :
  \ (\bm{\nu}_i\cdot\bm{v}_{h, i}^0)|_{\Gamma_i} = 0 \}, &
\quad V_h^0 &:= \bigoplus_i V_{h, i}^0, \label{eq: def Vh0} \\
V_{h, i}^\Gamma &:=  (\bm{\nu}_i\cdot V_{h, i})|_{\Gamma_i}, &
	V_h^\Gamma
	&:= \bigoplus_i V_{h, i}^\Gamma. \label{eq: def VhGamma}
\end{align}

Let $S_{H}$ be the following null-space:
\begin{align} \label{eq: definition S_H}
	S_{H, i} :=
	\{ w_{h,i} \in W_{h,i} :\
	b_i(\bm{v}_{h, i}^0, w_{h,i}) = 0, \
	\forall \bm{v}_{h, i}^0 \in V_{h, i}^0 \}, \quad S_H := \bigoplus_i S_{H,i},
\end{align}
in which the subscript $H$ refers to the characteristic subdomain size.
We note that in this case of Darcy flow, the local spaces can be characterized as
\begin{align} \label{SH-characterization}
	S_{H, i} &= \begin{cases}
		\mathbb{R}, & i \in I_{int}, \\
		0, & i \notin I_{int}.
	\end{cases}
\end{align}

For the interfaces, we introduce a shape-regular tessellation of
$\Gamma_{ij}$, denoted by $\Gamma_{h,ij}$, with a typical mesh size
$h_\Gamma$. Let $\Gamma_h = \bigcup_{i<j} \Gamma_{h,ij}$.
Let the discrete interface space $\Lambda_{h,ij}\subset
L^2(\Gamma_{ij})$ contain continuous or discontinuous piecewise
polynomials on $\Gamma_{h,ij}$ of degree $k_\Lambda$ and let $\Lambda_h = \bigoplus_{i<j}\Lambda_{h,ij}$.

The method involves incorporating the mortar flux data as Neumann boundary condition for the subdomain problems. Let $\mathcal{Q}_{h}: \Lambda \to V_{h}^\Gamma$ be a chosen projection operator and let $\mathcal{Q}_{h, i}: \Lambda \to V_{h, i}^\Gamma$ be its
restriction to $V_{h, i}^\Gamma$. Following \cite{boon2020flux}, we consider two choices for $\mathcal{Q}_{h, i}$, described in \Cref{ssub: flat projection,ssub: sharp projection}. The extension of the boundary data into the subdomains is described in \Cref{ssub: extension}.

\subsubsection{Projection onto the normal trace space of the velocity}
\label{ssub: flat projection}

The first choice is $\mathcal{Q}_{h, i} = \mathcal{Q}_{h, i}^\flat$, where
the operator $\mathcal{Q}_{h,i}^\flat: \Lambda \to V_{h, i}^\Gamma$ is the $L^2(\Gamma_i)$-orthogonal projection. It is computed 
for each $i$ by solving the problem:
Given $\lambda \in \Lambda$, find $\myQ_{h, i}^\flat\lambda \in
V_{h, i}^\Gamma$ such that
\begin{align} \label{eq: from-projection problem}
	(\lambda_i - \myQ_{h, i}^\flat\lambda, \xi_{h, i})_{\Gamma_i} = 0, \quad
	\forall \, \xi_{h, i} &\in V_{h, i}^\Gamma.
\end{align}

For the unique solvability of the mortar variable we need for the mortar space on a given interface to be controlled by the normal traces of the neighboring velocity spaces. We make the following assumption.

\begin{enumerate}[label=A\arabic*., ref=A\arabic*, start=1]
\item \label{A: mortar condition}
The following mortar condition holds:
\begin{align}\label{mortar-condition}
\forall \mu_h \in \Lambda_h, \quad \| \mu_h \|_{\Gamma_{ij}}
			&\lesssim
			\| \mathcal{Q}_{h, i}^\flat \mu_h \|_{\Gamma_{ij}}
			+ \| \mathcal{Q}_{h, j}^\flat \mu_h \|_{\Gamma_{ij}}, \quad
			\forall \, \Gamma_{ij}.
		\end{align}
        \end{enumerate}
        
\begin{remark}\label{rem:mortar-cond}
Assumption \ref{A: mortar condition} for $\mathcal{Q}_h^\flat$ is the conventional mortar assumption, see
e.g. \cite{ACWY,APWY}, implying that the mortar variable is controlled on
each interface by its $L^2$-projection onto the normal trace space on one of the two neighboring subdomains. It has been
  shown to hold for some very general mesh configurations \cite{APWY,PenYot}.
In particular, it 
is easy to satisfy in practice 
by choosing a sufficiently coarse mortar grid $\Gamma_h$ \cite{APWY}.
\end{remark}

\subsubsection{Projection onto the space of weakly continuous velocities}
\label{ssub: sharp projection}

The second option $\mathcal{Q}_{h, i} = \mathcal{Q}_{h, i}^\sharp$
is the orthogonal projection to the space of weakly continuous velocities. Following \cite{ACWY}, let the
space of weakly continuous fluxes $V_{h, c}$ and the associated trace
space $V_{h, c}^\Gamma$ be given by
\begin{subequations}
	\begin{align}
		V_{h, c}
		&:= \left\{ \bm{v}_h \in \bigoplus_i V_{h, i} :\
		\sum_i (\bm{\nu}_i \cdot \bm{v}_{h, i}, \mu_h)_{\Gamma_i} = 0, \
		\forall \mu_h \in \Lambda_h \right\}, \label{Vhc} \\
		V_{h, c}^\Gamma
		&:= \left\{ \xi_h \in V_h^\Gamma :\
		\sum_i (\xi_{h, i}, \mu_h)_{\Gamma_i} = 0, \
		\forall \mu_h \in \Lambda_h \right\}. \label{Vhc-Gamma}
	\end{align}
\end{subequations}

We construct the projection $\mathcal{Q}_h^\sharp: \Lambda \to V_{h}^\Gamma$
by solving the
following auxiliary problem \cite{ACWY,boon2020flux}: Given $\lambda
\in \Lambda$, find $\myQ_h^\sharp \lambda \in V_h^\Gamma$ and $\chi_h \in
\Lambda_h$ such that
\begin{subequations} \label{eq: to-projection problem}
\begin{align}
\sum_i (\lambda_i - \myQ_{h, i}^\sharp\lambda - \chi_h, \xi_{h, i})_{\Gamma_i} &= 0, &
\forall \xi_h &\in V_h^\Gamma, \label{eq: eq1 of psi sharp}\\
\sum_i (\myQ_{h, i}^\sharp \lambda, \mu_h)_{\Gamma_i} &= 0, &
\forall \mu_h &\in \Lambda_h. \label{eq: weak continuity of psi}
\end{align}
\end{subequations}
It is shown in \cite[Lemma~3.1]{boon2020flux} that, if \ref{A: mortar condition} holds, then problem \eqref{eq: to-projection problem} has a unique solution. Moreover, it is proved in \cite[Lemma~3.2]{boon2020flux} that $\myQ_h^\sharp\lambda$ is the $L^2$-projection of $\lambda$ onto
$V_{h, c}^\Gamma$, satisfying
\begin{align*}
\sum_i (\lambda_i - \myQ_{h, i}^\sharp\lambda, \xi_{h, i})_{\Gamma_i} &= 0, &
\forall \xi_h &\in V_{h, c}^\Gamma.
\end{align*}

For the unique solvability of the mortar variable in the case of
$\myQ_{h, i}^\sharp$, we make an assumption similar to assumption~\ref{A: mortar condition}.

\begin{enumerate}[label=A\arabic*., ref=A\arabic*, start=2]
\item \label{A: mortar condition sharp}
The following mortar condition holds:
\begin{align}\label{mortar-condition-sharp}
\forall \mu_h \in \Lambda_h, \quad \| \mu_h \|_{\Gamma_{ij}}
&\lesssim \| \mathcal{Q}_{h, i}^\sharp \mu_h \|_{\Gamma_{ij}}
+ \| \mathcal{Q}_{h, j}^\sharp \mu_h \|_{\Gamma_{ij}}, \quad \forall \, \Gamma_{ij}.
\end{align}
\end{enumerate}

\subsubsection{Discrete extension operator}
\label{ssub: extension}

Next, we define a discrete extension operator
$\myR_{h, i}: \Lambda \to V_{h, i}$. For given
$\lambda \in \Lambda$, we consider the following problem: Find $(\myR_{h, i} \lambda, p_{h, i}^\lambda, r_i)
\in V_{h, i} \times W_{h, i} \times S_{H, i}$ such that 
\begin{subequations} \label{eq: R_h problem Darcy}
\begin{align}
	a_i^h(\myR_{h, i} \lambda, \bm{v}_{h, i}^0)
	- b_i(\bm{v}_{h, i}^0, p_{h, i}^\lambda)
	&= 0
	, && \forall \bm{v}_{h, i}^0 \in V_{h, i}^0, \label{Rh-eq1}
	\\
	b_i(\myR_{h, i} \lambda, w_{h, i}) - (r_i, w_{h, i})_{\Omega_i}
	&= 0
	, && \forall w_{h, i} \in W_{h, i}, \label{Rh-eq2}
	\\
	(p_{h, i}^\lambda, s_i)_{\Omega_i}
	&= 0
	, && \forall s_i \in S_{H, i}, \label{Rh-eq3}
        \\
	\bm{\nu}_i \cdot \myR_{h, i} \lambda
	&= \myQ_{h, i}\lambda, && \text{ on } \Gamma_i. \label{Rh-bc}
\end{align}
\end{subequations}
We note that \eqref{Rh-bc} is an essential boundary condition and that,
for subdomains adjacent to $\partial \Omega$, the boundary condition
$p_i^\lambda = 0$ on $\partial \Omega_i \setminus \Gamma_i$ is natural and
has been incorporated in \eqref{Rh-eq1}. The use of the Lagrange multiplier
$r_i \in S_{H, i}$ ensures that the subproblem is solvable and that the auxiliary variable $p_{h, i}^\lambda$ is uniquely defined, i.e. orthogonal to $S_{H, i}$.

\begin{remark}
The use of $\myQ_{h, i}\lambda = \myQ_{h, i}^\flat\lambda$ from \eqref{eq: from-projection problem} in
\eqref{Rh-bc} leads to $\myR_{h, i} = \myR_{h, i}^\flat$, while
$\myQ_{h, i}\lambda = \myQ_{h, i}^\sharp\lambda$ from \eqref{eq: to-projection problem} results in $\myR_{h, i} = \myR_{h, i}^\sharp$. We will present the results that concern both variants by omitting the superscript.
\end{remark}

\subsection{Flux-mortar MFMFE method}
\label{sub: FM-MFMFE method}

Combining the subdomain and interface discretizations from \Cref{sub: Subdomain discretization,ssub: interface discretization}, let $\myR_h := \bigoplus_i \myR_{h, i}$
and let the composite spaces $V_h$ and $W_h$ be defined as
\begin{align} \label{eq: definition V_h}
	V_h &:= \bigoplus_i \left(V_{h, i}^0 \oplus \myR_{h, i} \Lambda_h\right)
	= V_h^0 \oplus \myR_h \Lambda_h, &
	W_h &:= \bigoplus_i W_{h, i}.
\end{align}
We are now ready to define the flux-mortar MFMFE method for
problem \eqref{eq: general form}:
Find $(\bm{u}_h^0, \lambda_h, p_h) \in V_h^0 \times \Lambda_h
\times W_h$ such that
\begin{subequations} \label{dd-formulation}
\begin{align}
	\sum_i a^h_{i}(\bm{u}_{h,i}^0 + \myR_{h,i} \lambda_h, \bm{v}_{h, i}^0)
	- b_i(\bm{v}_{h, i}^0, p_{h,i})
	& = 0,
	& \forall \bm{v}_h^0 \in V_h^0, \label{dd-1} 
	\\
	\sum_i a^h_i(\bm{u}_{h,i}^0 + \myR_{h,i} \lambda_h, \myR_{h,i} \mu_h)
	- b_i(\myR_{h,i} \mu_h, p_{h,i})
	& = 0,
	& \forall \mu_h \in \Lambda_h, \label{dd-2}
        \\
	\sum_i b_i(\bm{u}_{h,i}^0 + \myR_{h,i} \lambda_h), w_{h,i})
	& = (f, w_{h})_{\Omega},
	& \forall w_h \in W_h. \label{dd-3}
\end{align}
\end{subequations}
To shorten notation, let $\bm{u}_h := \bm{u}_h^0 + \myR_h \lambda_h$ and
$\bm{v}_h := \bm{v}_h^0 + \myR_h \mu_h$.
Then \eqref{dd-formulation} can be equivalently
written as: Find $\bm{u}_h \in V_h$ and $p_h \in W_h$ such that
\begin{subequations} \label{eq: general form-h}
	\begin{align}
		a^h(\bm{u}_h, \bm{v}_h)
		- b(\bm{v}_h, p_h)
		&= 0,
		& \forall \bm{v}_h &\in V_h, \label{eq: general form-h 1}\\
		b(\bm{u}_h, w_h)
		&= (f, w_h)_\Omega,
		& \forall w_h &\in W_h. \label{eq: general form-h 2}
	\end{align}
\end{subequations}
Note that the flux-mortar mixed finite element method \eqref{eq: general form-h}
is a non-conforming discretization of the weak formulation \eqref{eq: general form}, since in general $V_h \not\subset V$. We further emphasize that the
discrete trial and test functions from $V_h$ are naturally decomposed
into internal and interface degrees of freedom using $\myR_h$.

\section{Well-posedness and error analysis on simplicial grids}
\label{sec: Analysis simplices}

This section concerns the \emph{a priori} analysis of the flux-mortar MFMFE method proposed in \Cref{sec: Numerical method}. We first show that the discrete problem is well-posed in \Cref{sub: well-posedness simplices} and then present the error analysis in \Cref{sub: error analysis simplices}. 

\subsection{Well-posedness}
\label{sub: well-posedness simplices}

We follow the abstract analysis developed in \cite[Section~2.4]{boon2020flux} with modifications to take into account the non-conformity due to the use of the quadrature rule in $a^h(\cdot,\cdot)$. We begin with a variant of \cite[Theorem~2.1]{boon2020flux}.

\begin{theorem} \label{thm:well-posed}
Assume that problem \eqref{eq: R_h problem Darcy} has a unique solution and
  the resulting extension operator $\myR_h: \Lambda \to V_h$ is
  continuous, i.e. 
  \begin{align}\label{R-cont}
  \| \myR_h \lambda \|_V &\lesssim \| \lambda \|_\Lambda, \quad
    \forall \, \lambda \in \Lambda.
  \end{align}
Assume in addition that the following four inequalities hold:
\begin{subequations} \label{ineqs: Brezzi conditions}
\begin{align}
	\forall & \bm{u}_h, \bm{v}_h \in V_h:
	&  a^h(\bm{u}_h, \bm{v}_h) &\lesssim \| \bm{u}_h \|_V \| \bm{v}_h \|_V,
	\label{ineq: a_cont}\\
	\forall & \bm{v}_h \in V_h \text{ and } w_h \in W_h:
	&  b(\bm{v}_h, w_h) &\lesssim \| \bm{v}_h \|_V \| w_h \|_W,
	\label{ineq: b_cont}\\
	\forall & \bm{v}_h \in V_h \text{ with } b(\bm{v}_h, w_h) = 0 \ \forall w_h \in W_h:
	&  a^h(\bm{v}_h, \bm{v}_h) &\gtrsim \| \bm{v}_h \|_V^2,
	\label{ineq: a_coercive}\\
	\forall & w_h \in W_h, \ \exists \, 0 \ne \bm{v}_h \in V_h \text{ such that}:
	&  b(\bm{v}_h, w_h) &\gtrsim \| \bm{v}_h \|_V \| w_h \|_W.
	\label{ineq: b_infsup}
\end{align}
\end{subequations}
Then the discrete problem \eqref{eq: general form-h} admits a unique
solution that satisfies
\begin{align}\label{stab-est}
\| \bm{u}_h \|_V + \| p_h \|_W \lesssim \| f \|_{\Omega}.
\end{align}
\end{theorem}
\begin{proof}
Assumption \eqref{R-cont} ensures that the space $V_h$ is well-defined. The well-posedness of the discrete problem then follows directly from \eqref{ineqs: Brezzi conditions} and classical saddle point theory \cite{boffi2013mixed}.
\end{proof}

The verification of the conditions of Theorem~\ref{thm:well-posed} is presented in the next two lemmas, the first of which is proved in \cite[Lemma~3.3]{boon2020flux}.

\begin{lemma}\label{lem:R}
Problem \eqref{eq: R_h problem Darcy} has a unique solution, which satisfies \eqref{R-cont}. 
\end{lemma}

\begin{lemma}\label{lem:brezzi}
The inequalities \eqref{ineqs: Brezzi conditions} hold.
\end{lemma}
\begin{proof}
  Inequalities \eqref{ineq: a_cont} and \eqref{ineq: a_coercive} state the continuity and coercivity of the discrete bilinear form $a^h(\cdot,\cdot)$. They have been verified in \cite{WheYot}. We note that, due to \eqref{eq: div V = W},
  \eqref{ineq: a_coercive} is equivalent to $a^h(\bm{v}_h, \bm{v}_h) \gtrsim
  \| \bm{v}_h \|_\Omega^2 \ \forall \bv_h \in V_h$. The continuity of $b(\cdot,\cdot)$ \eqref{ineq: b_cont} follows easily from its definition, while the inf-sup condition \eqref{ineq: b_infsup} has been established in \cite[Lemma~3.4]{boon2020flux}.
\end{proof}

A combination of Theorem~\ref{thm:well-posed}, Lemma~\ref{lem:R}, and Lemma~\ref{lem:brezzi} implies the following result.

\begin{corollary} \label{cor: well-posed simplices}
The discrete problem \eqref{eq: general form-h} has a unique
solution that satisfies \eqref{stab-est}.
\end{corollary}

The next result, which is a variant of \cite[Theorem~2.2]{boon2020flux}, concerns the unique solvability of the mortar variable.

\begin{theorem}\label{thm:lambda}
  Let \ref{A: mortar condition} hold in the case of $\myQ_{h, i}^\flat$ and let 
  \ref{A: mortar condition} and \ref{A: mortar condition sharp} hold in the case of $\myQ_{h, i}^\sharp$. Then the mortar solution $\lambda_h$ of \eqref{eq: general form-h} is unique.
\end{theorem}

\begin{proof}
We utilize the discrete trace inequality
\begin{align}\label{discrete-trace-Darcy}
  \forall \, i \in I_\Omega, \quad \|\bnu_i\cdot\bv_{h,i}\|_{\Gamma_i} 
  \lesssim h^{-\frac12} \| \bm{v}_{h,i}\|_{\Omega_i}, \quad
  \forall \bm{v}_{h,i} \in V_{h,i},
\end{align}
which follows from a simple scaling argument. Since $\bm{u}_h = \bm{u}_h^0 + \myR_h \lambda_h$, it holds that $\bnu_i\cdot\bu_{h,i} = \mathcal{Q}_{h, i} \lambda_h$. Then we use
\ref{A: mortar condition} or \ref{A: mortar condition sharp}
and \eqref{discrete-trace-Darcy} to obtain
	\begin{align*}
		\| \lambda_h \|_{\Gamma}
		\lesssim \sum_i \| \mathcal{Q}_{h, i} \lambda_h \|_{\Gamma_i}
                = \sum_i \|\bnu_i\cdot\bu_{h,i}\|_{\Gamma_i}
                \lesssim h^{-\frac12} \| \bu_h \|_\Omega.
	\end{align*}
	The result now follows from \eqref{stab-est}.
\end{proof}

\subsection{Error analysis}
\label{sub: error analysis simplices}

With the well-posedness of the discrete system verified, we continue with the error analysis. We begin by defining suitable interpolation operators in \Cref{ssub: Interpolation operators simplices}, which are used to derive the \emph{a priori} error estimates in \Cref{ssub: error estimate simplices}.

\subsubsection{Interpolation operators}
\label{ssub: Interpolation operators simplices}

We follow the construction in \cite[Section~3.3]{boon2020flux}.
The building blocks in the construction of the interpolant in $V_h$
are the canonical subdomain interpolation operators $\Pi_i^V: V_i \cap
(H^\epsilon(\Omega_i))^n \to V_{h, i}$ with $\epsilon > 0$, with
the properties
\begin{align} 
(\nabla \cdot (\bm{v}_{i} - \Pi_i^V \bm{v}_{i}), w_{h, i})_{\Omega_i} &= 0, &
  \forall w_{h, i} &\in W_{h, i}, \label{eq: commutativity}\\
(\bnu_i\cdot(\bm{v}_{i} - \Pi_i^V \bm{v}_{i}), \xi_{h, i})_{\Gamma_i} &= 0, &
  \forall \xi_{h, i} &\in V_{h, i}^\Gamma. \label{eq: Pi-trace}
\end{align}
The interpolant $\Pi_i^V$ is locally constructed on each element and satisfies the continuity property
\begin{equation}\label{Pi-cont}
  \forall \, \omega \in \Omega_{h,i}, \quad \|\Pi_i^V\bv_i\|_{1,\omega} \lesssim \|\bv_i\|_{1,\omega}, \quad \forall \, \bv_i \in H^1(\omega).
\end{equation}
In addition, let $\Pi_i^W: L^2(\Omega_i) \to W_{h, i}$ and
$\Pi_{ij}^\Lambda: L^2(\Gamma_{ij}) \to \Lambda_{h, ij}$ denote the
$L^2$-projection operators onto $W_{h, i}$ and $\Lambda_{h, ij}$,
respectively. Together with the projection $\mathcal{Q}_{h, i}^\flat$
onto $V_{h, i}^\Gamma$ introduced earlier, we recall the
approximation properties \cite{boffi2013mixed}:
\begin{subequations} \label{eqs: approx props}
	\begin{align}
		\| \bm{v} - \Pi_i^V \bm{v} \|_{\Omega_i}
		&\lesssim h^{r_v} \| \bm{v} \|_{r_v, \Omega_i},
		& 0 &< r_v \le 2, \label{eq: approx prop v_i}\\
		\| \nabla \cdot (\bm{v} - \Pi_i^V \bm{v}) \|_{\Omega_i}
		&\lesssim h^{r_w} \| \nabla \cdot \bm{v} \|_{r_w, \Omega_i},
		& 0 &\le r_w \le 1,
		\label{eq: approx prop div v}\\
		\| w - \Pi_i^W w \|_{\Omega_i}
		&\lesssim h^{r_w} \| w \|_{r_w, \Omega_i},
		& 0 &\le r_w \le 1,
		\label{eq: approx prop w}\\
		\| \mu - \Pi_{ij}^\Lambda \mu \|_{\Gamma_{ij}}
		&\lesssim h_\Gamma^{r_\Lambda} \| \mu \|_{r_\Lambda, \Gamma_{ij}},
		& 0 &\le r_\Lambda \le k_\Lambda + 1,
		\label{eq: approx prop lambda}\\
		\| \mu - \mathcal{Q}_{h, i}^\flat \mu \|_{\Gamma_{ij}}
		& \lesssim
		h^{r_v} \| \mu \|_{r_v, \Gamma_{ij}}, &
		0 & \le r_v \le 2.
		\label{eq: approx prop Q}
	\end{align}
\end{subequations}
Let $\tilde\Pi^V: \bigoplus_i (V_i \cap
(H^\epsilon(\Omega_i))^n) \to \bigoplus_i V_{h,i}$,
$\Pi^W: W \to W_h$ and $\Pi^\Lambda: \Lambda \to \Lambda_h$
be defined as $\tilde\Pi^V := \bigoplus \Pi_i^V$,
$\Pi^W := \bigoplus_i \Pi_i^W$ and $\Pi^\Lambda:= \bigoplus_{i<j}\Pi^\Lambda_{ij}$, respectively.

Next, we introduce the composite interpolant
$\Pi^V: \overline V \to V_h$, where
$\overline V = \{\bm{v} \in V: \bm{v}|_{\Omega_i} \in (H^\epsilon(\Omega_i))^n \text{ and }
  (\bm{\nu} \cdot
\bm{u})|_\Gamma \in \Lambda \}$
with $\epsilon > 0$.
Given $\bm{u} \in \overline V$ with normal trace $\lambda := (\bm{\nu} \cdot
\bm{u})|_\Gamma \in \Lambda$, we define $\Pi^V \bm{u} \in V_h$ as
\begin{subequations}\label{Pi-defn}
\begin{align}
  & \Pi_\flat^V \bm{u} := \myR_h^\flat \Pi^\Lambda\lambda
  + \tilde\Pi^V (\bm{u} - \myR_{h}^\flat \lambda)
= \myR_h^\flat (\Pi^\Lambda\lambda - \lambda) + \tilde\Pi^V \bm{u}, \label{Pi-flat}\\
  & \Pi_\sharp^V \bm{u} := \myR_h^\sharp \Pi^\Lambda\lambda
+ \tilde\Pi^V (\bm{u} - \myR_{h}^\flat \lambda)
= \Pi_\flat^V \bm{u} - \myR_h^\flat \Pi^\Lambda\lambda + \myR_h^\sharp \Pi^\Lambda\lambda. \label{Pi-sharp}
\end{align}
\end{subequations}
We note that \eqref{Rh-bc} for $\myR_{h,i}^\flat \lambda$,
and \eqref{eq: Pi-trace} imply
$ \bnu_i\cdot\Pi_i^V (\bm{u}_i - \myR_{h,i}^\flat \lambda) =
\mathcal{Q}_{h,i}^\flat\lambda - \mathcal{Q}_{h,i}^\flat\lambda = 0,
$
so \eqref{Pi-defn} gives $\Pi_\flat^V\bm{u} \in V_h^\flat$ and
$\Pi_\sharp^V\bm{u} \in V_h^\sharp$. In the following, the use of $\Pi^V$ indicates
that the result is valid for both choices.

\begin{lemma}[{\cite[Lemma~3.7]{boon2020flux}}] \label{lem: B-compatible}
The interpolation operator $\Pi^V$ is $b$-compatible:
\begin{align}\label{b-comp-darcy}
b(\bm{u} - \Pi^V \bm{u}, w_h) &= 0,& \forall \, w_h &\in W_h.
\end{align}
\end{lemma}

The approximation properties of the interpolants $\Pi_\flat^V$ and $\Pi_\sharp^V$ are given below.

\begin{lemma}[{\cite[Lemma~3.8]{boon2020flux}}]
  \label{lem: approximation prop Pi}
Assuming that $\bu$ has sufficient regularity, then
\begin{subequations} \label{eq: approx prop v}
\begin{align}
\| \bm{u} - \Pi_\flat^V \bm{u} \|_V
&\lesssim h^{r_v} \sum_i \| \bm{u} \|_{r_v, \Omega_i}
+ h^{r_w} \sum_i \| \nabla \cdot \bm{u} \|_{r_w, \Omega_i}
+ h_\Gamma^{r_\Lambda} \sum_{i < j} \| \lambda \|_{r_\Lambda, \Gamma_{ij}}, \label{approx-Pi-flat} \\
\| \bm{u} - \Pi_\sharp^V \bm{u} \|_V
&\lesssim
h^{r_v} \sum_i \| \bm{u} \|_{r_v, \Omega_i}
+ h^{r_w} \sum_i \| \nabla \cdot \bm{u} \|_{r_w, \Omega_i}
+ h_\Gamma^{r_\Lambda} \sum_{i < j} \| \lambda \|_{r_\Lambda, \Gamma_{ij}} 
+ h^{\tilde r_v} \sum_{i < j} \| \lambda \|_{\tilde r_v, \Gamma_{ij}}, \label{approx-Pi-sharp}
\end{align}
\end{subequations}
for $0 < r_v \le 2$, $0 \le r_w \le 1$, $0 \le r_\Lambda \le k_\Lambda+1$, 
and $0 \le \tilde r_v \le 2$.
\end{lemma}

\subsubsection{Error estimate}
\label{ssub: error estimate simplices}

We proceed with the error estimate in the case of simplicial grids. The following theorem is a variant of \cite[Theorem~2.3]{boon2020flux}, accounting for the quadrature error.

\begin{theorem}\label{thm:error}
It holds that
	\begin{align} \label{eq: Error estimate general}
		\| \bm{u} - \bm{u}_h \|_V
		+
		\| p - p_h \|_W
		\lesssim
		\| \tilde \Pi^V \bm{u} - \bm{u} \|_V
                + \| \Pi^V \bm{u} - \bm{u} \|_V
		+
		\| \Pi^W p - p \|_W
		+
		\myE_c + \myE_\sigma,
	\end{align}
with $\myE_c$ the consistency error defined as
\begin{align}\label{consist-error}
\myE_c := \sup_{ 0 \ne \bm{v}_h \in V_h}
\frac{a(\bm{u}, \bm{v}_h) - b(\bm{v}_h, p)}{\| \bm{v}_h \|_V},
\end{align}
and $\myE_\sigma$ the quadrature error defined as
\begin{equation}\label{quad-error}
  \myE_\sigma: = \sup_{ 0 \ne \bm{v}_h \in V_h}
\frac{\sigma(\tilde\Pi^V\bm{u}, \bm{v}_h)}{\| \bm{v}_h \|_V}
\end{equation}
with $\sigma(\cdot, \cdot)$ from \eqref{quad-error sigma}.
\end{theorem}

\begin{proof}
The proof is a modification of the proof of \cite[Theorem~2.3]{boon2020flux}, with the main difference being in forming the error equations. In particular,
from \eqref{eq: general form-h} and 
\eqref{eq: general form 2} we obtain the error equations
\begin{subequations} \label{eqs: step1 errors}
\begin{align}
	a^h(\Pi^V\bm{u} - \bm{u}_h, \bm{v}_h)
	- b(\bm{v}_h, \Pi^W p - p_h)
	&=
	a^h(\Pi^V\bm{u}, \bm{v}_h)
	- b(\bm{v}_h, p), \label{err-eq-1}
	\\
	b(\Pi^V \bm{u} - \bm{u}_h, w_h)
	&= 0,
	\label{eq: pi u - u_h div free}
\end{align}
\end{subequations}
for all $(\bm{v}_h, w_h) \in V_h \times W_h$, where in \eqref{eq: pi u - u_h div free} we used the $b$-compatibility of $\Pi^V$ \eqref{b-comp-darcy} 
and the fact that $W_h \subset W$. We manipulate the right hand side of \eqref{err-eq-1} as follows:
\begin{align}\label{error-rhs}
  \begin{split}
  a^h(\Pi^V\bm{u}, \bm{v}_h) - b(\bm{v}_h, p)
  & = a^h(\Pi^V\bm{u} - \tilde\Pi^V\bm{u}, \bm{v}_h) 
  + a^h(\tilde\Pi^V\bm{u}, \bm{v}_h) - a(\tilde\Pi^V\bm{u}, \bm{v}_h) \\
  & \qquad + a(\tilde\Pi^V\bm{u} - \bu, \bm{v}_h)
  + a(\bm{u}, \bm{v}_h) - b(\bm{v}_h, p).
\end{split}
  \end{align}
We recognize that the second and third terms on the right in \eqref{error-rhs} form the numerator of the quadrature error $\myE_\sigma$ from \eqref{quad-error}, and the last two terms form the numerator of the consistency error $\myE_c$ from \eqref{consist-error}. We next set the test functions in \eqref{eqs: step1 errors} as
\begin{align} \label{eq: test functions}
\bm{v}_h := \Pi^V \bm{u} - \bm{u}_h - \delta \bm{v}_h^p, \quad
	w_h := \Pi^W p - p_h.
\end{align}
Here $\bm{v}_h^p \in V_h$ is constructed, using the inf-sup condition on $b$ \eqref{ineq: b_infsup}, to satisfy
\begin{align} \label{eq: properties v_h^p}
	b(\bm{v}_h^p, \Pi^W p - p_h) &= \| \Pi^W p - p_h \|_W^2, &
	\| \bm{v}_h^p \|_V &\lesssim \| \Pi^W p - p_h \|_W,
\end{align}
and $\delta > 0$ is a constant to be chosen later. Now
\eqref{eqs: step1 errors}, in combination with \eqref{error-rhs}, leads to
\begin{align} \label{eq: definition terms}
  \begin{split}
&  a^h(\Pi^V \bm{u} - \bm{u}_h, \Pi^V \bm{u} - \bm{u}_h)
  + \delta \| \Pi^W p - p_h \|_W^2 
  = a^h(\Pi^V\bm{u} - \bm{u}_h, \delta \bm{v}_h^p)\\
  & \quad
  + a^h(\Pi^V\bm{u} - \tilde\Pi^V\bm{u}, \bm{v}_h)
  - \sigma(\tilde\Pi^V\bm{u}, \bm{v}_h)
+ a(\tilde\Pi^V\bm{u} - \bm{u}, \bv_h)
+ \big(	a(\bm{u}, \bm{v}_h) - b(\bm{v}_h, p)
\big).
\end{split}
\end{align}
For the left-hand side of \eqref{eq: definition terms}, \eqref{eq: pi
  u - u_h div free} and the coercivity of $a^h$ \eqref{ineq: a_coercive} imply
\begin{subequations} \label{eqs: component bounds}
\begin{align}
	\| \Pi^V \bm{u} - \bm{u}_h \|_V^2
	\lesssim a^h(\Pi^V \bm{u} - \bm{u}_h, \Pi^V \bm{u} - \bm{u}_h).
\end{align}
For the
first term on the right in \eqref{eq: definition terms}, using
the continuity of $a^h$ \eqref{ineq: a_cont}, Young's inequality
with $\epsilon_1 > 0$, and the bound on $\bm{v}_h^p$ from \eqref{eq:
  properties v_h^p}, we obtain
\begin{align}
	a^h(\Pi^V\bm{u} - \bm{u}_h, \delta \bm{v}_h^p)
	&\lesssim
	\frac{\epsilon_1}{2} \| \Pi^V\bm{u} - \bm{u}_h \|_V^2
	+ \frac{1}{2 \epsilon_1} \delta^2 \| \Pi^W p - p_h \|_W^2.
\end{align}
Similarly, for the second and fourth terms on the right-hand side of \eqref{eq: definition terms} we derive, respectively, with $\epsilon_2 > 0$ and
$\epsilon_3 > 0$,
\begin{align}
a^h(\Pi^V\bm{u} - \tilde\Pi^V\bm{u}, \bm{v}_h)  
\lesssim
\left(\frac{1}{2\epsilon_2} + \frac12 \right)
\| \Pi^V\bm{u} - \tilde\Pi^V\bm{u} \|_V^2
+ \frac{\epsilon_2}{2} \| \Pi^V\bm{u} - \bm{u}_h \|_V^2
+ \frac12 \delta^2 \| \Pi^W p - p_h \|_W^2,
\end{align}
and
\begin{align}
a(\tilde\Pi^V\bm{u} - \bm{u}, \bv_h)
\lesssim
\left(\frac{1}{2\epsilon_3} + \frac12 \right) \|\tilde\Pi^V\bm{u} - \bm{u} \|_V^2
+ \frac{\epsilon_3}{2} \| \Pi^V\bm{u} - \bm{u}_h \|_V^2
+ \frac12 \delta^2 \| \Pi^W p - p_h \|_W^2.
\end{align}
For the third term on the right in \eqref{eq: definition terms} we write, with $\epsilon_4 > 0$,
\begin{align}
  \begin{split}
\sigma(\tilde\Pi^V\bm{u}, \bm{v}_h) \le
	\| \bm{v}_h \|_V \myE_\sigma 
	\lesssim \frac{\epsilon_4}{2} \| \Pi^V \bm{u} - \bm{u}_h \|_V^2
	+ \frac12 \delta^2 \| \Pi^W p - p_h \|_W^2
	+ \left( \frac{1}{2\epsilon_4} + \frac12 \right) \myE_\sigma^2.
  \end{split}
\end{align}
Finally, for the last two terms in \eqref{eq: definition terms} we obtain,
with $\epsilon_5 > 0$,
\begin{align}
	a(\bm{u}, \bm{v}_h) - b(\bm{v}_h, p)
	\le
	\| \bm{v}_h \|_V \myE_c 
	\lesssim \frac{\epsilon_5}{2} \| \Pi^V \bm{u} - \bm{u}_h \|_V^2
	+ \frac12 \delta^2 \| \Pi^W p - p_h \|_W^2
	+ \left( \frac{1}{2\epsilon_5} + \frac12 \right) \myE_c^2.
\end{align}
\end{subequations}
Collecting \eqref{eqs: component bounds} and setting all $\epsilon_i$
sufficiently small, we arrive at
\begin{align*}
	\| \Pi^V \bm{u} - \bm{u}_h \|_V^2
	+
	\delta \| \Pi^W p - p_h \|_W^2 
	\lesssim
        \| \tilde\Pi^V \bm{u} - \bm{u} \|_V^2 +
	\| \Pi^V \bm{u} - \bm{u} \|_V^2
	+
	\delta^2 \| \Pi^W p - p_h \|_W^2
	+
	\myE_c^2 + \myE_\sigma^2 .
\end{align*}
We now set $\delta$ sufficiently small to obtain
\begin{align} \label{eq: pre-intermediate}
	\| \Pi^V \bm{u} - \bm{u}_h \|_V
	+
	\| \Pi^W p - p_h \|_W
	\lesssim
        \| \tilde\Pi^V \bm{u} - \bm{u} \|_V +
	\| \Pi^V \bm{u} - \bm{u} \|_V
	+
	\myE_c + \myE_\sigma.
\end{align}
Combining this with the triangle inequality gives us \eqref{eq: Error estimate general}.
\end{proof}

To complete the error estimate, we need to control $\myE_c$ and $\myE_\sigma$. For the consistency error $\myE_c$ we write, using integration by parts on $\Omega_i$ and the boundary condition $p = 0$ on $\partial \Omega$,
\begin{align} \label{eq: nonconformity term}
  \myE_c
  	= \sup_{\bm{v}_h \in V_h} \| \bm{v}_h \|_V^{-1}
        \Big((K^{-1} \bm{u}, \bm{v}_h)_\Omega
	- \sum_i (p, \nabla \cdot \bm{v}_h)_{\Omega_i}\Big) 
	=
	\sup_{\bm{v}_h \in V_h} \| \bm{v}_h \|_V^{-1}
	\sum_i -(p, \bm{\nu}_i \cdot \bm{v}_{h, i})_{\Gamma_i},
\end{align}
where we used that $K^{-1} \bm{u} = - \nabla p$ from \eqref{model}. This error has been bounded in \cite{boon2020flux} for both variants $\myR_h^\sharp$ and $\myR_h^\flat$. 
In the case of $\myR_h = \myR_h^\sharp$, an additional interpolation operator is utilized. Denoting the discrete subspace
consisting of continuous mortar functions by $\Lambda_{h,
  c} \subset \Lambda_h$, let $\Pi_c^\Lambda: H^1(\Gamma) \to
\Lambda_{h, c}$ be the Scott-Zhang interpolant \cite{Scott-Zhang} into $\Lambda_{h,c}$. This interpolant has the approximation property
\begin{align} \label{eq: approx Scott Zhang}
	\| p - \Pi_c^\Lambda p \|_{s_\Lambda, \Gamma}
	\lesssim h_\Gamma^{r_\Lambda - s_\Lambda} \| p \|_{r_\Lambda, \Gamma}, \quad
        1 \le r_\Lambda \le k_\Lambda + 1, \,\,
        0 \le s_\Lambda \le 1.
\end{align}
We next state the bounds on $\myE_c$ established in \cite{boon2020flux}. 

\begin{lemma}\label{lem:consist-error}
Let \ref{A: mortar condition} hold.
In the case $\myR_h = \myR_h^\flat$, it holds that
\begin{equation}\label{consist-error-flat}
\myE_c \lesssim h^{-\frac12} \sum_i \| p - \mathcal{Q}_{h, i}^\flat p \|_{\Gamma_i}.
\end{equation}
In the case $\myR_h = \myR_h^\sharp$, it holds that
\begin{equation}\label{consist-error-sharp}
  \myE_c \lesssim \| p - \Pi_c^\Lambda p \|_{\frac12, \Gamma}.
\end{equation}
\end{lemma}
\begin{proof}
	The bound in the case $\myR_h = \myR_h^\flat$ is given in \cite[(3.34)]{boon2020flux} and the bound in the case $\myR_h = \myR_h^\sharp$ is given in \cite[(3.33)]{boon2020flux}.
\end{proof}

The quadrature error $\sigma(\cdot,\cdot)$ has been bounded in \cite{WheYot}. In particular, assuming that $K^{-1}|_\omega \in W^{1,\infty}(\omega)$ for all elements $\omega$, it is shown in \cite[Lemma~3.5]{WheYot} for each subdomain $\Omega_i$ that
\begin{equation}\label{sigma-bound}
  \sigma_i(\bq,\bv) \lesssim \sum_{\omega \in \Omega_{h,i}} h \|K^{-1}\|_{1,\infty,\omega}  \|\bq\|_{1,\omega}\|\bv\|_\omega   \quad \forall \, \bq,\bv \in V_{h,i},
\end{equation}
where $\sigma_i := \sigma|_{\Omega_i}$. We then obtain the following bound.

\begin{lemma}\label{lem:Esigma-bound}
Assuming that $K^{-1}|_\omega \in W^{1,\infty}(\omega)$ for all elements $\omega$, it holds that
\begin{equation}\label{Esigma-bound}
  \myE_\sigma \lesssim h \sum_i \|\bu\|_{1,\Omega_i}.
\end{equation}
\end{lemma}
\begin{proof}
The result follows from the definition \eqref{quad-error} of $\myE_\sigma$, bound \eqref{sigma-bound}, and the continuity of $\tilde\Pi$ \eqref{Pi-cont}.
\end{proof}

Combining Theorem~\ref{thm:error}, Lemma~\ref{lem:consist-error}, Lemma~\ref{lem:Esigma-bound}, and the approximation properties
\eqref{eqs: approx props}, \eqref{eq: approx prop v}, and
\eqref{eq: approx Scott Zhang}, we arrive at the following error estimate.

\begin{theorem}\label{thm:error-final}
  Assume that \ref{A: mortar condition} holds and 
  that the solution to \eqref{weak-model} is sufficiently smooth. Then, in the case $\myR_h = \myR_h^\flat$, it holds that
\begin{align*}
\| \bm{u} - \bm{u}_h \|_V  + \| p - p_h \|_W
\lesssim \ h \sum_i \left( \| \bm{u} \|_{1, \Omega_i}
+ \| \nabla \cdot \bm{u} \|_{1, \Omega_i} + \| p \|_{1, \Omega_i}
+ \| p \|_{\frac32, \Gamma_i} \right)
+ h_\Gamma^{k_\Lambda+1}\sum_{i < j} \| \lambda \|_{k_\Lambda + 1, \Gamma_{ij}}. 
\end{align*}
In the case $\myR_h = \myR_h^\sharp$, it holds that
\begin{align*}
\| \bm{u} - \bm{u}_h \|_V + \| p - p_h \|_W
\lesssim& \ h \left( \sum_i \left(\| \bm{u} \|_{1, \Omega_i}
+ \| \nabla \cdot \bm{u} \|_{1, \Omega_i}
+ \| p \|_{1, \Omega_i} \right)   
  + \sum_{i < j} \| \lambda \|_{1, \Gamma_{ij}} \right) \\
& \  + h_\Gamma^{k_\Lambda + 1} \sum_{i < j} \| \lambda \|_{k_\Lambda + 1, \Gamma_{ij}} 
+ h_\Gamma^{k_\Lambda + \frac12} \| p \|_{k_\Lambda + 1, \Gamma}.
\end{align*}
\end{theorem}

\section{Extension to quadrilateral and hexahedral grids}
\label{sec: Analysis quads}

In this section we discuss the flux-mortar MFMFE method for quadrilateral and hexahedral grids. We follow the same steps as in the simplicial case and thus first define the discrete function spaces in \Cref{sub:numerical method quads}, in analogy with \Cref{sec: Numerical method}. Second, we show that the discrete system is well-posed in \Cref{sub: well-posedness quads}, similar to \Cref{sub: well-posedness simplices}. Finally, we derive error estimates in \Cref{sub: Error analysis quads}.

\subsection{Numerical method}
\label{sub:numerical method quads}

For each $i$, let $V_{h, i} \times
W_{h, i} \subset V_i \times W_i$ be the BDM$_1$ spaces on  quadrilaterals
\cite{BDM} or the enhanced BDDF$_1$ spaces on hexahedra \cite{IngWheYot}. On the reference square or cube, the spaces are defined as
\begin{equation*}
\hat V(\hE) := P_1(\hE)^n + S_n, \quad \hat W(\hE) := P_0(\hE),
\end{equation*}
where the space $S_2$ is the span of two divergence-free curl vectors and the space $S_3$ is the span of twelve divergence-free curl vectors. Thus, $\text{dim} (\hat V(\hE)) = 8$ in $\mathbb{R}^2$ and $\text{dim} (\hat V(\hE)) = 24$ in $\mathbb{R}^3$. The vector functions in $S_n$ are specially chosen, so that $\bnu_{\he} \cdot \hv \in P_1(\he)$ on any facet $\he$ of $\hE$ in $\mathbb{R}^2$ and $\bnu_{\he} \cdot \hv \in Q_1(\he)$ on any facet $\he$ of $\hE$ in $\mathbb{R}^3$, where $Q_1$ is the space of bilinear functions. As a result,
there are two degrees of freedom per facet in two dimensions and four in three dimensions, which can be chosen to be the values of $\bnu_{\he} \cdot \hv$ at the vertices of $\he$. We note the original BDDF$_1$ space \cite{BDDF} is of dimension 18 and contains only three degrees of freedom per facet. It was enhanced in \cite{IngWheYot} for the purpose of the MFMFE method.  
The spaces $V_{h,i}$ and $W_{h,i}$ are defined as in \eqref{eq:BDM}.

Here we focus on $h^2$-parallelograms or $h^2$-parallelepipeds and a symmetric quadrature rule, which have been studied in \cite{WheYot,IngWheYot}. In two dimensions, an element $\omega$ is an $h^2$-parallelogram if
\begin{equation*}
|\vec a - \vec b|_{\mathbb{R}^2} \lesssim h^2,
\end{equation*}
where $\vec a$ and $\vec b$ are any two opposite facets of $\omega$ and $|\cdot|_{\mathbb{R}^2}$ is the Euclidean vector norm. In three dimensions, an element $\omega$ is an $h^2$-parallelepiped if all of its facets are $h^2$-parallelograms.

\begin{remark}\label{rem:polytopes}
  The developments in this section can be extended to general quadrilaterals and hexahedra using a non-symmetric quadrature rule and employing the techniques developed in \cite{klausen2006robust,WheXueYot}, as well as to general polytopes using the MFD formulation of the MPFA method and the theory developed in \cite{LSY}. 
\end{remark}

The main difference from the case of simplicial elements is that the space $S_n$ contains quadratic functions. As a result, the quadrature error bound \eqref{sigma-bound} no longer holds. In order to obtain a similar bound, one needs to restrict the second argument $\bv$ in $\sigma_i(\bq,\bv)$ to be at most piecewise linear. To this end, following \cite{WheYot,IngWheYot,wheeler2012multiscale}, we consider the lowest order Raviart-Thomas ($\RT_0$) spaces, which are defined on the reference square or cube as
\begin{equation*}
\hat V^\RT(\hE) := P_0(\hE)^n + (\alpha_1 x_1, \ldots, \alpha_n x_n)^T, \quad \hat W^\RT(\hE) = W(\hE) = P_0(\hE),
\end{equation*}
where $\alpha_i$ are real numbers. The spaces $V_{h,i}^\RT$ and $W_{h,i}^\RT$ are defined as in \eqref{eq:BDM}.

Let $V_h^{\Gamma,\RT}$,
$V_{h, c}^\RT$, and $V_{h, c}^{\Gamma,\RT}$ be the $\RT_0$ counterparts of the spaces $V_h^{\Gamma}$, $V_{h, c}$, and $V_{h, c}^\Gamma$ defined in \eqref{eq: def VhGamma}, \eqref{Vhc}, and \eqref{Vhc-Gamma}, respectively. We modify the projection operator $\myQ_{h,i}$ as follows. In the first option we set 
$\mathcal{Q}_{h, i} = \mathcal{Q}_{h, i}^{\flat,\RT}$, where
$\mathcal{Q}_{h,i}^{\flat,\RT}: \Lambda \to V_{h, i}^{\Gamma,\RT}$ is the $L^2(\Gamma_i)$-orthogonal projection satisfying
\begin{align*} 
(\lambda_i - \myQ_{h, i}^{\flat,\RT}\lambda, \xi_{h, i})_{\Gamma_i} = 0, \quad
\forall \, \xi_{h, i} &\in V_{h, i}^{\Gamma,\RT}.
\end{align*}
The mortar condition \ref{A: mortar condition} is replaced by

\begin{enumerate}[label=A\arabic*$^\RT$., ref=A\arabic*$^\RT$, start=1]
\item \label{A: mortar condition RT}
The following mortar condition holds:
\begin{align}\label{mortar-condition-RT}
\forall \mu_h \in \Lambda_h, \quad \| \mu_h \|_{\Gamma_{ij}}
			&\lesssim
			\| \mathcal{Q}_{h, i}^{\flat,\RT} \mu_h \|_{\Gamma_{ij}}
			+ \| \mathcal{Q}_{h, j}^{\flat,\RT} \mu_h \|_{\Gamma_{ij}}, \quad \forall \, \Gamma_{ij}.
		\end{align}
        \end{enumerate}

In the second option we set $\mathcal{Q}_{h, i} = \mathcal{Q}_{h, i}^{\sharp,\RT}$, where $\mathcal{Q}_{h}^{\sharp,\RT}: \Lambda \to V_{h, c}^{\Gamma,\RT}$ is the $L^2$-projection of $\lambda$ onto $V_{h, c}^{\Gamma,\RT}$, satisfying
\begin{align*}
\sum_i (\lambda_i - \myQ_{h, i}^{\sharp,\RT}\lambda, \xi_{h, i})_{\Gamma_i} = 0, \quad
\forall \xi_h \in V_{h, c}^{\Gamma,\RT}.
\end{align*}
The mortar condition \ref{A: mortar condition sharp} is replaced by

\begin{enumerate}[label=A\arabic*$^\RT$., ref=A\arabic*$^\RT$, start=2]
\item \label{A: mortar condition sharp RT}
The following mortar condition holds:
\begin{align}\label{mortar-condition-sharp-RT}
\forall \mu_h \in \Lambda_h, \quad \| \mu_h \|_{\Gamma_{ij}}
&\lesssim \| \mathcal{Q}_{h, i}^{\sharp,\RT} \mu_h \|_{\Gamma_{ij}}
+ \| \mathcal{Q}_{h, j}^{\sharp,\RT} \mu_h \|_{\Gamma_{ij}}, \quad \forall \, \Gamma_{ij}.
\end{align}
\end{enumerate}

The above modifications affect the discrete extension operator $\myR_{h,i}$ defined in \eqref{eq: R_h problem Darcy}, cf. \eqref{Rh-bc}, and consequently the definition of the discrete velocity space $V_h$, cf. \eqref{eq: definition V_h}. With these modifications, the 
flux-mortar MFMFE method on quadrilateral and hexahedral grids is defined as in \eqref{eq: general form-h}.

\subsection{Well-posedness}
\label{sub: well-posedness quads}

Similar to \Cref{cor: well-posed simplices}, the well-posedness of the discrete problem follows directly.

\begin{corollary}
The discrete problem \eqref{eq: general form-h} in the case of $h^2$-parallelograms or $h^2$-parallelepipeds has a unique
solution that satisfies \eqref{stab-est}.
\end{corollary}
\begin{proof}
The statement follows from Theorem~\ref{thm:well-posed}. In particular, Lemma~\ref{lem:R} still holds in this case. Inequalities \eqref{ineq: a_cont} and \eqref{ineq: a_coercive} for these types of elements have been verified in
\cite{WheYot,IngWheYot}. The continuity of $b(\cdot,\cdot)$ \eqref{ineq: b_cont} and the inf-sup condition \eqref{ineq: b_infsup} follow as in the case of simplices. 
\end{proof}

The proof of the following theorem is the same as the proof of Theorem~\ref{thm:lambda}.

\begin{theorem}\label{thm:lambda-RT}
  Let \ref{A: mortar condition RT} hold in the case of $\myQ_{h, i}^{\flat,\RT}$ and let 
  \ref{A: mortar condition RT} and \ref{A: mortar condition sharp RT} hold in the case of $\myQ_{h, i}^{\sharp,\RT}$. Then the mortar solution $\lambda_h$ of \eqref{eq: general form-h} is unique.
\end{theorem}

\subsection{Error analysis}
\label{sub: Error analysis quads}

In this section, we follow the same steps as in \Cref{sub: error analysis simplices} to present the \emph{a priori} error analysis. Thus, we first define suitable interpolation operators in \Cref{ssub: Interpolation operators quads} and present the error estimates in \Cref{ssub: Error estimate quads}.

\subsubsection{Interpolation operators}
\label{ssub: Interpolation operators quads}

The definition of the composite interpolant
$\Pi^V: \overline V \to V_h$ is modified from \eqref{Pi-defn} to
\begin{subequations}\label{Pi-defn-RT}
\begin{align}
  & \Pi_\flat^V \bm{u} := \myR_h^{\flat,\RT} \Pi^\Lambda\lambda
  + \tilde\Pi^V (\bm{u} - \myR_{h}^{\flat} \lambda)
  = \myR_h^{\flat,\RT} (\Pi^\Lambda\lambda - \lambda)
+ \myR_h^{\flat,\RT}\lambda - \myR_h^{\flat}\lambda
  + \tilde\Pi^V \bm{u}, \label{Pi-flat-RT}\\
  & \Pi_\sharp^V \bm{u} := \myR_h^{\sharp,\RT} \Pi^\Lambda\lambda
+ \tilde\Pi^V (\bm{u} - \myR_{h}^{\flat} \lambda)
= \Pi_\flat^V \bm{u} - \myR_h^{\flat,\RT} \Pi^\Lambda\lambda + \myR_h^{\sharp,\RT} \Pi^\Lambda\lambda. \label{Pi-sharp-RT}
\end{align}
\end{subequations}
As in \eqref{Pi-defn}, $\Pi_i^V(\bm{u}_i - \myR_{h,i}^\flat \lambda) \in V_i^0$, so \eqref{Pi-defn-RT} gives $\Pi_\flat^V\bm{u} \in V_h^\flat$ and
$\Pi_\sharp^V\bm{u} \in V_h^\sharp$. We next note that Lemma~\ref{lem: B-compatible} still holds. We also have the following approximation properties.

\begin{lemma}
  \label{lem: approximation prop Pi RT}
Assuming that $\bu$ has sufficient regularity, then
\begin{align}
\| \bm{u} - \Pi^V \bm{u} \|_V
&\lesssim h^{r_v} \sum_i \| \bm{u} \|_{r_v, \Omega_i}
+ h^{r_w} \sum_i \| \nabla \cdot \bm{u} \|_{r_w, \Omega_i}
+ h_\Gamma^{r_\Lambda} \sum_{i < j} \| \lambda \|_{r_\Lambda, \Gamma_{ij}}
+ h^{\tilde r_v} \sum_{i < j} \| \lambda \|_{\tilde r_v, \Gamma_{ij}},
\label{approx-Pi-RT} 
\end{align}
for $0 < r_v \le 2$, $0 \le r_w \le 1$, $0 \le r_\Lambda \le k_\Lambda+1$, 
and $0 \le \tilde r_v \le 1$.
\end{lemma}
\begin{proof}
Compared to the second expression in \eqref{Pi-flat}, there is an additional term $\myR_h^{\flat,\RT}\lambda - \myR_h^{\flat}\lambda$ in \eqref{Pi-flat-RT}. Since this is the discrete extension \eqref{eq: R_h problem Darcy} with boundary data $\myQ_{h, i}^{\flat,\RT}\lambda - \myQ_{h, i}^{\flat}\lambda$ in \eqref{Rh-bc}, the continuity of $\myR_{h,i}$ \eqref{R-cont}, cf. Lemma~\ref{lem:R}, implies that $\|\myR_{h,i}^{\flat,\RT}\lambda - \myR_{h,i}^{\flat}\lambda\|_{\Omega_i} \lesssim \|\myQ_{h, i}^{\flat,\RT}\lambda - \myQ_{h, i}^{\flat}\lambda\|_{\Omega_i}$, which leads to the last term in \eqref{approx-Pi-RT}. The rest of the terms are bounded as in the proof of Lemma~\ref{lem: approximation prop Pi}, cf. \cite[Lemma~3.8]{boon2020flux}. We note that both variants of $\Pi^V \bm{u}$ now have the same approximation properties. We also remark that, compared to  
the range of the index $\tilde r_v$ in \eqref{approx-Pi-sharp}, we now have 
$0 \le \tilde r_v \le 1$. The reason for the change is that the space $V_{h,i}^{\Gamma,\RT}$ contains piecewise constant functions, rather than piecewise linears, as is the case for $V_{h,i}^\Gamma$. 
\end{proof}

Let $\Pi_i^\RT: V_i \cap (H^\epsilon(\Omega_i))^n \to V_{h, i}^\RT$ be the canonical mixed interpolant of $\RT_0$, with properties
\begin{subequations}
\begin{align} 
(\nabla \cdot (\bm{v}_{i} - \Pi_i^\RT \bm{v}_{i}), w_{h, i})_{\Omega_i} &= 0, &
\forall \, w_{h, i} &\in W_{h, i}, \label{eq: commutativity RT}\\
(\bnu_i\cdot(\bm{v}_{i} - \Pi_i^\RT \bm{v}_{i}), \xi_{h, i})_{\Gamma_i} &= 0, &
\forall \, \xi_{h, i} &\in V_{h, i}^{\Gamma,\RT}, \label{eq: Pi-trace-RT}\\
\|\Pi_i^\RT\bv_i\|_{1,\omega} &\lesssim \|\bv_i\|_{1,\omega}, & 
\forall \, \bv_i &\in H^1(\omega). \label{Pi-cont-RT}
\end{align}
\end{subequations}
Moreover, $\Pi_i^\RT$ preserves the divergence and is continuous:
\begin{subequations}
\begin{align}
	\nabla\cdot\Pi_i^{\RT}\bv_{h,i} &= \nabla\cdot\bv_{h,i}, &
	\quad \forall \, \bv_{h,i} &\in V_{h,i}, \label{div-Pi-RT} \\
	\|\Pi_i^\RT \bv_{h,i}\|_{\omega} &\lesssim \|\bv_{h,i}\|_{\omega}, &
	\quad \forall \, \bv_{h,i} &\in V_{h,i}. \label{Pi-RT-cont} 
\end{align}
\end{subequations}

\subsubsection{Error estimate}
\label{ssub: Error estimate quads}

Using the interpolation operators from \Cref{ssub: Interpolation operators quads}, we arrive at the following error estimate.

\begin{theorem}\label{thm:error-RT}
In the case of $h^2$-parallelograms or $h^2$-parallelepipeds, it holds that
	\begin{align} \label{eq: Error estimate general RT}
		\| \bm{u} - \bm{u}_h \|_V
		+
		\| p - p_h \|_W
		\lesssim
		\| \tilde \Pi^V \bm{u} - \bm{u} \|_V
                + \| \Pi^V \bm{u} - \bm{u} \|_V
		+
		\| \Pi^W p - p \|_W
		+
		\myE_c^\RT + \myE_\sigma^\RT + \myE_h^\RT,
	\end{align}
where 
\begin{subequations}
\begin{align}
& \myE_c^\RT := \sup_{ 0 \ne \bm{v}_h \in V_h}
\frac{a(\bm{u}, \Pi^\RT \bm{v}_h) - b(\Pi^\RT\bm{v}_h, p)}{\| \bm{v}_h \|_V},
\label{consist-error-RT}\\
&  \myE_\sigma^\RT : = \sup_{ 0 \ne \bm{v}_h \in V_h}
  \frac{\sigma(\tilde\Pi^V\bm{u}, \Pi^\RT\bm{v}_h)}{\| \bm{v}_h \|_V},
  \label{quad-error-RT}\\
& \myE_h^\RT := \sup_{ 0 \ne \bm{v}_h \in V_h}
  \frac{a^h(\tilde\Pi^V\bm{u}, \bm{v}_h - \Pi^\RT\bm{v}_h)}{\| \bm{v}_h \|_V}.
  \label{ah-error}
\end{align}
\end{subequations}
\end{theorem}

\begin{proof}
  We start with the error equations \eqref{eqs: step1 errors} from the proof of Theorem~\ref{thm:error}, but now we manipulate the right hand side of \eqref{err-eq-1} in a different way:
\begin{align*}
  \begin{split}
  & a^h(\Pi^V\bm{u}, \bm{v}_h) - b(\bm{v}_h, p)
  = a^h(\Pi^V\bm{u} - \tilde\Pi^V\bm{u}, \bm{v}_h) 
  + a^h(\tilde\Pi^V\bm{u}, \bm{v}_h - \Pi^\RT\bm{v}_h) \\
& \qquad   
+ a^h(\tilde\Pi^V\bm{u}, \Pi^\RT\bm{v}_h)
  - a(\tilde\Pi^V\bm{u}, \Pi^\RT\bm{v}_h) 
  + a(\tilde\Pi^V\bm{u} - \bu, \Pi^\RT\bm{v}_h)
  + a(\bm{u}, \Pi^\RT\bm{v}_h) - b(\Pi^\RT\bm{v}_h, p),
\end{split}
  \end{align*}
where we used \eqref{div-Pi-RT} for the last term.
We recognize that the second term on the right forms the numerator of the error 
$\myE_h^\RT$ from \eqref{ah-error},
the third and fourth terms form the numerator of the quadrature error $\myE_\sigma^\RT$ from \eqref{quad-error-RT}, and the last two terms form the numerator of the consistency error $\myE_c^\RT$ from \eqref{consist-error-RT}. The rest of the proof follows the proof of Theorem~\ref{thm:error}.
\end{proof}

We proceed with the bounds of the three error terms on the right in \eqref{eq: Error estimate general RT}. For $\myE_c^\RT$, recalling \eqref{eq: nonconformity term}, we write
\begin{align} \label{eq: nonconformity term RT}
  \myE_c^\RT = \sup_{\bm{v}_h \in V_h} \| \bm{v}_h \|_V^{-1}
  \sum_i -(p, \bm{\nu}_i \cdot \Pi_i^\RT\bm{v}_{h, i})_{\Gamma_i}
  = \sup_{\bm{v}_h \in V_h} \| \bm{v}_h \|_V^{-1}
  \sum_i -(p, \bm{\nu}_i \cdot \bm{v}_{h, i})_{\Gamma_i},
\end{align}
where we used that $\Pi_i^\RT$ satisfies \eqref{eq: Pi-trace-RT} and that by construction $\bv_{h,i} \in V_{h,i}^{\Gamma,\RT}$. Therefore, the arguments leading to (3.33) and (3.34) in \cite{boon2020flux} also hold in this case, leading the following result, similar to Lemma~\ref{lem:consist-error}.

\begin{lemma}\label{lem:consist-error-RT}
Let \ref{A: mortar condition RT} hold.
In the case $\myR_h = \myR_h^{\flat,\RT}$, it holds that
\begin{equation}\label{consist-error-flat-RT}
\myE_c^\RT \lesssim h^{-\frac12} \sum_i \| p - \mathcal{Q}_{h, i}^{\flat,\RT} p \|_{\Gamma_i}.
\end{equation}
In the case $\myR_h = \myR_h^{\sharp,\RT}$, it holds that
\begin{equation}\label{consist-error-sharp-RT}
  \myE_c^\RT \lesssim \| p - \Pi_c^\Lambda p \|_{\frac12, \Gamma}.
\end{equation}
\end{lemma}

For the quadrature error, we refer to \cite[Lemma~3.5]{WheYot} for $h^2$-parallelograms and \cite[Lemma~3.8]{IngWheYot} for $h^2$-parallelepipeds, where it is shown that
\begin{equation}\label{sigma-bound-RT}
  \sigma_i(\bq,\bv) \lesssim \sum_{\omega \in \Omega_{h,i}} h \|K^{-1}\|_{1,\infty,\omega}  \|\bq\|_{1,\omega}\|\bv\|_\omega   \quad \forall \, \bq \in V_{h,i}, \bv \in V_{h,i}^\RT.
\end{equation}

\begin{lemma}\label{lem:Esigma-bound-RT}
Assuming that $K^{-1}|_\omega \in W^{1,\infty}(\omega)$ for all elements $\omega$, it holds that
\begin{equation}\label{Esigma-bound-RT}
  \myE_\sigma^\RT \lesssim h \sum_i \|\bu\|_{1,\Omega_i}.
\end{equation}
\end{lemma}
\begin{proof}
The result follows from the definition \eqref{quad-error-RT} of $\myE_\sigma^\RT$, bound \eqref{sigma-bound-RT}, and the continuity properties  \eqref{Pi-cont} and \eqref{Pi-RT-cont}.
\end{proof}

Finally, $\myE_h^\RT$ has been bounded in \cite[Lemma~3.3]{WheYot} for $h^2$-parallelograms and \cite[Lemma~3.7]{IngWheYot} for $h^2$-parallelepipeds.

\begin{lemma}\label{lem:Eh-bound}
Assuming that $K^{-1}|_\omega \in W^{1,\infty}(\omega)$ for all elements $\omega$, it holds that
\begin{equation}\label{Eh-bound}
  \myE_h^\RT \lesssim h \sum_i \|\bu\|_{1,\Omega_i}.
\end{equation}
\end{lemma}

Combining Theorem~\ref{thm:error-RT}, Lemma~\ref{lem:consist-error-RT}, Lemma~\ref{lem:Esigma-bound-RT}, Lemma~\ref{lem:Eh-bound},
and the approximation properties
\eqref{eqs: approx props}, \eqref{eq: approx prop v}, and
\eqref{eq: approx Scott Zhang}, we arrive at the main error estimate.

\begin{theorem}\label{thm:error-final-RT}
  Assume that \ref{A: mortar condition RT} holds,
  that the solution to \eqref{weak-model} is sufficiently smooth, and that the grids consist of $h^2$-parallelograms or $h^2$-parallelepipeds. Then, in the case $\myR_h = \myR_h^{\flat,\RT}$, it holds that
  \begin{align*}
\| \bm{u} - \bm{u}_h \|_V + \| p - p_h \|_W
  \lesssim &\ h \left( \sum_i \left(\| \bm{u} \|_{1, \Omega_i}
+ \| \nabla \cdot \bm{u} \|_{1, \Omega_i}
+ \| p \|_{1, \Omega_i} \right)   
  + \sum_{i < j} \| \lambda \|_{1, \Gamma_{ij}} \right) \\
  &+ h_\Gamma^{k_\Lambda + 1} \sum_{i < j} \| \lambda \|_{k_\Lambda + 1, \Gamma_{ij}}
+ \sum_i h^{\frac12}\| p \|_{1, \Gamma_i}.
  \end{align*}
In the case $\myR_h = \myR_h^{\sharp,\RT}$, it holds that
\begin{align*}
\| \bm{u} - \bm{u}_h \|_V + \| p - p_h \|_W
  \lesssim &\ h \left( \sum_i \left(\| \bm{u} \|_{1, \Omega_i}
+ \| \nabla \cdot \bm{u} \|_{1, \Omega_i}
+ \| p \|_{1, \Omega_i} \right)   
  + \sum_{i < j} \| \lambda \|_{1, \Gamma_{ij}} \right) \\
&+ h_\Gamma^{k_\Lambda + 1} \sum_{i < j} \| \lambda \|_{k_\Lambda + 1, \Gamma_{ij}} 
+ h_\Gamma^{k_\Lambda + \frac12} \| p \|_{k_\Lambda + 1, \Gamma}.
\end{align*}
\end{theorem}

\section{Non-overlapping domain decomposition}
\label{sec: DD}

In this section we present a non-overlapping domain decomposition algorithm for the solution of the algebraic system of the flux-mortar MFMFE method. It is based on reduction to an interface problem for the flux-mortar variable. We further develop a preconditioner for the resulting interface problem. The general domain decomposition methodology is based on techniques developed in \cite{GW}. We refer to \cite[Section~2.5]{boon2020flux} for a detailed presentation and analysis of the method. 

\subsection{Reduction to an interface problem}

We assume that \ref{A: mortar condition} hold in the case of $\myQ_{h, i}^\flat$ and that
  \ref{A: mortar condition} and \ref{A: mortar condition sharp} hold in the case of $\myQ_{h, i}^\sharp$, in which case the mortar solution $\lambda_h$ of \eqref{eq: general form-h} is unique, see Theorem~\ref{thm:lambda}.

In the solution process we utilize the extension
$\tilde\myR_h\mu \in \bigoplus_i  V_{h,i}$ with $\Tr \tilde\myR_{h,i}\mu
= \mathcal{Q}_{h,i}\mu$ on $\Gamma_i$, such that all of its degrees of
freedom not associated with $\Gamma_i$ are equal to zero. We also utilize the orthogonal decomposition
$\Lambda_h = \Lambda_h^0 \oplus \overline\Lambda_h$, where
\begin{align} \label{eq: def Lambda^0}
\Lambda_h^0 := \{ \mu \in \Lambda_h : \ b(\tilde\myR_h \mu, s) = 0, \ \forall s \in S_H \}.
\end{align}
Let $B:
\Lambda_h \to S_H'$ be defined as: $\forall \, \mu \in \Lambda_h, \langle B
\mu, s\rangle := b(\tilde\myR_h\mu,s) \ \forall \, s \in S_H$.

The solution method has the following five steps. 

\begin{enumerate}
\item Find $\overline\lambda_f \in \overline\Lambda_h$ such that
\begin{align}
  b(\tilde\myR_h \overline\lambda_f, s) = \langle f, s \rangle,
  \quad \forall s &\in S_H. \label{preproblem}
\end{align}
This is a global coarse problem $B \overline\lambda_f = \overline f$ in $S_H'$, which captures the influence of $f$ in \eqref{dd-3} with respect to the space $S_H$, c.f. \eqref{SH-characterization}. 

\item Use $\overline\lambda_f$ to solve independent, local subproblems to
satisfy \eqref{dd-1} and capture the remaining influence of $f$ in \eqref{dd-3}: Find
$(\bm{u}_f^0, p_f^0, r_f) \in V_h^0 \times W_h \times S_H$ such that
\begin{subequations}\label{pre-2}
\begin{align}
	a^h(\bm{u}_f^0, \bm{v}^0) - b(\bm{v}^0, p_f^0)
	&=
	- a^h(\tilde\myR_h \overline\lambda_f, \bm{v}^0)
	+ \langle g, \bm{v}^0 \rangle
	,
	& \forall \bm{v}^0 &\in V_h^0, \\
	b(\bm{u}_f^0, w) - (r_f, w)_{\Omega}
	&= - b(\tilde\myR_h \overline\lambda_f, w) + \langle f, w \rangle,
	& \forall w &\in W_h,
	\\
	(p_f^0, s)_{\Omega}
	&= 0,
	& \forall s &\in S_H.
\end{align}
\end{subequations}
We note that setting $w =
r_f \in S_H$ and using \eqref{preproblem} with \eqref{eq: definition S_H} implies that $r_f = 0$. Therefore, the variable $\bm{u}_f := \bm{u}_f^0 + \tilde\myR_h \overline\lambda_f$ satisfies \eqref{dd-3}. In addition, $\bm{u}_f$ and $p_f^0$ satisfy \eqref{dd-1}. 

\item Satisfy the continuity equation \eqref{dd-2} in $\Lambda_h^0$ by solving the interface problem:
Find $\lambda^0 \in \Lambda_h^0$ such that
\begin{align}\label{eqs: reduced problem}
a^h(\myR_h \lambda^0, \tilde\myR_h \mu^0) - b(\tilde\myR_h \mu^0, p^{\lambda^0}) &=
- a^h(\bm{u}_f, \tilde\myR_h \mu^0) + b(\tilde\myR_h \mu^0, p_f^0), &
\forall \mu^0 \in \Lambda_h^0,
\end{align}
where $(\myR_h \lambda^0, p^{\lambda^0})$ solve the discrete extension problem \eqref{eq: R_h problem Darcy}. 

\item Guarantee that \eqref{dd-2} holds in $\overline\Lambda_h$ and obtain the correct variable $p$: Find $\overline{p}_\lambda \in S_H$ such that
\begin{align} \label{eq: postproblem}
b(\tilde\myR_h \overline\mu, \overline{p}_\lambda)
&= a^h(\bm{u}_f + \myR_h \lambda^0, \tilde\myR_h \overline\mu)
- b(\tilde\myR_h \overline\mu, p^{\lambda^0} + p_f^0), &
\forall \overline\mu \in \overline\Lambda_h,
\end{align}
which is a coarse grid problem $B^T \overline{p}_\lambda = \overline{g}$ in
$\overline\Lambda_h'$.

\item Construct:
\begin{align}\label{eqs: construct solution}
  \bm{u} := \bm{u}_f + \myR_h \lambda^0 = 
  \bm{u}_f^0 + \myR_h \lambda^0 + \tilde\myR_h\overline\lambda_f, \quad
  p := p_f^0 + p^{\lambda^0} + \overline{p}_\lambda, \quad \lambda = \lambda^0 + \overline\lambda_f.
\end{align}
It is easy to check that $(\bm{u}, p) \in V_h \times W_h$ solve \eqref{eq: general form-h}.

\end{enumerate}

The interface problem \eqref{eqs: reduced problem} is of the type $a_\Gamma(\lambda^0,\mu^0) = \langle r, \mu^0 \rangle$, where $a_\Gamma(\lambda^0,\mu^0) := a^h(\myR_h \lambda^0, \tilde\myR_h \mu^0) - b(\tilde\myR_h \mu^0, p^{\lambda^0})$. It is shown in \cite[Lemma~2.5]{boon2020flux} that $a_\Gamma(\lambda^0,\mu^0) = a^h(\myR_h \lambda^0, \tilde\myR_h \mu^0)$ and that it is symmetric and positive definite on $\Lambda_h^0\times \Lambda_h^0$. Therefore \eqref{eqs: reduced problem} can be solved using a Krylov space iterative method such as the Conjugate Gradient or GMRES.  

\subsection{Preconditioner for the interface problem}
\label{sub: preconditioner}

To speed up the convergence of the iterative interface solver for \eqref{eqs: reduced problem} we develop a Dirichlet--Dirichlet type preconditioner, cf. the FETI Dirichlet preconditioner for the primal formulation \cite{Toselli-Widlund}, which requires solving local Dirichlet problems at each iteration: Given $g \in L^2(\Gamma)$, for $i = 1, \ldots, n_\Omega$, find $(\bm{u}_{h,i}^g, p_{h,i}^g) \in V_{h,i}\times W_{h,i}$ such that
\begin{subequations} \label{eq: R_h Dir}
\begin{align}
	a_i^h(\bm{u}_{h,i}^g, \bm{v}_{h, i})
	- b_i(\bm{v}_{h, i}, p_{h, i}^g)
	&= (g,\bm{v}_{h, i}\cdot\bnu_i)_{\Gamma_i}
	, && \forall \bm{v}_{h, i} \in V_{h, i}, \label{RhD-eq1}
	\\
	b_i(\bm{u}_{h,i}^g, w_{h, i}) 
	&= 0
	, && \forall w_{h, i} \in W_{h, i}. \label{RhD-eq2}
\end{align}
\end{subequations}
The preconditioner $M^{-1}: \Lambda_h' \to \Lambda_h'$ is defined as follows: 
$$
\forall \, g \in \Lambda_h', \quad \langle M^{-1} g, \mu \rangle = \sum_{i=1}^{n_\Omega} (\bm{u}_{h,i}^g\cdot\bnu_i,\mu)_{\Gamma_i} \quad \forall \mu \in \Lambda_h.
$$
In other words, the preconditioner takes Dirichlet mortar data on the interfaces, solves subdomain problems, and returns the jump in flux. This is known as a Dirichlet-to-Neumann operator.

Let the interface problem \eqref{eqs: reduced problem} be written in an operator form as
$$
S\lambda^0 = r \quad \mbox{in } (\Lambda_h^0)',
$$
where $S:\Lambda_h \to \Lambda_h'$ is defined as follows: $\forall \, \lambda \in \Lambda_h$, 
$\langle S\lambda,\mu \rangle = a_\Gamma(\lambda,\mu) \ \forall \, \mu \in \Lambda_h$. The preconditioned problem is
$$
M^{-1}S\lambda^0 = M^{-1}r \quad \mbox{in } (\Lambda_h^0)',
$$
which is implemented as solving
$$
P M^{-1}S P\lambda = P M^{-1}r \quad \mbox{in } \Lambda_h'
$$
and setting $\lambda^0 =  P\lambda$, where $P = I - B^T(B B^T)^{-1}B$ is the projection operator onto $\Lambda_h^0 = \ker(B)$. In practice, one of the applications of $P$ on the left hand side can be omitted, since the iterate is in $\Lambda_h^0$. In particular, we solve $P M^{-1}S \lambda = P M^{-1}r \mbox{ in } \Lambda_h'$. We note that 
$B B^T: S_H \to S_H'$, thus the application of $(B B^T)^{-1}$ involves solving a coarse problem. In fact, this is the same coarse operator as in \eqref{preproblem}, which is implemented as solving $B B^T r_f = \overline f$ in $S_H'$ and setting $\overline\lambda_f = B^T r_f$, as well as in \eqref{eq: postproblem}, which is implemented as solving $B B^T \overline{p}_\lambda = B \overline{g}$ in $S_H'$. We note that the operator $B B^T$ is invertible, since $B$ is an isomorphism from $\overline\Lambda_h$ to $S_H'$, which is shown in \cite[Lemma~2.4]{boon2020flux}.

The analysis of the condition number of the preconditioned operator $M^{-1}S$ is beyond the scope of this work. In \cite[Theorem~6.15]{Toselli-Widlund} it is shown that in the case of the primal formulation for elliptic problems with matching subdomain grids the condition number is $O(1 + \log(\frac{H}{h}))^2$, where $H$ is the subdomain size. This implies that the number of iterations of the preconditioned Krylov solver grows very weakly when the grids are refined. We observe a similar behavior in our numerical experiments.

\section{Numerical results} 
\label{sec: numerical results}

In this section, we present several numerical experiments in order to verify the analytical results and illustrate the flexibility of the proposed method. The experiments are subdivided into four examples. The first is presented
in~\Cref{sub: convergence tests} and investigates the orders of convergence with respect to the mesh size for different element types, as well as the efficiency of the interface preconditioner. \Cref{sub: fault} concerns an example that simulates a geological case with low permeable layers and a highly conductive fault zone. Third, we present an example with a highly heterogeneous permeability based on the Society of Petroleum Engineers SPE10 benchmark in \Cref{sub:SPE10}, which illustrates the multiscale capabilities of the method. Finally, in the example in \Cref{sub:highly_oscillatory_permeability} the subdomain grids are chosen according to the local spatial frequency of the permeability, resulting in an a suitable local resolution of the solution.

The numerical tests are performed using an implementation of the method in DuMu$^\text{X}$ \cite{Kochetal2020Dumux}, which uses the MPFA method of \Cref{ssub: MPFA} as the local discretization of the subdomain problems. Due to the close relationship with the MFMFE method, we expect the analytical results of \Cref{sub: error analysis simplices,sub: Error analysis quads} to remain valid for the MPFA method.

\iftrue 

\subsection{Example 1: Convergence tests}
\label{sub: convergence tests}

The \emph{a priori} analysis presented in \Cref{thm:error-final,thm:error-final-RT} shows that the flux-mortar MFMFE is expected to converge with first order in both the pressure and velocity variables. We test this by considering examples with known solutions in two and a three dimensions. The domains are decomposed into subdomains of equal size. Starting with an initial coarse discretization with non-matching meshes, we run a sequence of uniform mesh refinements. We investigate the rates of convergence in the $L^2$-norm for all variables except for $\bm{u}$, for which we use the following discrete error norm, which is equivalent to the $L^2$-norm in the space $V_h$:
\begin{equation}
    e_u = \sqrt{
      \sum_{\omega_i \in \Omega_h} 
      | \omega_i |
      \sum_{\gamma_{ij} \subset \partial \omega_i}
\frac{1}{| \gamma_{ij}|}
      \int_{\gamma_{ij}} \left(\bm{u} \cdot \bm{n} - \bm{u}_h \cdot \bm{n} \right)^2
    }.
\end{equation}

For both setups, we choose a continuous, piecewise linear mortar space $\mathcal{P}_1$ on $\Gamma$ and investigate both variants of the projection operator $\myQ_h^\flat$ and $\myQ_h^\sharp$.

\subsubsection{Two-dimensional setup}
\label{ssub: 2D}
Let the domain $\Omega = \left(0, 2\right) \times \left(0, 2\right)$, the
permeability $K = I$, and the pressure be given by:
    \begin{align}
        \tilde{p} \left( x, y \right)
    		&= y^2 \left( 1 - \frac{y}{3} \right)
    		   + x \left( 2-x \right) y \sin \left( 2 \pi x \right).
                   \label{eq:conv_2d_exact}
    \end{align}
Dirichlet boundary conditions are used and are given by~\cref{eq:conv_2d_exact}. The domain is decomposed into $3 \times 3$ subdomains of equal size,
discretized by unstructured meshes composed of either triangles or quadrilaterals.
On the coarsest level the subdomain grids have either six or eight elements along each subdomain side, while the mortar grid has three elements on each interface. We note that on the coarsest level the quadrilateral mesh has general quadrilaterals, but the uniform refinement strategy results in $h^2$-parallelograms.
\Cref{fig:conv_2d_results} shows the pressure and velocity distributions after
the first refinement, while
\Cref{tab:conv_2d_triangles,tab:conv_2d_quads}
show the errors and rates obtained on all refinements. 
First-order convergence can be observed for $p$ and $\bm{u}$ on both meshes and with both projection operators. We further observe that the number of iterations increases slightly upon refinement, which is consistent with the expected condition number $O(1 + \log(\frac{H}{h}))^2$ of the preconditioned interface operator, cf. \Cref{sub: preconditioner}.

\begin{figure}
    \centering
    \begin{subfigure}{0.32\textwidth}
        \includegraphics[width=\textwidth]{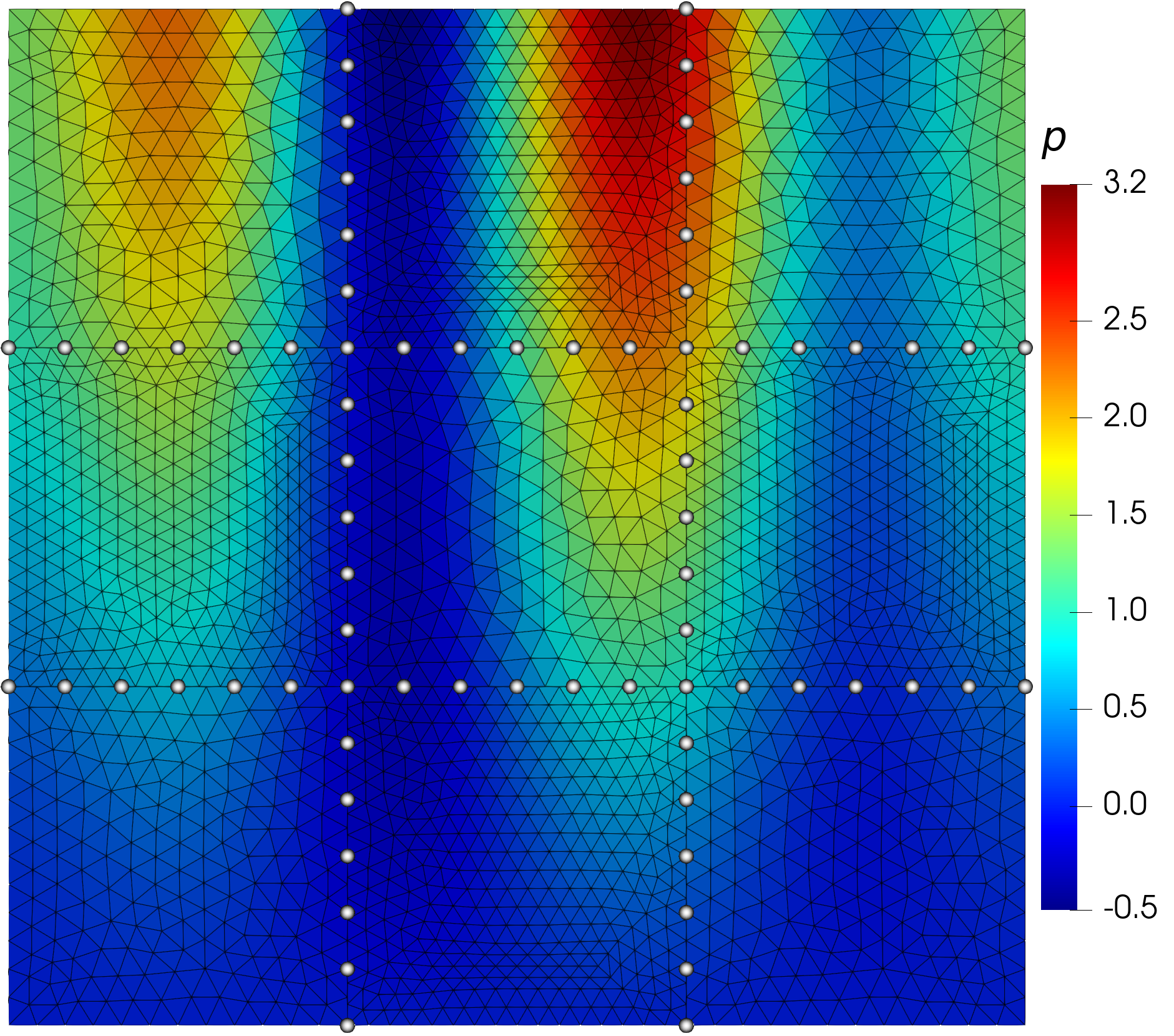}
        \caption{}
    \end{subfigure}
    \begin{subfigure}{0.32\textwidth}
        \includegraphics[width=\textwidth]{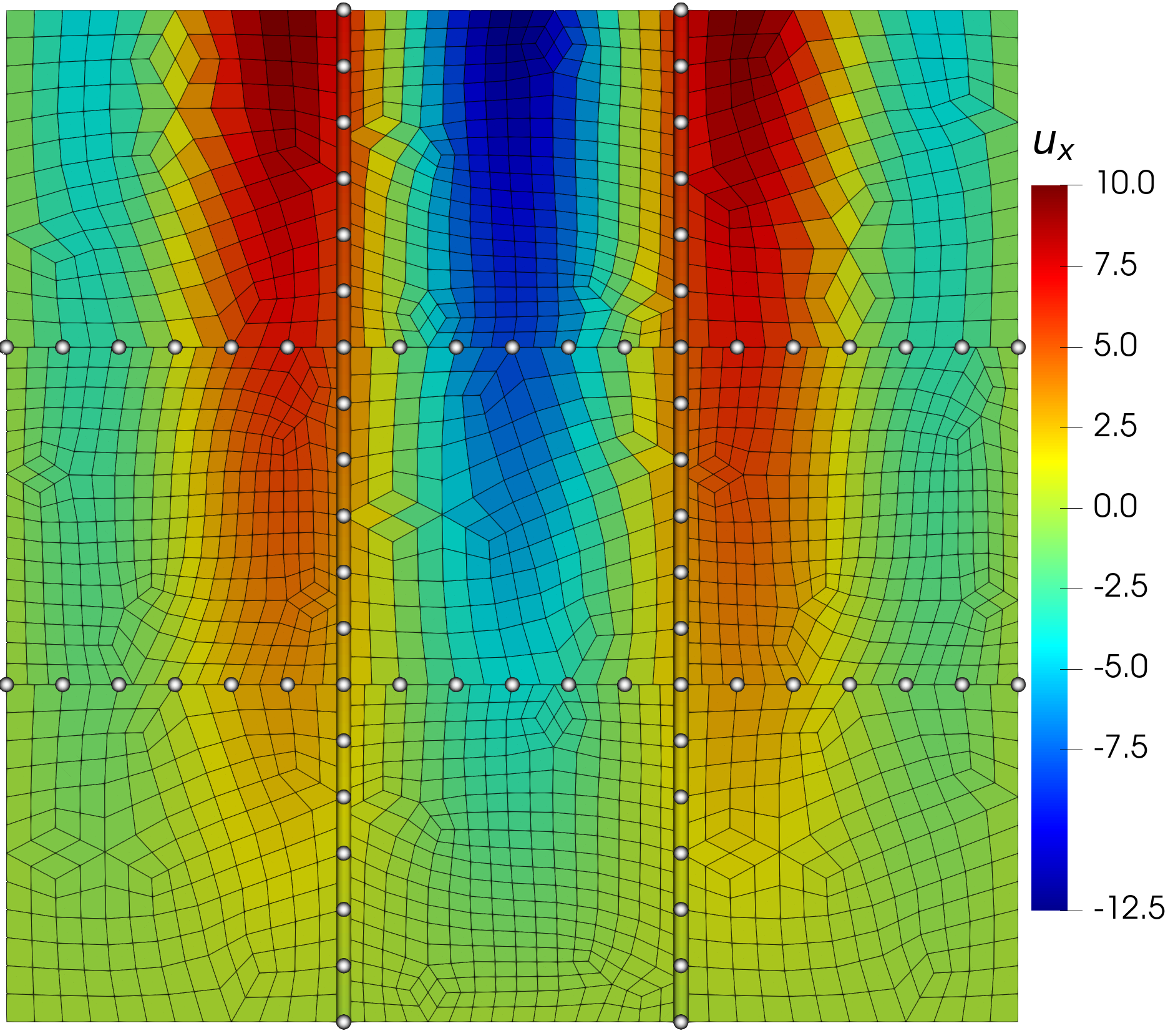}
        \caption{}
    \end{subfigure}
    \begin{subfigure}{0.32\textwidth}
        \includegraphics[width=\textwidth]{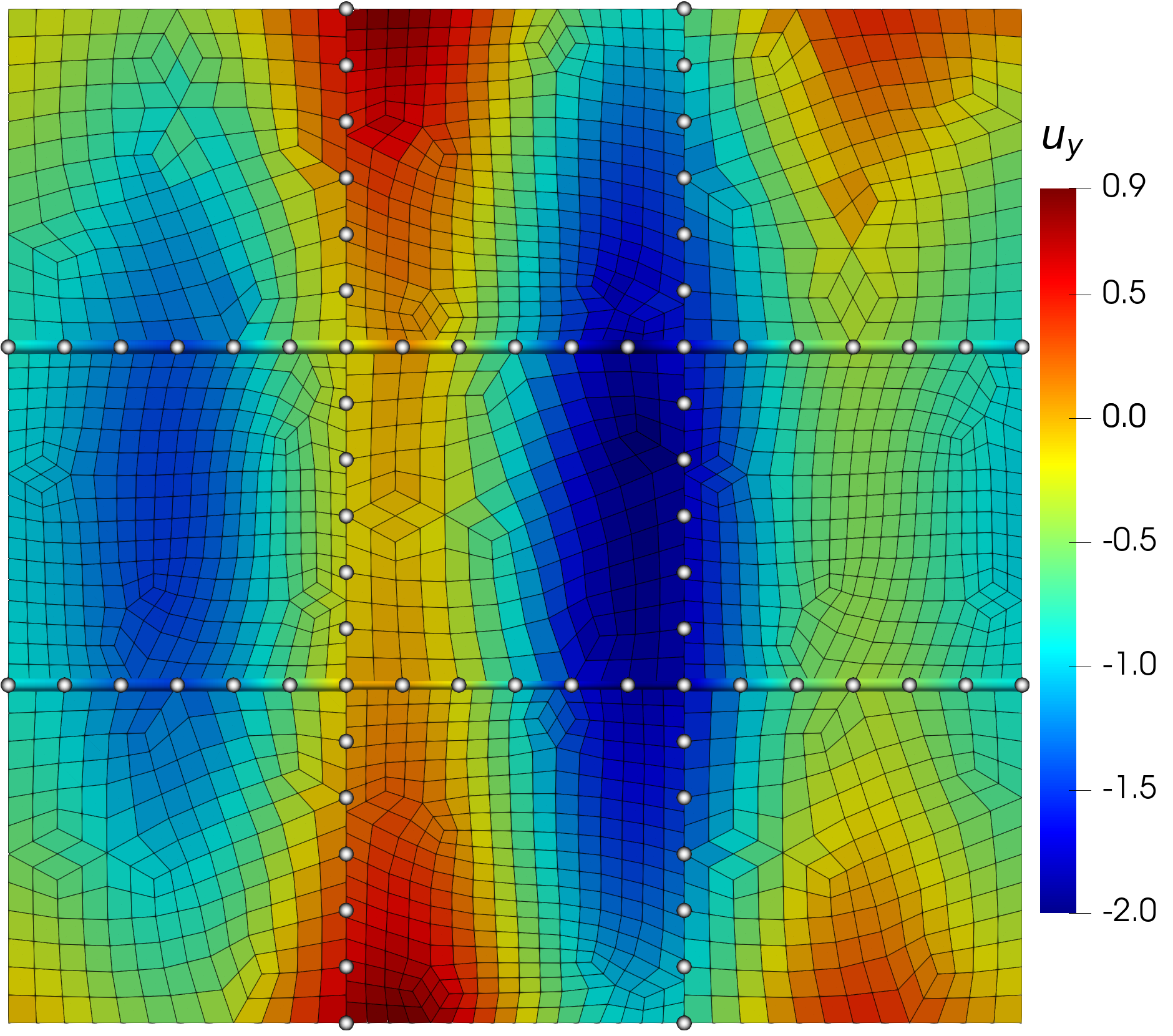}
        \caption{}
    \end{subfigure}
    \caption{Example 1 (2D): numerical solution on the first mesh refinement. Left: pressure on triangular grids. Middle and right: $x$ and $y$ components of the velocity on quadrilateral grids. The mortar grid vertices are indicated by white circles and the tubes in the velocity plots show the mortar solution.}
    \label{fig:conv_2d_results}
\end{figure}

\begin{table}
    {\footnotesize
      \caption{Example 1 (2D): errors and convergence rates on triangular grids.}
    \label{tab:conv_2d_triangles}
    \begin{center}
    \begin{tabular}{l |l l |l l |l l |l l |l}
        \toprule
        $h_{\mathrm{min}}$ & $e_u^\flat$ & $r_u^\flat$ & $e_p^\flat$ & $r_p^\flat$ & $e_\lambda^\flat$ & $r_\lambda^\flat$ & $e_{\mathcal{Q}\lambda}^\flat$ & $r_{\mathcal{Q}\lambda}^\flat$ & $\#it$ \\
        \midrule
        7.11e-02 & 1.42e+00 &      & 1.58e-01 &      & 1.93e-01 &      & 4.47e-01 &      & 29 \\
        3.56e-02 & 7.21e-01 & 0.98 & 7.88e-02 & 1.01 & 8.23e-02 & 1.23 & 2.18e-01 & 1.04 & 33 \\
        1.78e-02 & 3.64e-01 & 0.99 & 3.93e-02 & 1.00 & 4.82e-02 & 0.77 & 1.14e-01 & 0.93 & 36 \\
        8.89e-03 & 1.83e-01 & 0.99 & 1.97e-02 & 1.00 & 3.27e-02 & 0.56 & 6.42e-02 & 0.83 & 37 \\
        4.44e-03 & 9.15e-02 & 1.00 & 9.83e-03 & 1.00 & 2.37e-02 & 0.46 & 3.93e-02 & 0.71 & 38 \\
        2.22e-03 & 4.58e-02 & 1.00 & 4.92e-03 & 1.00 & 1.79e-02 & 0.41 & 2.65e-02 & 0.57 & 39 \\
        \toprule
        $h_{\mathrm{min}}$ & $e_u^\sharp$ & $r_u^\flat$ & $e_p^\sharp$ & $r_p^\sharp$ & $e_\lambda^\sharp$ & $r_\lambda^\sharp$ & $e_{\mathcal{Q}\lambda}^\sharp$ & $r_{\mathcal{Q}\lambda}^\sharp$ & $\#it$ \\
        \midrule
        7.11e-02 & 1.42e+00 &      & 1.58e-01 &      & 1.94e-01 &      & 4.48e-01 &      & 29 \\
        3.56e-02 & 7.21e-01 & 0.98 & 7.88e-02 & 1.01 & 8.23e-02 & 1.23 & 2.18e-01 & 1.04 & 33 \\
        1.78e-02 & 3.64e-01 & 0.99 & 3.93e-02 & 1.00 & 4.82e-02 & 0.77 & 1.14e-01 & 0.93 & 35 \\
        8.89e-03 & 1.83e-01 & 0.99 & 1.97e-02 & 1.00 & 3.27e-02 & 0.56 & 6.42e-02 & 0.83 & 37 \\
        4.44e-03 & 9.15e-02 & 1.00 & 9.83e-03 & 1.00 & 2.37e-02 & 0.46 & 3.93e-02 & 0.71 & 38 \\
        2.22e-03 & 4.58e-02 & 1.00 & 4.92e-03 & 1.00 & 1.79e-02 & 0.41 & 2.65e-02 & 0.57 & 39 \\
        \bottomrule
    \end{tabular}
    \end{center}
    }
\end{table}

\begin{table}
    {\footnotesize
      \caption{Example 1 (2D): errors and convergence rates on quadrilateral grids.}
    \label{tab:conv_2d_quads}
    \begin{center}
    \begin{tabular}{l |l l |l l |l l |l l |l}
        \toprule
        $h_{\mathrm{min}}$ & $e_u^\flat$ & $r_u^\flat$ & $e_p^\flat$ & $r_p^\flat$ & $e_\lambda^\flat$ & $r_\lambda^\flat$ & $e_{\mathcal{Q}\lambda}^\flat$ & $r_{\mathcal{Q}\lambda}^\flat$ & $\#it$ \\
        \midrule
        8.31e-02 & 1.17e+00 &      & 2.16e-01 &      & 2.74e-01 &      & 5.24e-01 &      & 30 \\
        4.12e-02 & 6.19e-01 & 0.91 & 1.06e-01 & 1.01 & 1.02e-01 & 1.40 & 2.34e-01 & 1.15 & 34 \\
        2.03e-02 & 3.21e-01 & 0.93 & 5.29e-02 & 0.99 & 5.05e-02 & 1.00 & 1.17e-01 & 0.99 & 36 \\
        1.01e-02 & 1.64e-01 & 0.96 & 2.64e-02 & 1.00 & 3.11e-02 & 0.69 & 6.30e-02 & 0.88 & 38 \\
        5.05e-03 & 8.32e-02 & 0.98 & 1.32e-02 & 1.00 & 2.19e-02 & 0.50 & 3.76e-02 & 0.74 & 39 \\
        2.52e-03 & 4.18e-02 & 0.99 & 6.61e-03 & 1.00 & 1.65e-02 & 0.41 & 2.49e-02 & 0.59 & 40 \\
        \toprule
        $h_{\mathrm{min}}$ & $e_u^\sharp$ & $r_u^\flat$ & $e_p^\sharp$ & $r_p^\sharp$ & $e_\lambda^\sharp$ & $r_\lambda^\sharp$ & $e_{\mathcal{Q}\lambda}^\sharp$ & $r_{\mathcal{Q}\lambda}^\sharp$ & $\#it$ \\
        \midrule
        8.31e-02 & 1.17e+00 &      & 2.16e-01 &      & 2.74e-01 &      & 5.25e-01 &      & 30 \\
        4.12e-02 & 6.19e-01 & 0.91 & 1.06e-01 & 1.01 & 1.02e-01 & 1.40 & 2.34e-01 & 1.15 & 34 \\
        2.03e-02 & 3.21e-01 & 0.93 & 5.29e-02 & 0.99 & 5.05e-02 & 0.99 & 1.17e-01 & 0.99 & 36 \\
        1.01e-02 & 1.64e-01 & 0.96 & 2.64e-02 & 1.00 & 3.11e-02 & 0.69 & 6.30e-02 & 0.88 & 38 \\
        5.05e-03 & 8.32e-02 & 0.98 & 1.32e-02 & 1.00 & 2.19e-02 & 0.50 & 3.75e-02 & 0.75 & 39 \\
        2.52e-03 & 4.18e-02 & 0.99 & 6.61e-03 & 1.00 & 1.65e-02 & 0.41 & 2.49e-02 & 0.59 & 40 \\
        \bottomrule
    \end{tabular}
    \end{center}
    }
\end{table}

\subsubsection{Three-dimensional setup}
\label{ssub: 3D}

Let the domain now be $\Omega = \left(0, 2\right) \times \left(0, 2\right) \times \left(0, 2\right)$.
Based on \eqref{eq:conv_2d_exact}, we construct a three-dimensional analytical solution by multiplying 
$p(x, y)$ with $\cos \left( 2 \pi z \right)$ and use it as Dirichlet
boundary condition. The domain is subdivided into 8 subdomains,
and we again consider two grid types comprising tetrahedra or hexahedra.
On the coarsest level, the subdomain tetrahedral grids have either two or three elements along each subdomain edge, while the hexahedral grids have either three or five elements. In both cases the mortar grid is $2\times 2$ quadrilateral on each interface. Similarly to the 2D case, the uniform refinement in the hexahedral grids results in $h^2$-parallelepiped elements. \Cref{fig:conv_3d_results} shows the pressure distributions obtained on the second refinement,
and~\Cref{tab:conv_3d_hexes,tab:conv_3d_tets}
list the errors and convergence rates for all refinements.
First-order convergence in $p$ and $\bm{u}$ is again observed for both meshes and projection operators. The number of iterations grows slowly with the
refinement level, however the dependence is more pronounced compared to the two-dimensional tests. This may be due to the effect of the projection between the mortar space and the normal trace of the subdomain velocity spaces, which in 3D are defined on two-dimensional grids. Moreover, using the projection operator $\myQ_h^\sharp$ requires several more iterations in comparison with $\myQ_h^\flat$, possibly caused by the fact that the preconditioner from \Cref{sub: preconditioner} is based on the $L^2$-projection $\myQ_h^\flat$.

\begin{figure}
    \centering
    \begin{subfigure}{0.32\textwidth}
        \includegraphics[width=\textwidth]{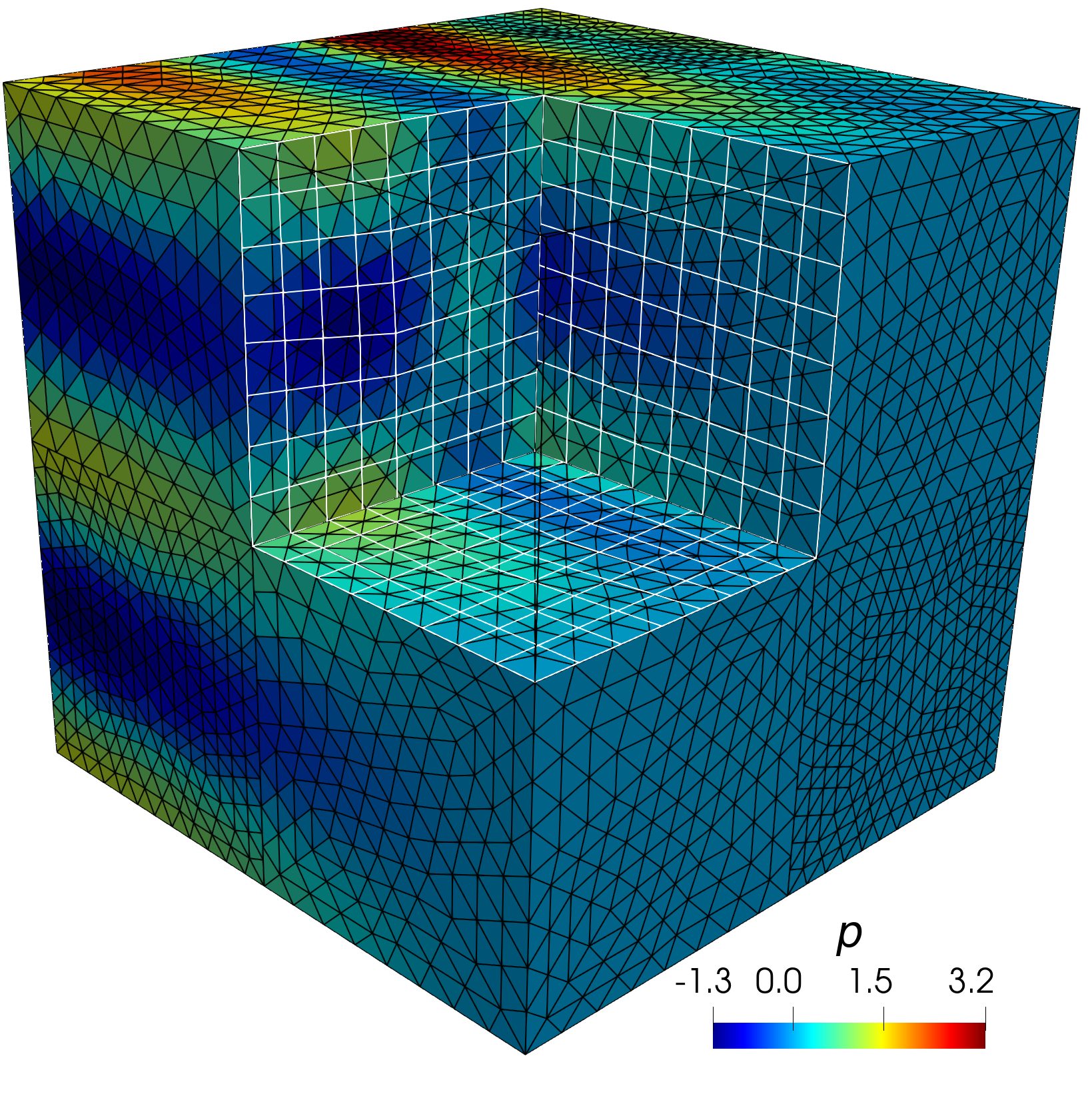}
        \caption{}
    \end{subfigure}
    \begin{subfigure}{0.32\textwidth}
        \includegraphics[width=\textwidth]{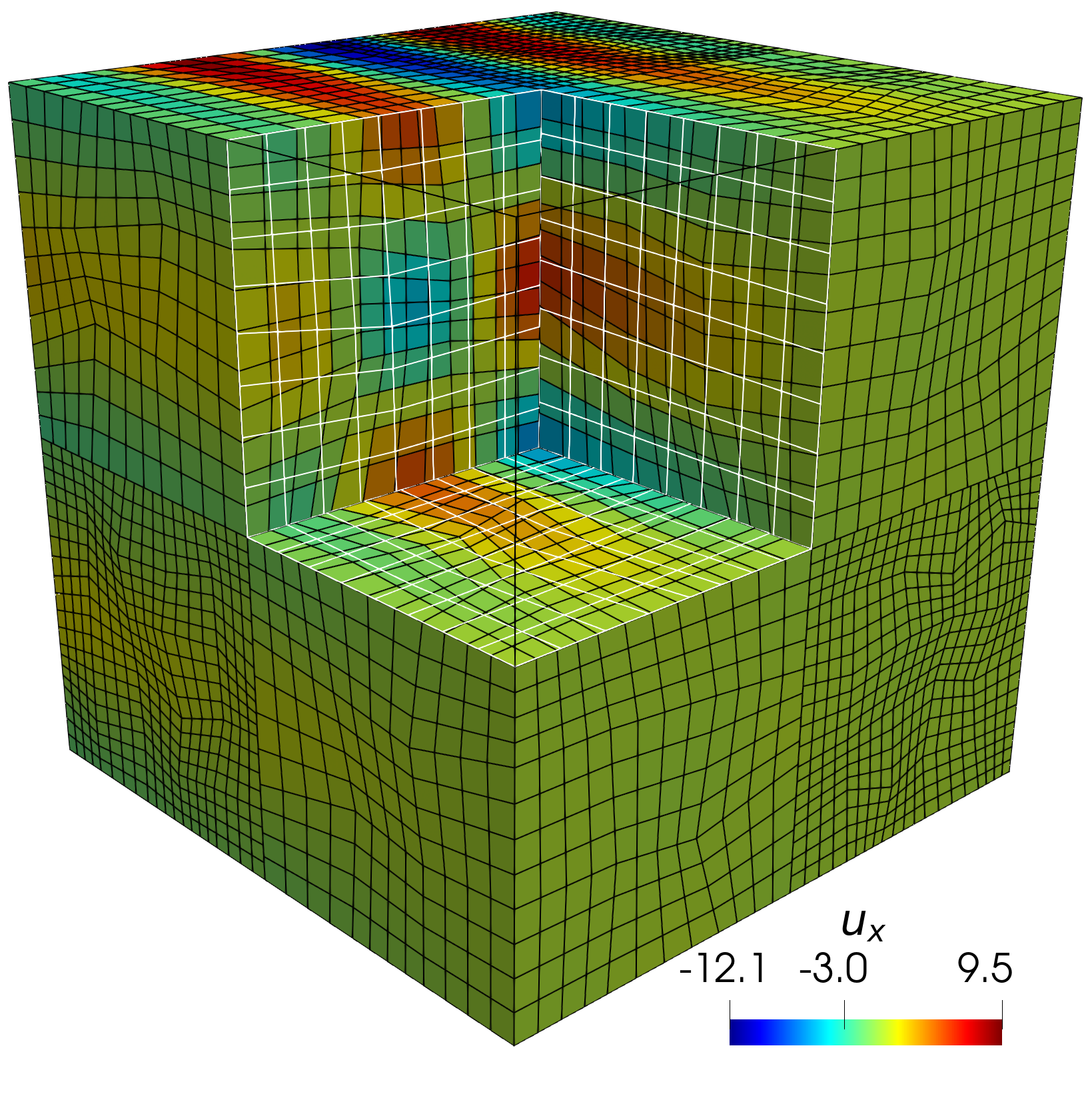}
        \caption{}
    \end{subfigure}
    \begin{subfigure}{0.32\textwidth}
        \includegraphics[width=\textwidth]{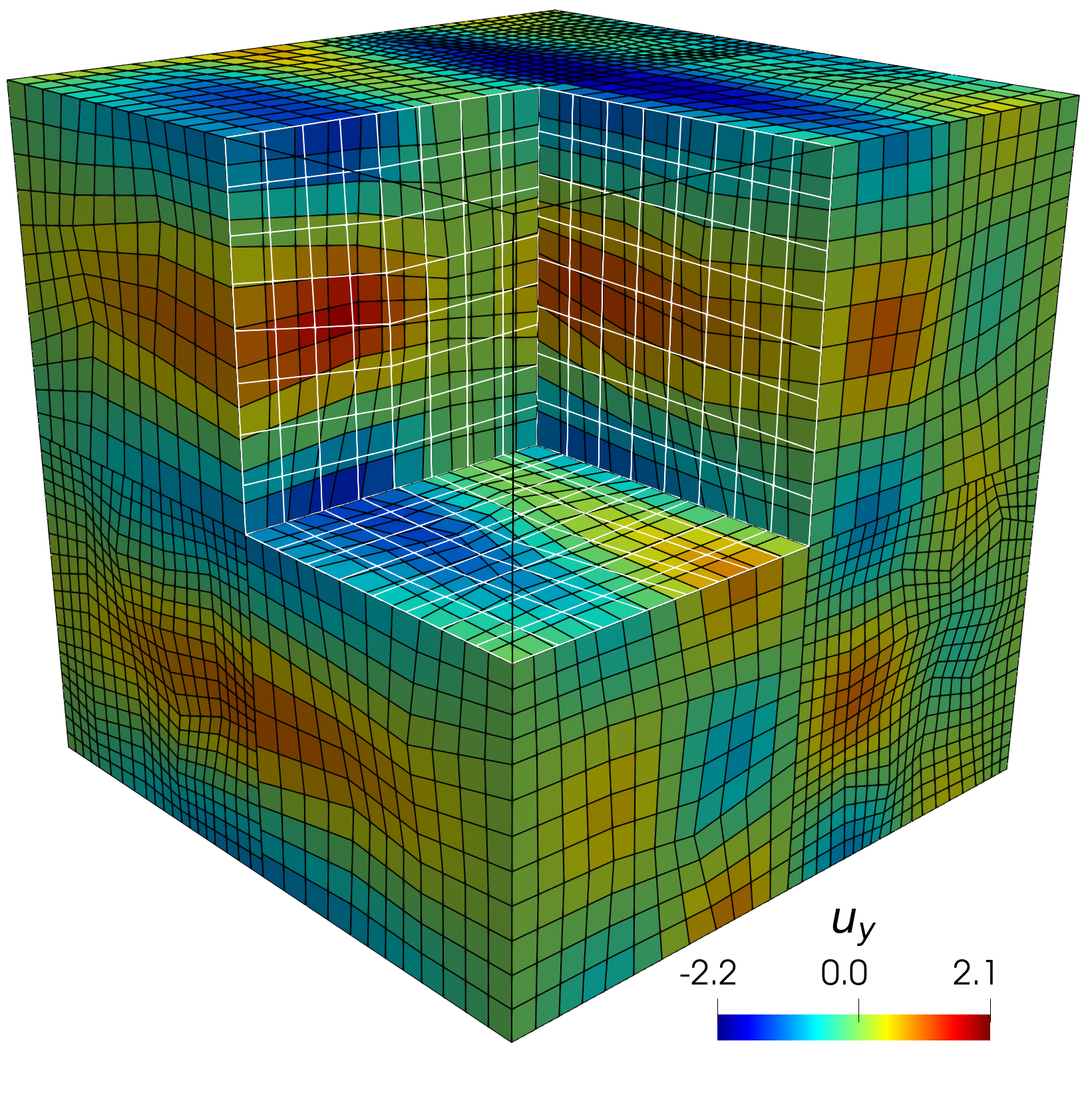}
        \caption{}
    \end{subfigure}
    \caption{Example 1 (3D): numerical solution on the second mesh refinement. Left: pressure on tetrahedral grids. Middle and right: $x$ and $y$ components of the velocity on hexahedral grids. The white lines visualize the mortar grid.}
    \label{fig:conv_3d_results}
\end{figure}

\begin{table}
    {\footnotesize
      \caption{Example 1 (3D): errors and convergence rates on tetrahedral grids.} 
    \label{tab:conv_3d_tets}
    \begin{center}
    \begin{tabular}{l |l l |l l |l l |l l |l}
        \toprule
        $h_{\mathrm{min}}$ & $e_u^\flat$ & $r_u^\flat$ & $e_p^\flat$ & $r_p^\flat$ & $e_\lambda^\flat$ & $r_\lambda^\flat$ & $e_{\mathcal{Q}\lambda}^\flat$ & $r_{\mathcal{Q}\lambda}^\flat$ & $\#it$ \\
        \midrule
        3.43e-01 & 1.62e+01 &      & 1.28e+00 &      & 3.40e+00 &      & 8.24e+00 &      & 67 \\
        1.71e-01 & 8.93e+00 & 0.86 & 6.13e-01 & 1.06 & 1.80e+00 & 0.92 & 4.10e+00 & 1.01 & 89 \\
        8.57e-02 & 4.72e+00 & 0.92 & 3.02e-01 & 1.02 & 5.63e-01 & 1.67 & 1.93e+00 & 1.08 & 105 \\
        4.28e-02 & 2.41e+00 & 0.97 & 1.51e-01 & 1.00 & 2.70e-01 & 1.06 & 9.71e-01 & 1.00 & 125 \\
        \toprule
        $h_{\mathrm{min}}$ & $e_u^\sharp$ & $r_u^\flat$ & $e_p^\sharp$ & $r_p^\sharp$ & $e_\lambda^\sharp$ & $r_\lambda^\sharp$ & $e_{\mathcal{Q}\lambda}^\sharp$ & $r_{\mathcal{Q}\lambda}^\sharp$ & $\#it$ \\
        \midrule
        3.43e-01 & 1.64e+01 &      & 1.28e+00 &      & 5.49e+00 &      & 8.95e+00 &      & 76 \\
        1.71e-01 & 8.94e+00 & 0.88 & 6.13e-01 & 1.06 & 1.85e+00 & 1.57 & 4.13e+00 & 1.12 & 96 \\
        8.57e-02 & 4.73e+00 & 0.92 & 3.02e-01 & 1.02 & 5.77e-01 & 1.68 & 1.94e+00 & 1.09 & 112 \\
        4.28e-02 & 2.41e+00 & 0.97 & 1.51e-01 & 1.00 & 2.74e-01 & 1.07 & 9.72e-01 & 1.00 & 133 \\
        \bottomrule
    \end{tabular}
    \end{center}
    }
\end{table}

\begin{table}
    {\footnotesize
      \caption{Example 1 (3D): errors and convergence rates on hexahedral grids.} 
    \label{tab:conv_3d_hexes}
    \begin{center}
    \begin{tabular}{l |l l |l l |l l |l l |l}
        \toprule
        $h_{\mathrm{min}}$ & $e_u^\flat$ & $r_u^\flat$ & $e_p^\flat$ & $r_p^\flat$ & $e_\lambda^\flat$ & $r_\lambda^\flat$ & $e_{\mathcal{Q}\lambda}^\flat$ & $r_{\mathcal{Q}\lambda}^\flat$ & $\#it$ \\
        \midrule
        3.17e-01 & 2.02e+01 &      & 1.48e+00 &      & 3.91e+00 &      & 8.87e+00 &      & 57 \\
        1.47e-01 & 9.21e+00 & 1.02 & 6.45e-01 & 1.08 & 1.89e+00 & 0.95 & 4.38e+00 & 0.92 & 73 \\
        7.13e-02 & 4.45e+00 & 1.01 & 3.10e-01 & 1.01 & 5.98e-01 & 1.59 & 2.10e+00 & 1.02 & 88 \\
        3.42e-02 & 2.21e+00 & 0.96 & 1.54e-01 & 0.96 & 3.31e-01 & 0.81 & 1.07e+00 & 0.92 & 105 \\
        \toprule
        $h_{\mathrm{min}}$ & $e_u^\sharp$ & $r_u^\flat$ & $e_p^\sharp$ & $r_p^\sharp$ & $e_\lambda^\sharp$ & $r_\lambda^\sharp$ & $e_{\mathcal{Q}\lambda}^\sharp$ & $r_{\mathcal{Q}\lambda}^\sharp$ & $\#it$ \\
        \midrule
        3.17e-01 & 2.02e+01 &      & 1.46e+00 &      & 3.87e+00 &      & 9.12e+00 &      & 58 \\
        1.47e-01 & 9.22e+00 & 1.02 & 6.44e-01 & 1.07 & 1.85e+00 & 0.96 & 4.38e+00 & 0.95 & 74 \\
        7.13e-02 & 4.45e+00 & 1.01 & 3.10e-01 & 1.01 & 5.82e-01 & 1.60 & 2.10e+00 & 1.02 & 88 \\
        3.42e-02 & 2.21e+00 & 0.96 & 1.54e-01 & 0.96 & 3.27e-01 & 0.79 & 1.07e+00 & 0.92 & 106 \\
        \bottomrule
    \end{tabular}
    \end{center}
    }
\end{table}

\subsection{Example 2: Faulted geology}
\label{sub: fault}

This example illustrates the flexibility of the method with respect to the choice
of computational meshes in different parts of the domain. This is particularly
useful in geological applications, which often involve layered
structures of materials with different properties, as well as faults and fractures.

\subsubsection{Two-dimensional setup}

We consider a faulted geology consisting of two permeable layers on top and bottom, separated by a low-permeable barrier and cut through by a high-permeable fault.
Dirichlet boundary conditions are applied on the left and right boundaries in the
permeable layers, while no-flow conditions are imposed on all remaining boundaries.
An illustration of the domain, the computational grid, and the chosen Dirichlet boundary conditions is given in~\Cref{fig:fault_2d_setup_grid_and_bcs}.
We use $K=I$ in the top and bottom layers, while $K=\num{e3}I$ and $K=\num{e-4}I$
are used in the fault zone and the barrier layers, respectively. 
The subdomain mesh sizes are chosen depending on the permeability such that
elhighly-permeable regions are discretized with finer meshes, yielding a
total number of cells of \num{24262}.

The computed solution is shown in~\Cref{fig:fault_2d_setup_solution}. For comparison, a monolithic reference solution is shown in~\Cref{fig:fault_2d_setup_solution_conforming}, which is computed on a fine conforming mesh with \num{66020} cells given in \Cref{fig:fault_2d_setup_grid_conforming}. Despite the difference in the number of degrees of freedom, a very good agreement in both pressure $p$ and the velocity $\bm{u}$ can be observed. Furthermore, plots of the pressure along the diagonal of the domain and the flux along the interface highlighted in~\Cref{fig:fault_2d_setup_grid_and_bcs} are shown in~\Cref{fig:fault_plots_2d} for both the flux-mortar and the fine scale solutions. Again, even though the flux-mortar solution uses coarser subdomain grids that do not match along the interfaces, as well as coarse mortar grids, it matches very well with the fine scale solution in both the pressure and interface flux.

\begin{figure}
    \centering
    \begin{subfigure}{0.49\textwidth}
        \centering
        \includegraphics[width=0.78\textwidth]{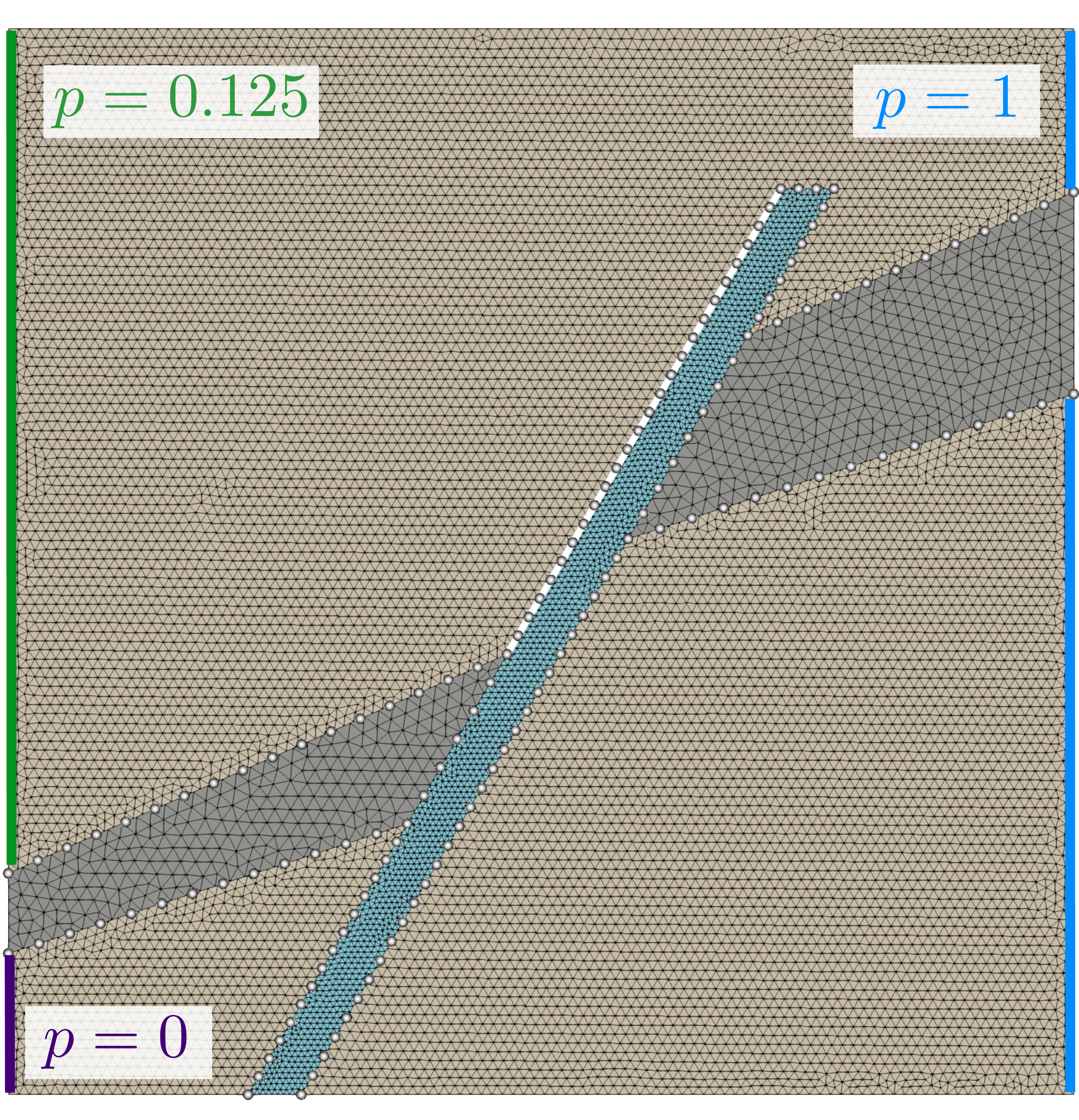}
        \caption{}
        \label{fig:fault_2d_setup_grid_and_bcs}
    \end{subfigure}
     \begin{subfigure}{0.49\textwidth}
        \centering
        \includegraphics[width=0.8\textwidth]{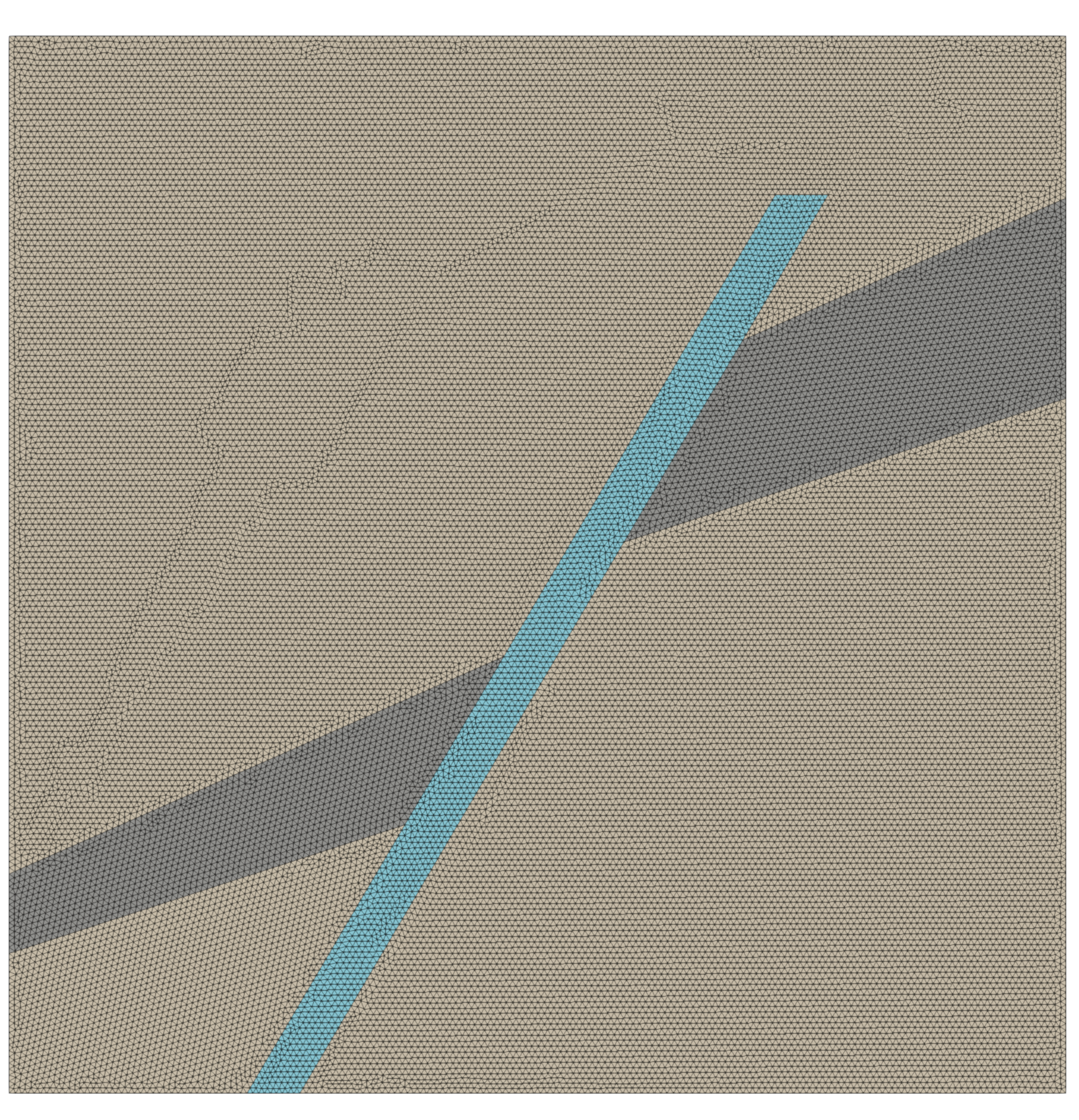}
        \caption{}
        \label{fig:fault_2d_setup_grid_conforming}
    \end{subfigure}
    \begin{subfigure}{0.45\textwidth}
        \centering
        \includegraphics[width=\textwidth]{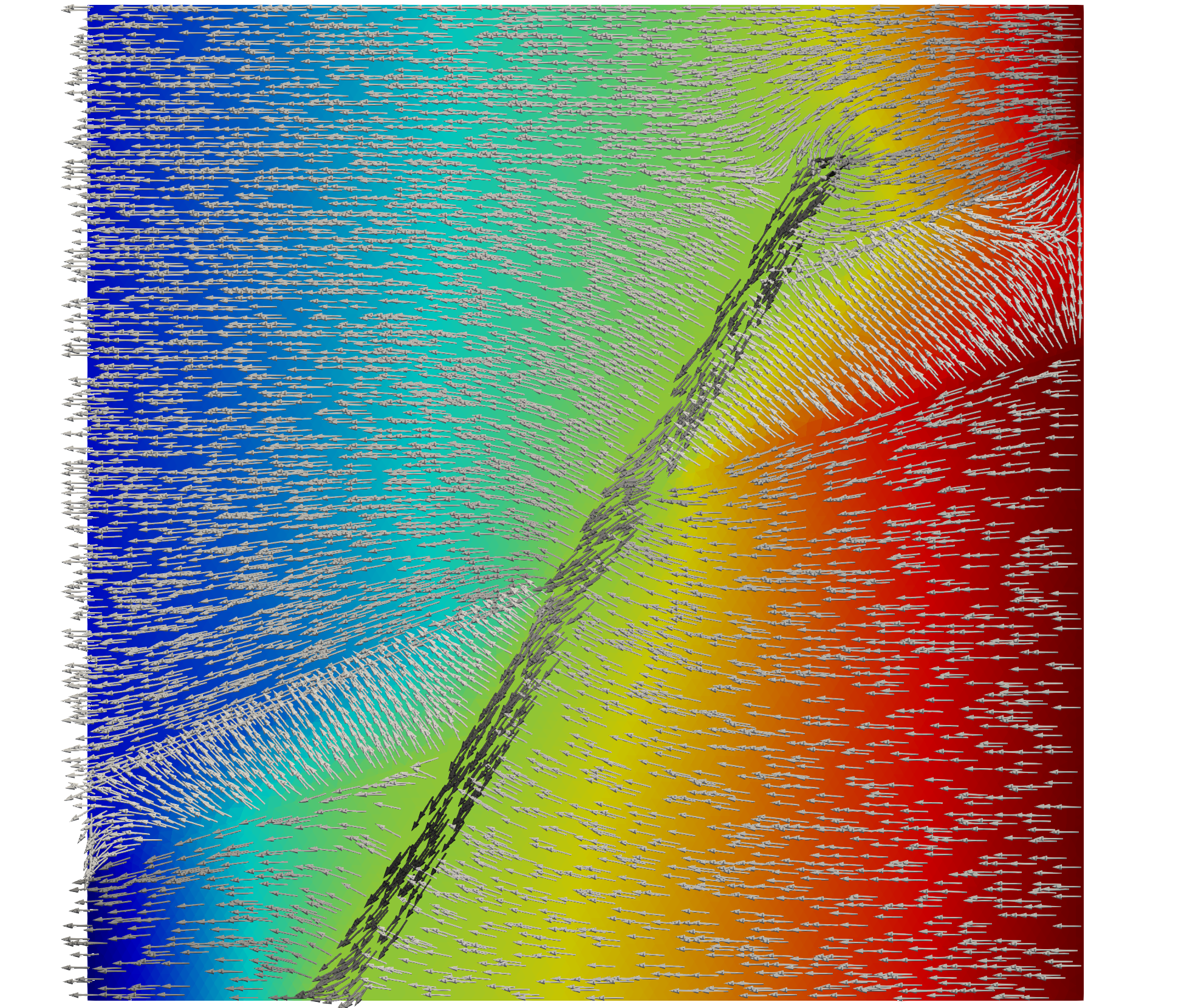}
        \caption{}
        \label{fig:fault_2d_setup_solution}
    \end{subfigure}
    \begin{subfigure}{0.45\textwidth}
        \centering
        \includegraphics[width=\textwidth]{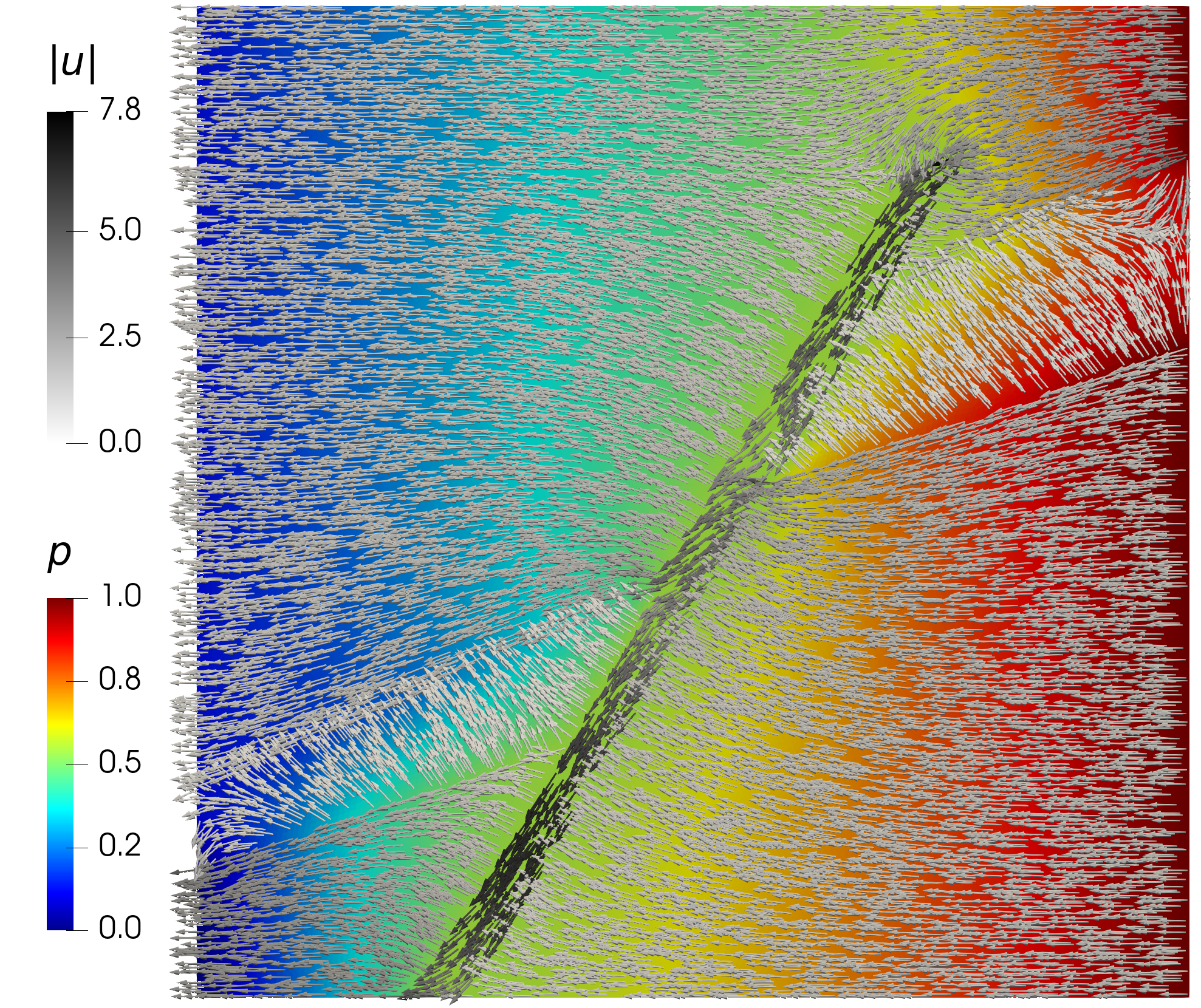}
        \caption{}
        \label{fig:fault_2d_setup_solution_conforming}
    \end{subfigure}
    \caption{Example~2 (2D): (a) Domain, grid and boundary conditions. The grey circles depict the vertices of the mortar mesh and the white line illustrates the interface across which the fluxes are plotted in~\Cref{fig:fault_plots_2d}. (b) The fine conforming mesh for comparison of the results. (c) and (d) Pressure and velocity distributions obtained with the flux-mortar method and on the fine mesh, respectively.}
    \label{fig:fault_2d_setup}
\end{figure}

\begin{figure}[h]
    \centering
    \begin{subfigure}{0.35\textwidth}
        \includegraphics[width=\textwidth]{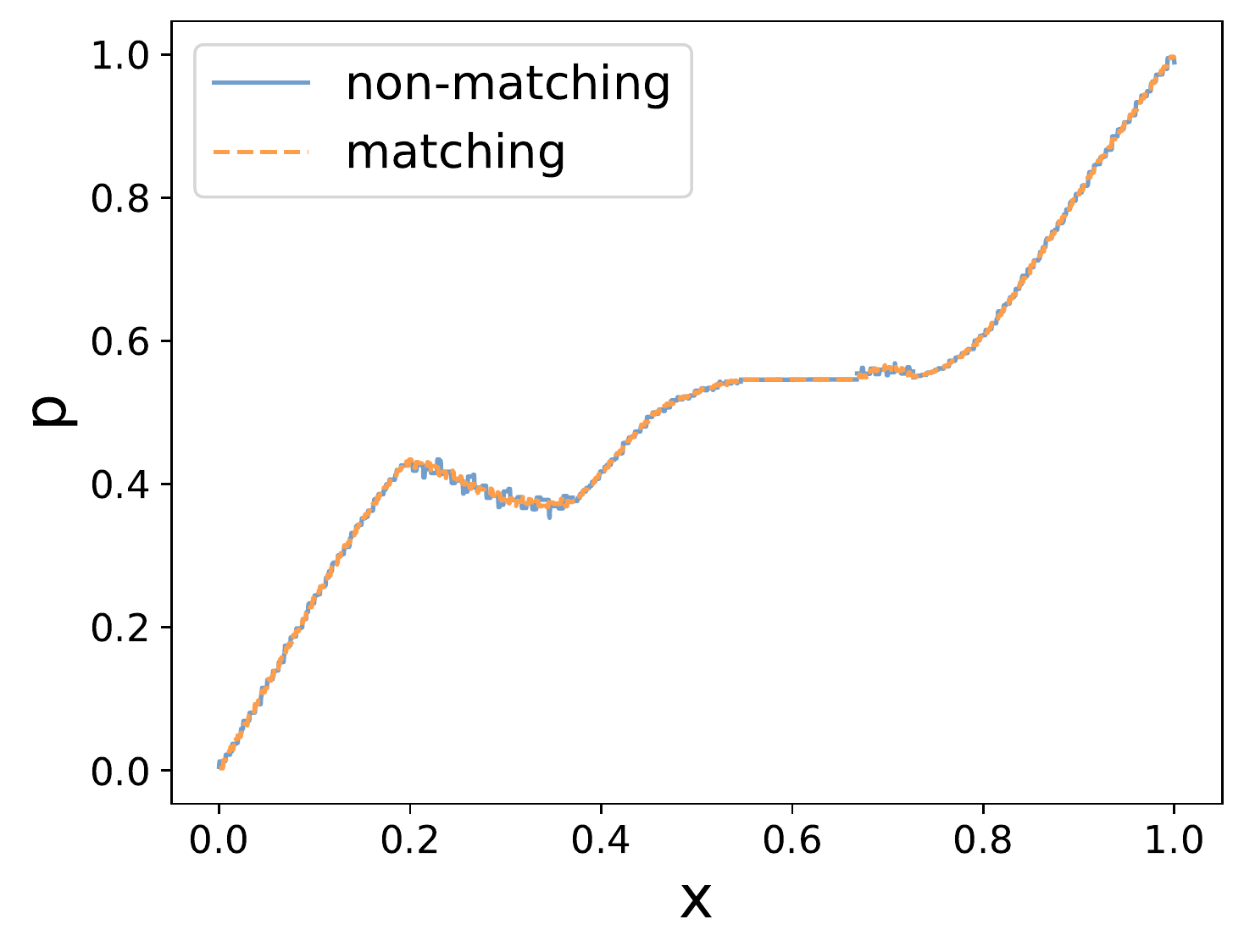}
        \caption{}
    \end{subfigure}
    \begin{subfigure}{0.35\textwidth}
        \includegraphics[width=\textwidth]{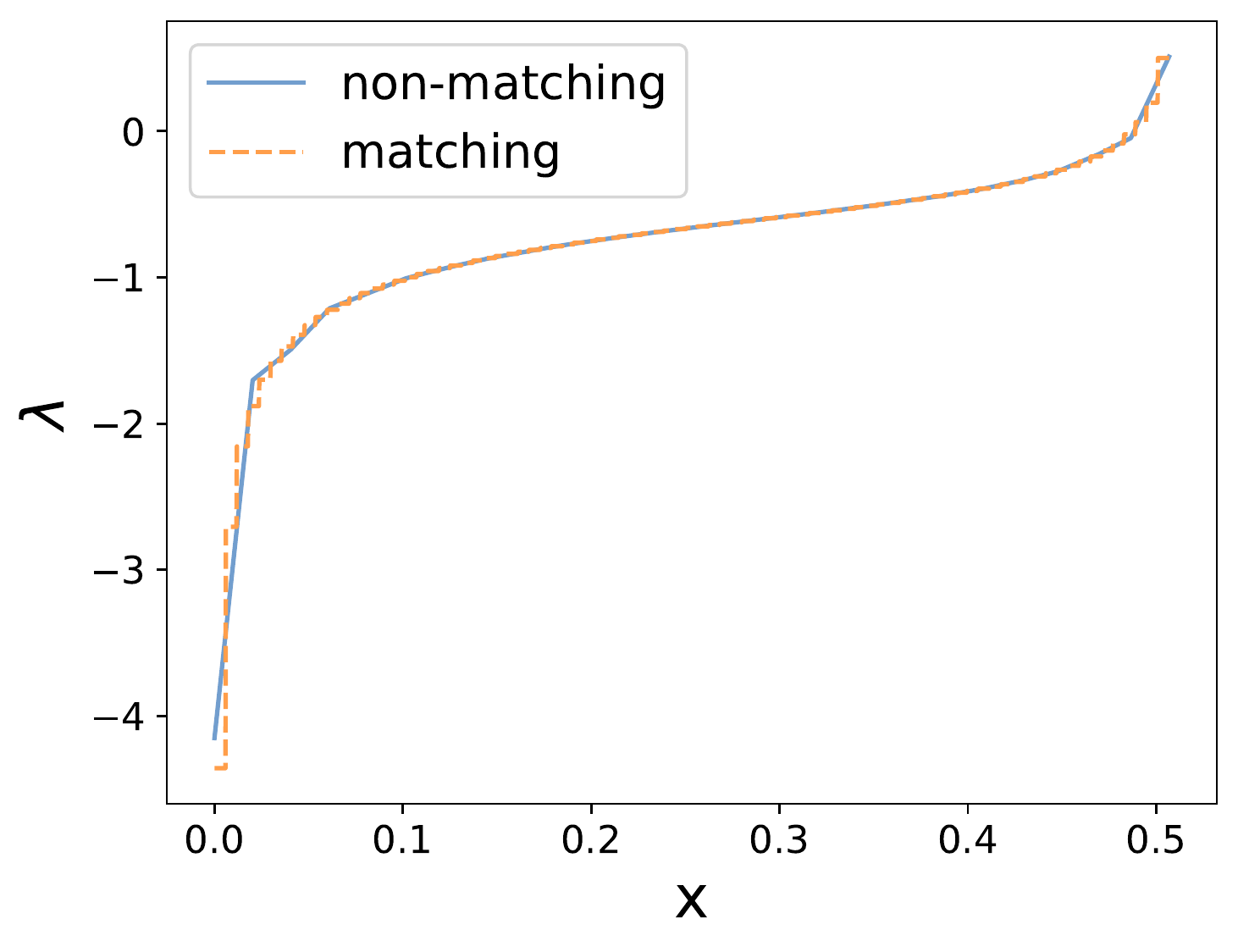}
        \caption{}
    \end{subfigure}
    \caption{Example 2 (2D): (a) Pressure along the diagonal of the domain; (b) Interface flux along the white line depicted in~\Cref{fig:fault_2d_setup_grid_and_bcs}.}
    \label{fig:fault_plots_2d}
\end{figure}

\subsubsection{Three-dimensional setup}
A three-dimensional variant of the setup shown in~\Cref{fig:fault_2d_setup_grid_and_bcs}
is obtained by extrusion of the domain in the third dimension, and we modify the 
boundary conditions to incorporate a pressure drop of $\Delta p = 0.25$ in the 
y-direction. The computational grid consists of \num{86266} cells in total, and we again compare the results against a monolithic reference solution obtained on a conforming grid with \num{133626} cells. \Cref{fig:fault_results_3d} provides an illustration of the domain, the grids and the results. As in the 2D case, we observe very good agreement between the flux-mortar and fine scale solutions. 
\Cref{fig:fault_plots_3d} shows plots of the pressure along the diagonal of the domain and the interface flux along the purple line shown in~\Cref{fig:fault_results_3d_pressure}. It can be observed that the pressure agrees very well with the reference solution inside the
fault and in the permeable layers, while larger deviations seem to occur inside the
barrier layers. However, the good match outside the barrier layers indicates that this is likely a post-processing artifact, resulting from plotting a piecewise-constant function on the intersection of the one-dimensional diagonal with the three-dimensional grid, which is rather coarse in the barrier layers. On the other hand, the match in the interface flux is very good, despite the coarse mortar meshes used in the flux-mortar method.

\begin{figure}
    \centering
    \begin{subfigure}{0.4\textwidth}
        \centering
        \includegraphics[width=0.9\textwidth]{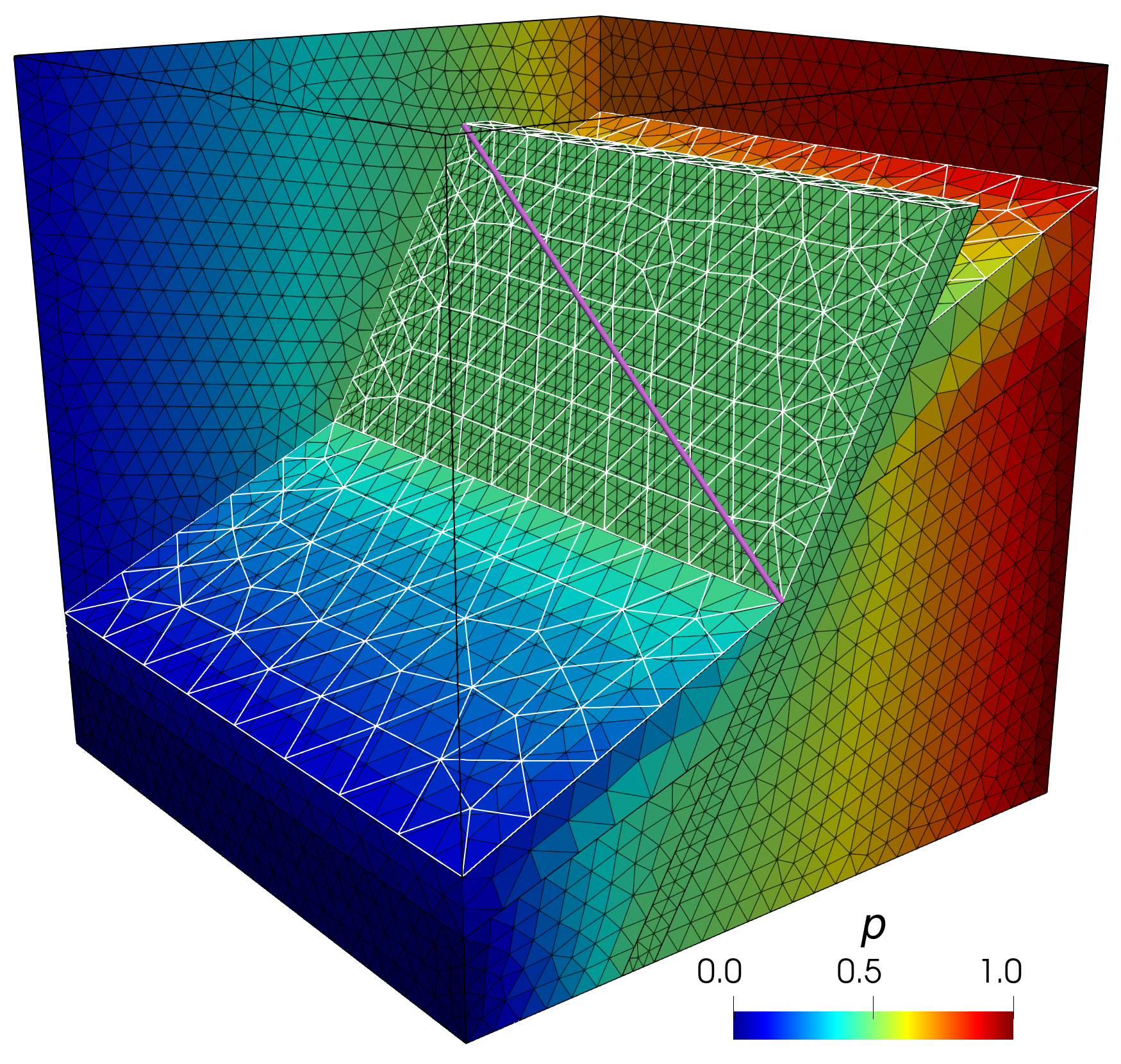}
        \caption{}
        \label{fig:fault_results_3d_pressure}
    \end{subfigure}
    \begin{subfigure}{0.4\textwidth}
        \centering
        \includegraphics[width=0.9\textwidth]{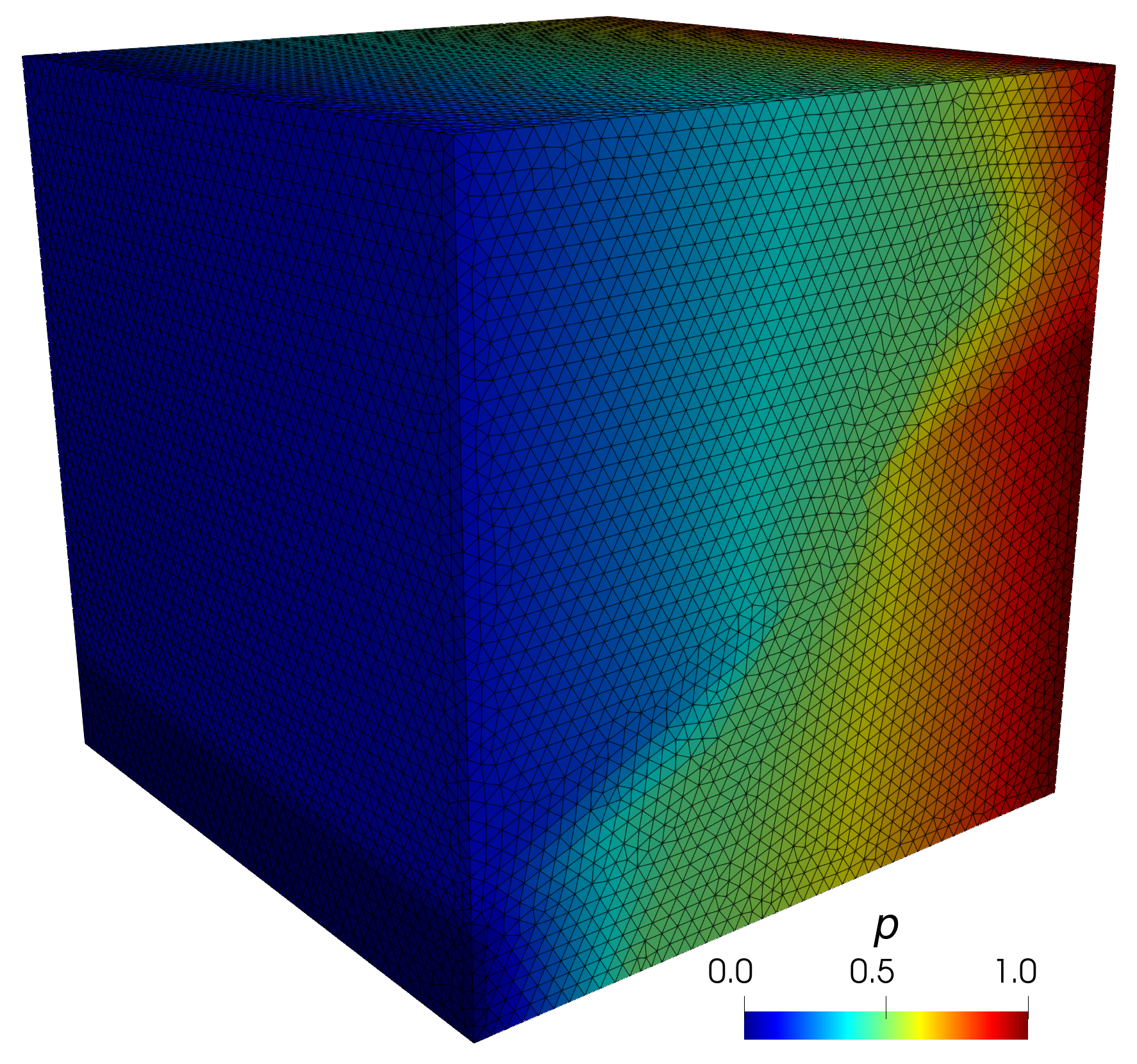}
        \caption{}
    \end{subfigure}
    \begin{subfigure}{0.4\textwidth}
        \centering
        \includegraphics[width=0.9\textwidth]{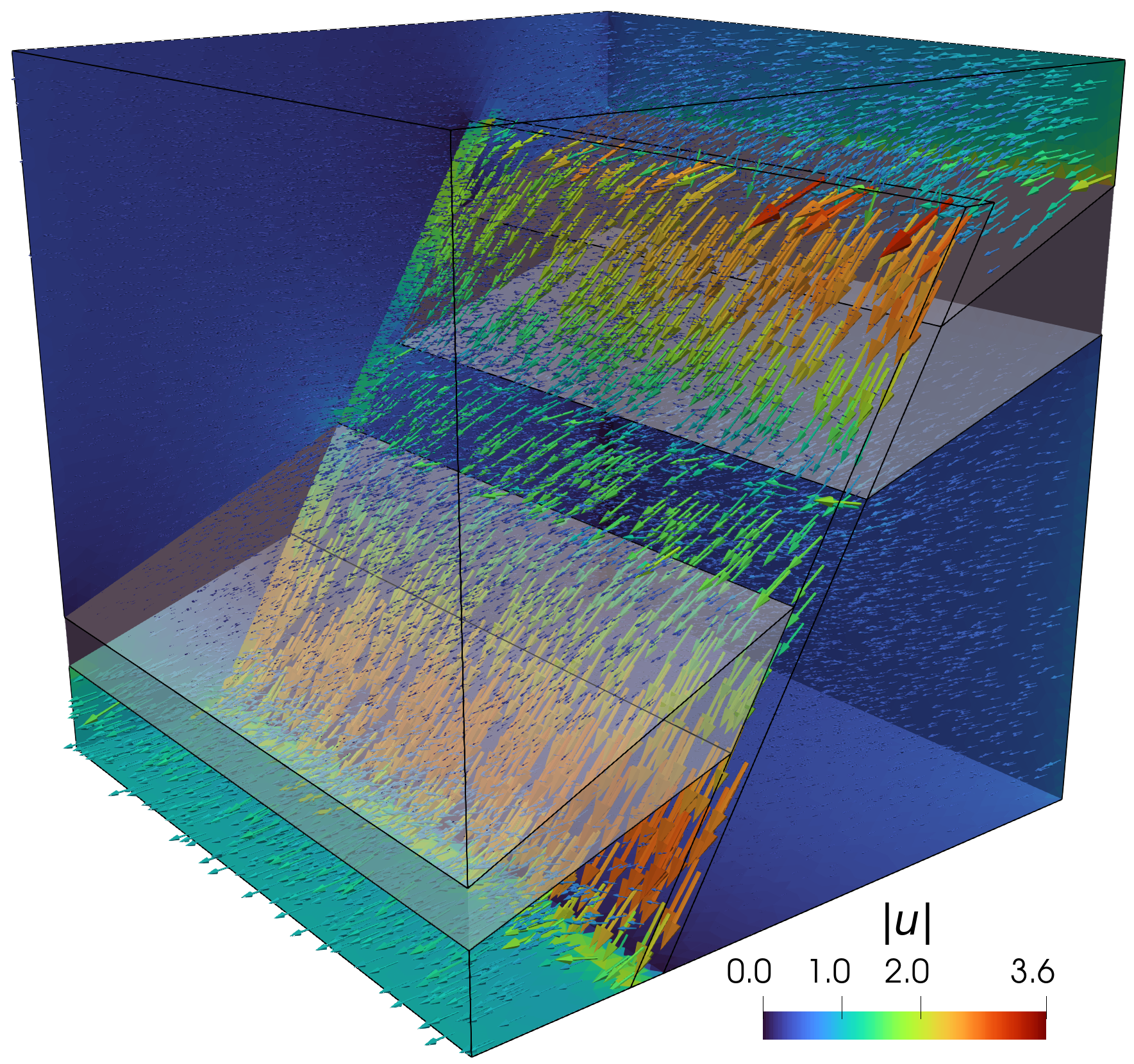}
        \caption{}
        \label{fig:fault_results_3d_velocity}
    \end{subfigure}
    \begin{subfigure}{0.4\textwidth}
        \centering
        \includegraphics[width=0.9\textwidth]{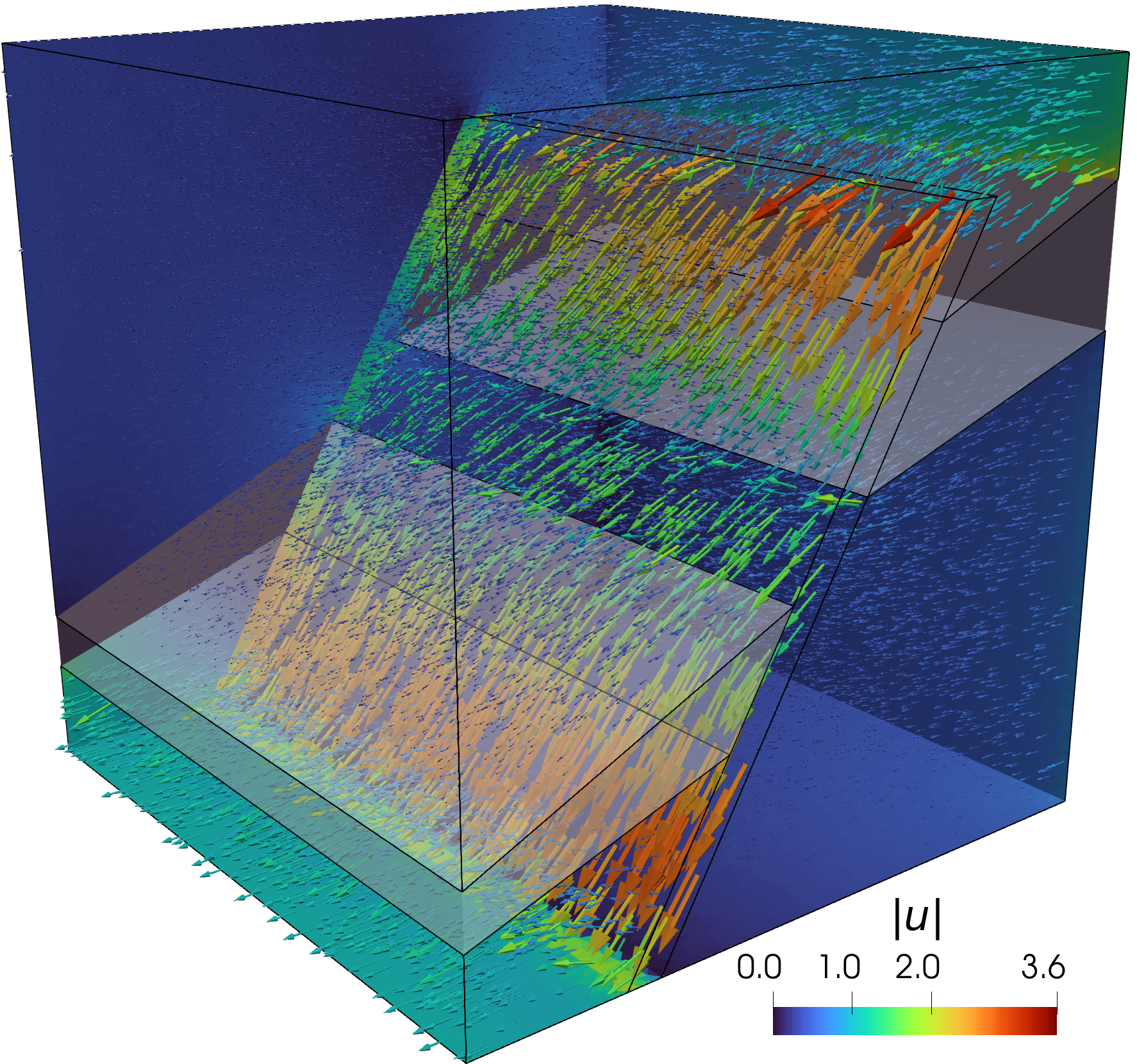}
        \caption{}
    \end{subfigure}
    \caption{Example~2 (3D): computed pressure (top) and velocity (bottom).
      Left column: flux-mortar solution. Right column: fine scale reference solution. The line depicted in \Cref{fig:fault_results_3d_pressure} illustrates the segment along which the plots of~\Cref{fig:fault_plots_3d} were created.}
    \label{fig:fault_results_3d}
\end{figure}

\begin{figure}
    \centering
    \begin{subfigure}{0.4\textwidth}
        \includegraphics[width=\textwidth]{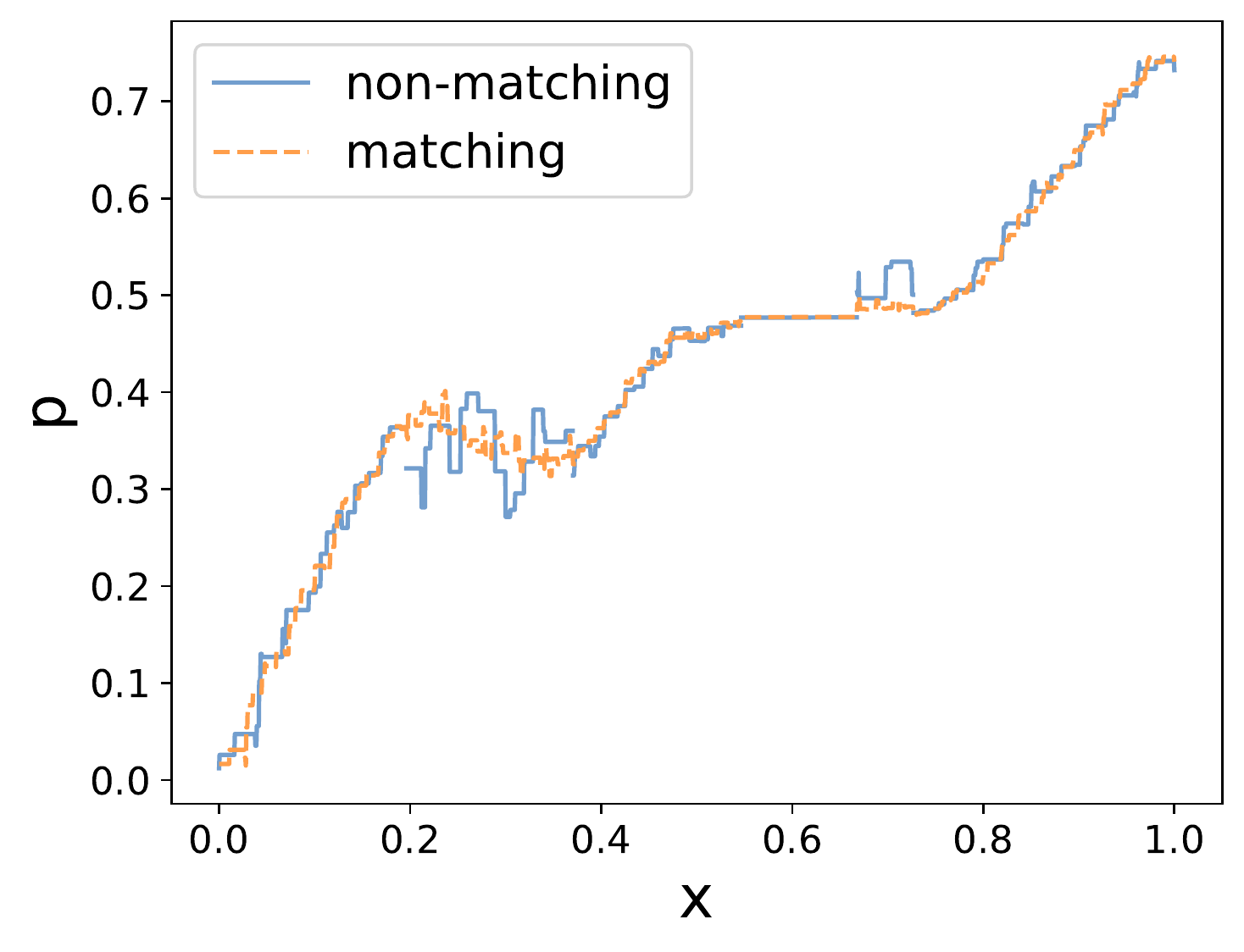}
        \caption{}
    \end{subfigure}
    \begin{subfigure}{0.4\textwidth}
        \includegraphics[width=\textwidth]{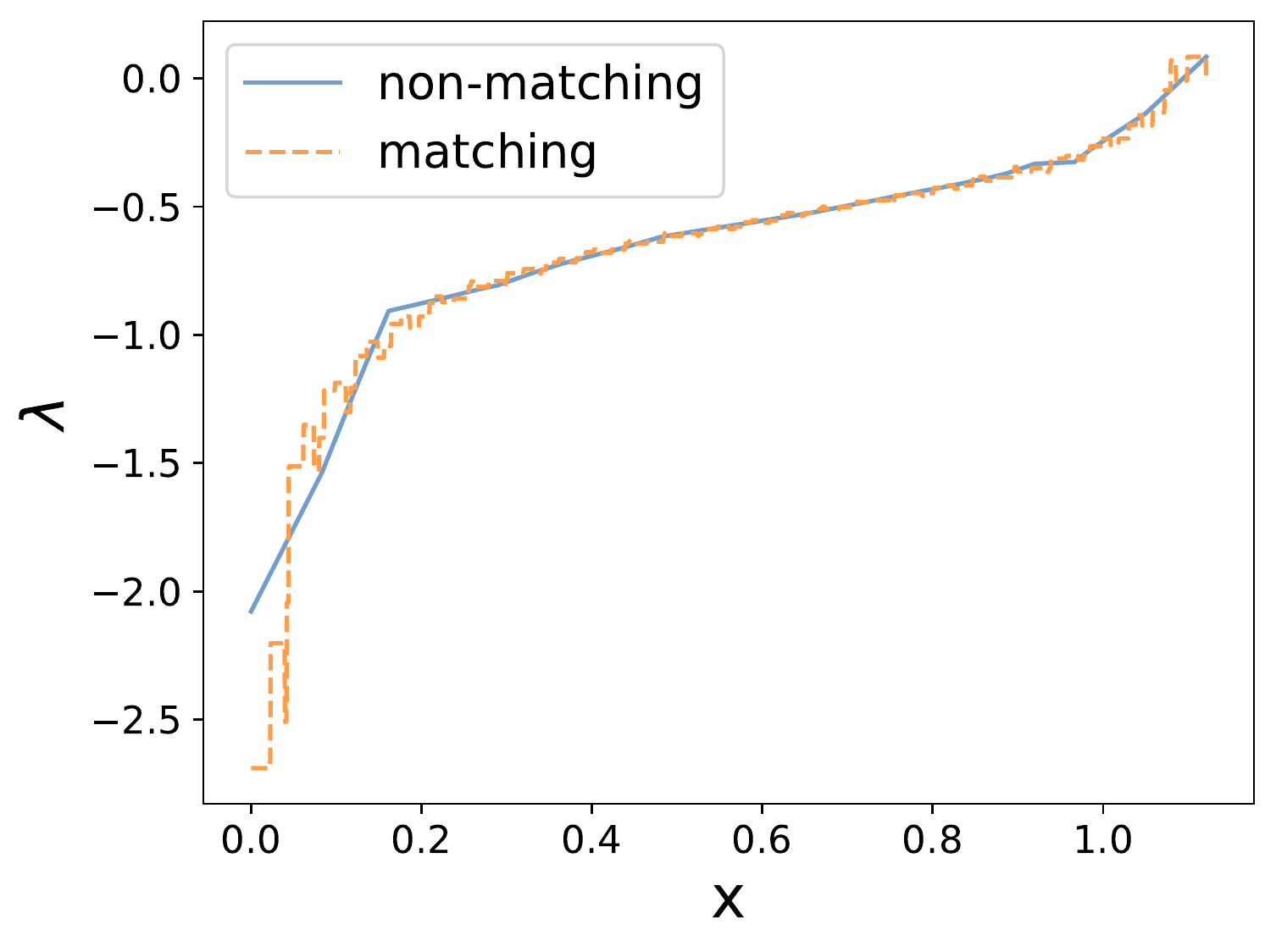}
        \caption{}
    \end{subfigure}
    \caption{Example~2 (3D): (a) Pressure along the diagonal of the domain; (b) Interface flux along the line depicted in~\Cref{fig:fault_results_3d_pressure}.}
    \label{fig:fault_plots_3d}
\end{figure}

\subsection{Example 3: Based on benchmark SPE10}
\label{sub:SPE10}

Following \cite[Example 2]{GanYot}, we consider a permeability field from the second data set of the Society of Petroleum Engineers (SPE) Comparative Solution Project SPE10
(see \href{https://www.spe.org/web/csp/datasets/set02.htm#download}{spe.org/csp/}).
The data set describes a two-dimensional permeability field that varies six orders of magnitude on a domain consisting of $60 \times 220$ cells, depicted in~\Cref{fig:spe_perm_and_results}. The goal of this example is to illustrate the multiscale capability of the flux-mortar method for highly heterogeneous porous media. To this end, we decompose the domain into $3 \times 5$ subdomains, which yields subdomain grids with $20 \times 44$ cells. We consider a coarse scale piecewise-linear mortar space with $10$ cells per interface. We impose a unit pressure drop from right to left with no-flow on the top and bottom boundaries.
As in the previous test case, we compare the results obtained from the flux-mortar method to a conforming fine scale solution. The pressure and velocity distributions obtained from the two methods are in very good agreement, as shown in~\Cref{fig:spe_perm_and_results}.

\begin{figure}[h]
    \centering
    \includegraphics[width=0.875\textwidth]{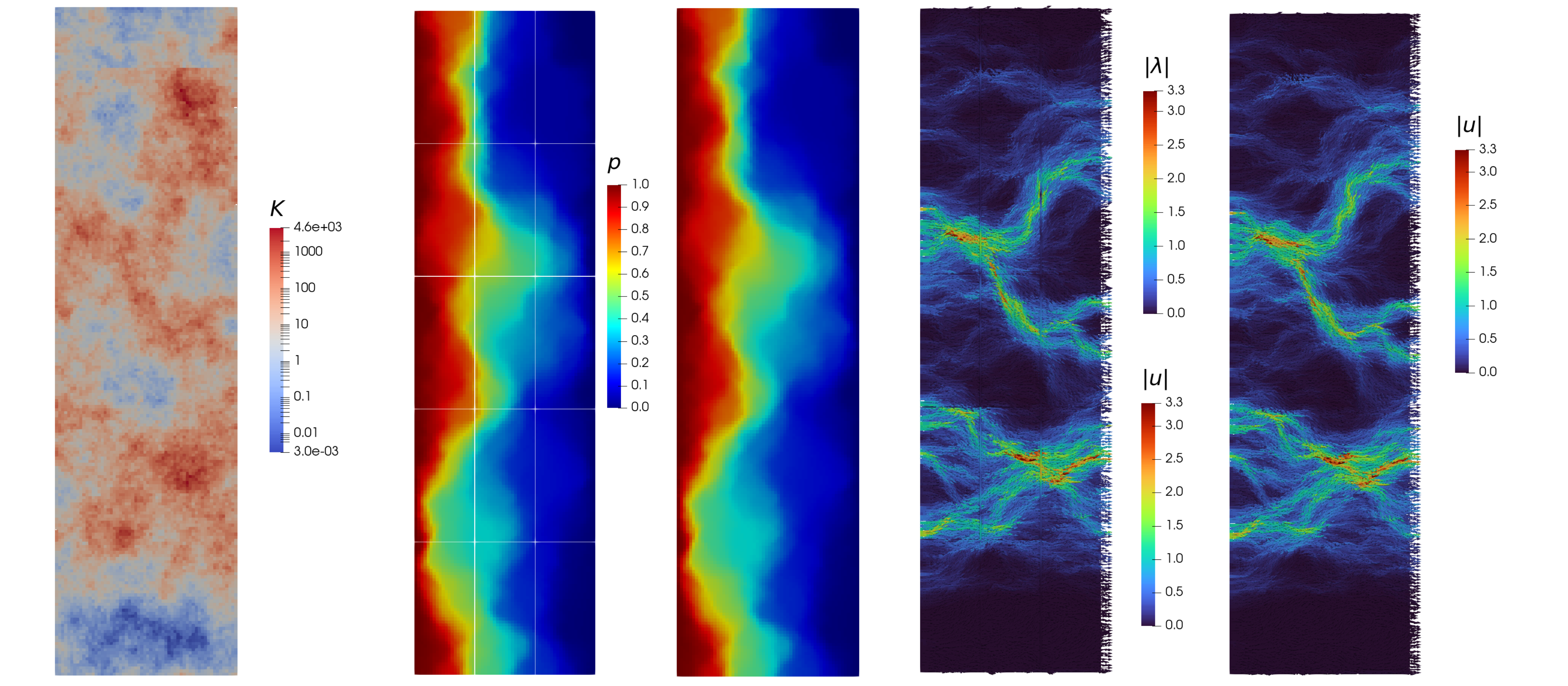}
    \caption{Example 3: permeability (left), pressure (center) and velocity (right). Pressure and velocity are depicted for both the flux-mortar solution (left of the legend) and a fine scale solution (right of the legend). In the flux-mortar solution, the mortar interface is visualized by white lines, while the absolute value of the mortar flux is depicted in the visualization of the velocity as tubes.}
    \label{fig:spe_perm_and_results}
\end{figure}

\subsection{Example 4: Locally adapted grids} 
\label{sub:highly_oscillatory_permeability}

In this example we choose the subdomain mesh size according to
the local spatial frequency of the permeability field to illustrate the method's flexibility of refining the grid locally where needed.
We again consider a permeability field from the second data set of the SPE10 benchmark. For $K_{xx}$ the permeability given in the dataset, we define the permeability
tensor $\bm{K} = \bm{R}^{-1} \left( \begin{matrix} K_{xx} & 0 \\ 0 &  0.2 K_{xx}
\end{matrix} \right) \bm{R}$, with $\bm{R}$ being the two-dimensional rotation matrix in clockwise direction around an angle of \SI{20}{\degree}. This results in a full-tensor permeability, which is correctly handled by the MFMFE discretization.
A square region with $60 \times 60$ cells of the original permeability data set is chosen as the domain of interest, which is further decomposed into $4 \times 4$ subdomains. Depending on the spatial frequency of the permeability, the subdomain
grids are refined between $0$ and $3$ times. This procedure results in $5$ of the
$25$ subdomains being refined (see~\Cref{fig:oscillatory_perm_and_grid}) and a total
of \num{36675} cells. Each interface is discretized with $10$ cells and piecewise-linear mortars. For comparison, a monolithic solution is computed on a conforming mesh with a cell size corresponding to that of the third level, yielding a mesh with \num{230400} cells.

A unit pressure drop is applied to the left and right domain boundaries, while a no-flow condition is imposed on the top and bottom boundaries. \Cref{fig:oscillatory_results} depicts the pressure and velocity distributions
obtained from the flux-mortar method and on the conforming fine grid. The two solutions are in very good agreement despite the large differences in the number of cells of the discretizations. In particular, the flux-mortar method captures very well the high velocities occurring in the highly-permeable channels near the upper right corner. We do observe that in the lower left corner of the domain the locally coarser grid used in the flux-mortar method does not fully capture the velocity field. However, the flexibility of the method would allow for a further subdivision of the lower-left block, using a finer mesh where the highly-permeable channel is located. The generation of locally adapted subdomain and mortar grids could be automated with the use of \emph{a posteriori} error estimates, which is a topic of future research.

\begin{figure}
    \centering
    \begin{subfigure}{0.33\textwidth}
        \includegraphics[width=\textwidth]{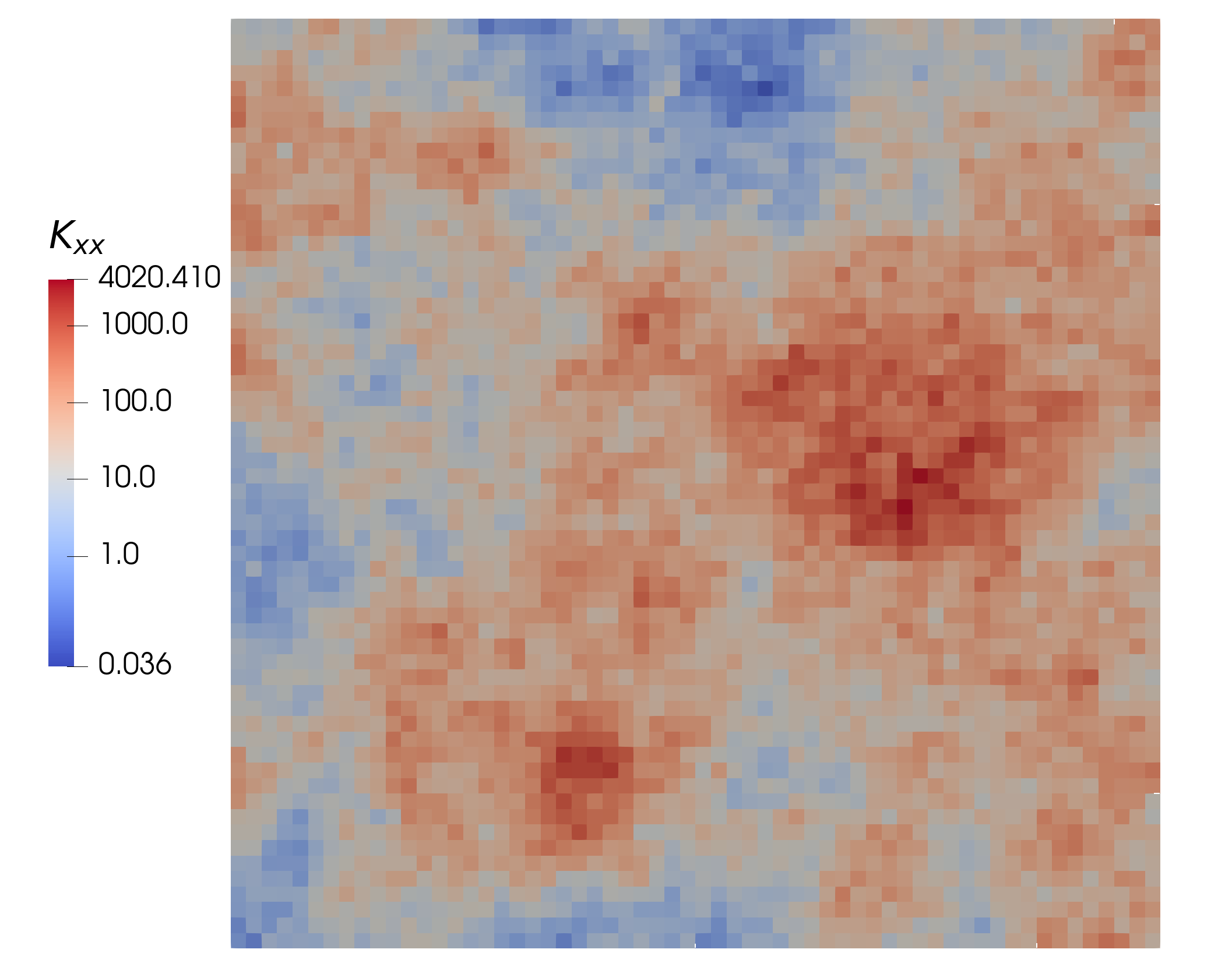}
        \caption{}
    \end{subfigure}
    \begin{subfigure}{0.33\textwidth}
        \includegraphics[width=\textwidth]{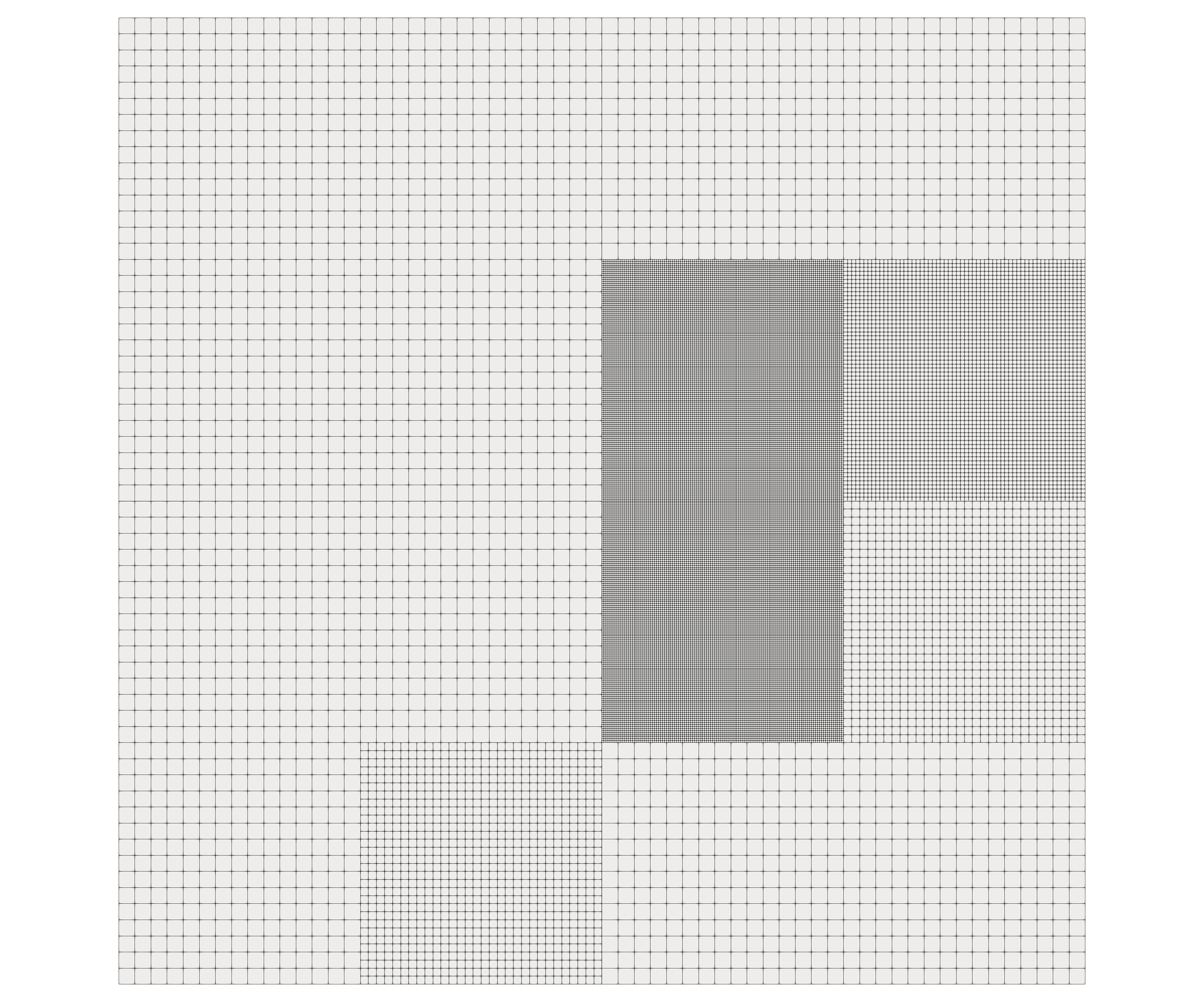}
        \caption{}
    \end{subfigure}
    \caption{Example 4: permeability distribution (left) and locally refined grids (right).}
    \label{fig:oscillatory_perm_and_grid}
\end{figure}

\begin{figure}
    \centering
    \begin{subfigure}{0.49\textwidth}
        \centering
        \includegraphics[width=0.825\textwidth]{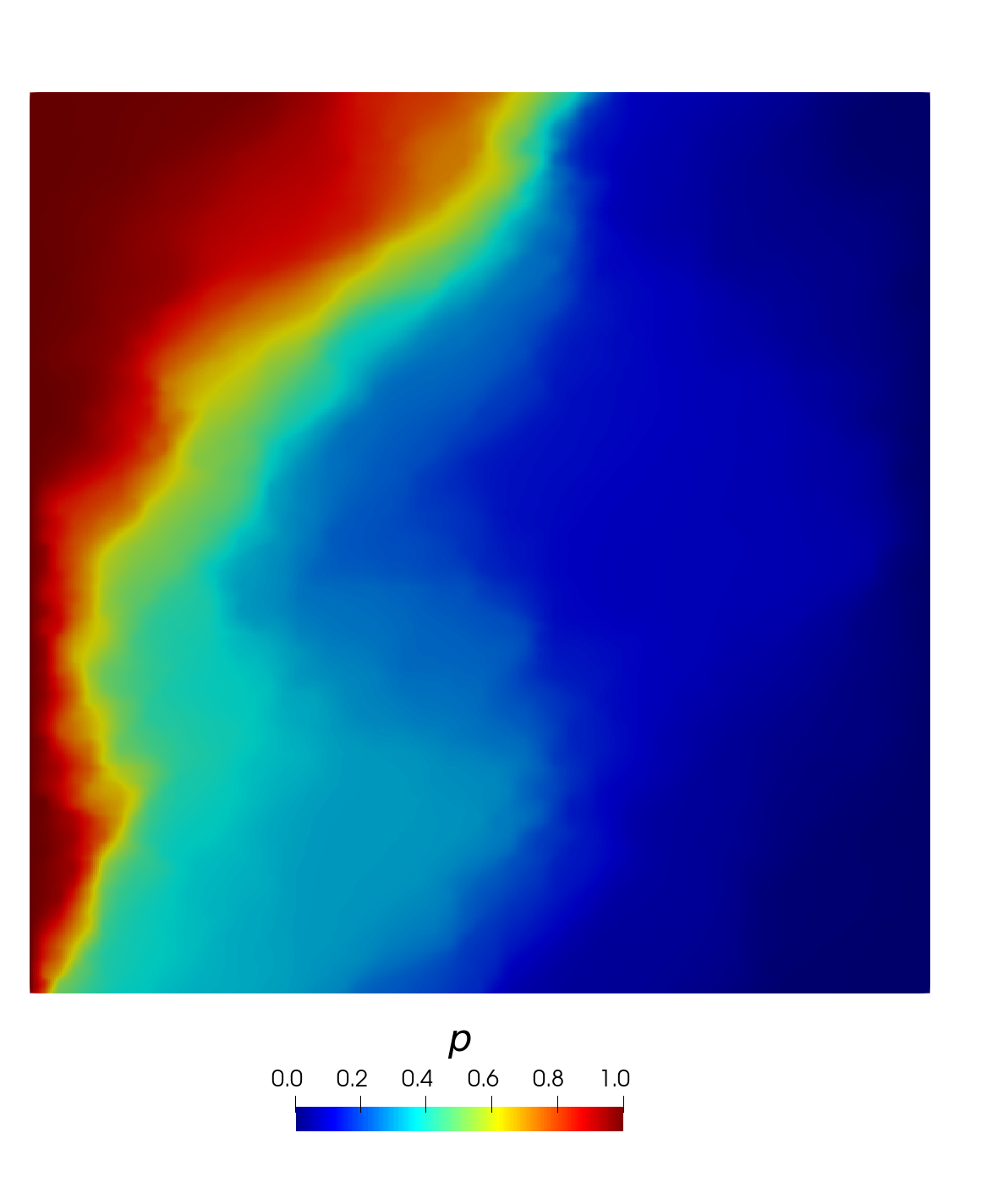}
        \caption{}
    \end{subfigure}
    \begin{subfigure}{0.49\textwidth}
        \centering
        \includegraphics[width=0.825\textwidth]{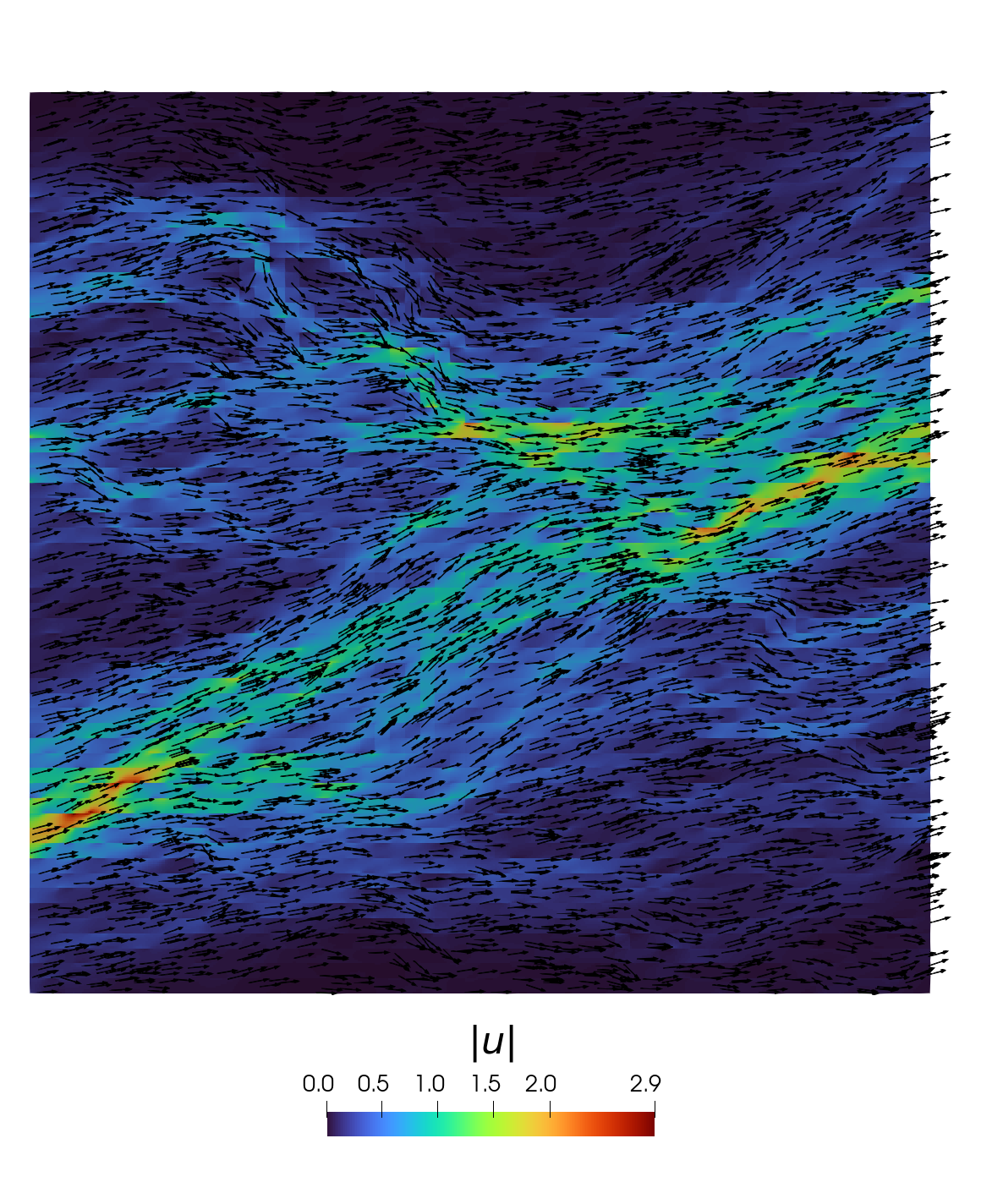}
        \caption{}
    \end{subfigure}
    \begin{subfigure}{0.49\textwidth}
        \centering
        \includegraphics[width=0.8\textwidth]{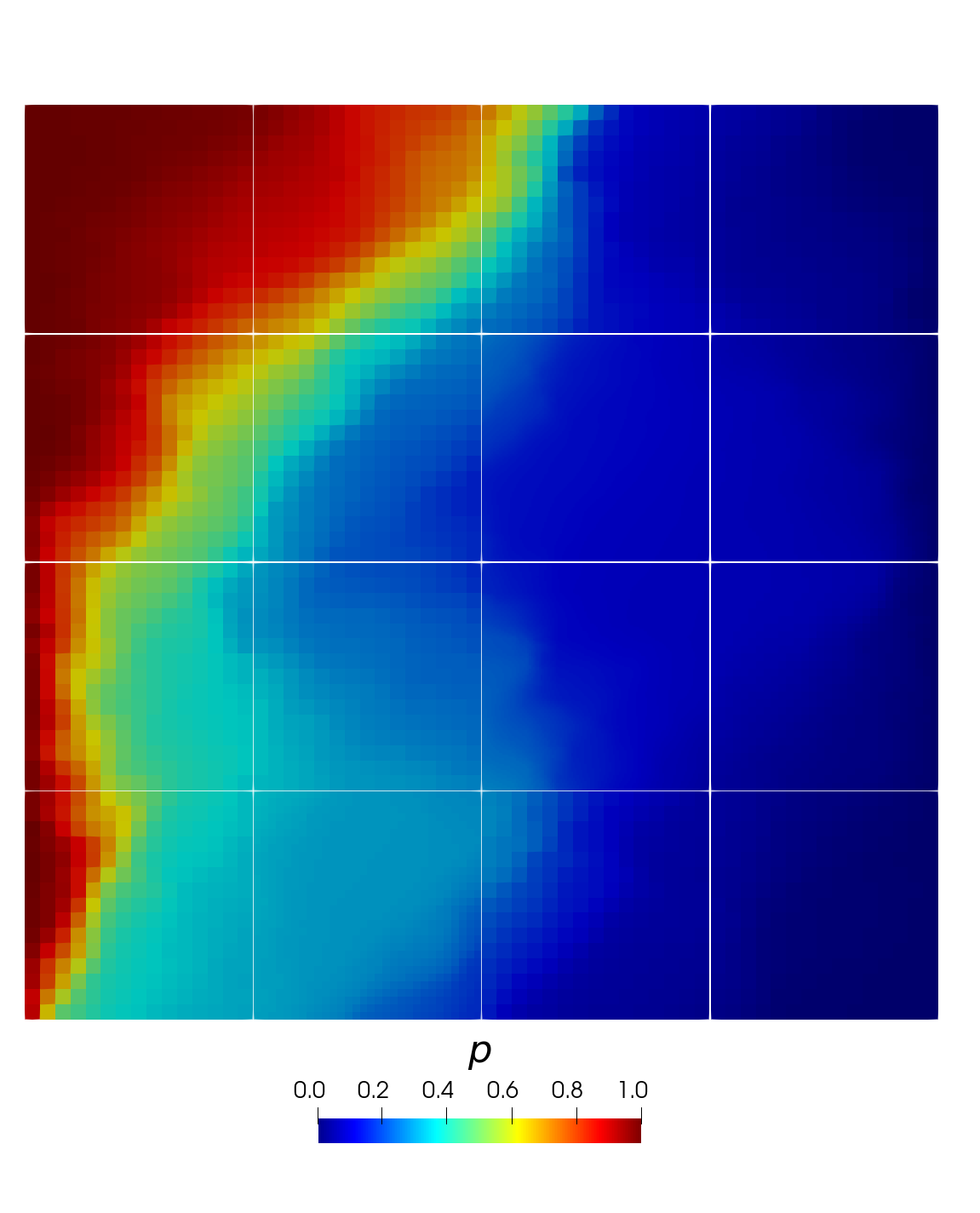}
        \caption{}
    \end{subfigure}
    \begin{subfigure}{0.49\textwidth}
        \centering
        \includegraphics[width=0.8\textwidth]{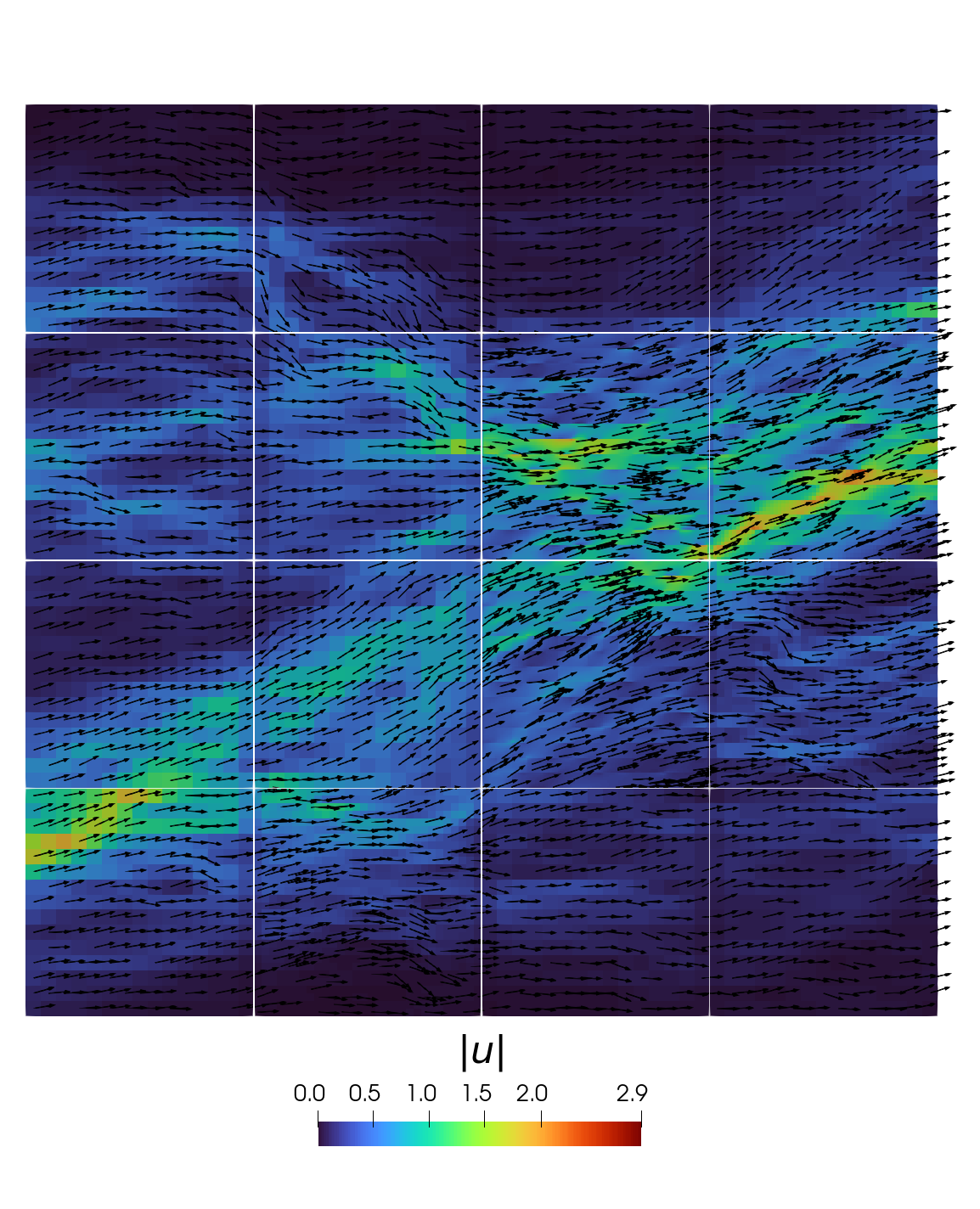}
        \caption{}
    \end{subfigure}
    \caption{Example 4: computed pressure (left column) and velocity (right column). The top row depicts the solution from a monolithic solve on a fine grid with $480 \times 480$ cells, while the bottom row shows the flux-mortar solution with $4 \times 4$ subdomains and locally refined grids (see~\Cref{fig:oscillatory_perm_and_grid}).}
    \label{fig:oscillatory_results}
\end{figure}

\fi

\section{Conclusions}
\label{sec: Conclusions}

We proposed the flux-mortar MFMFE method (\Cref{sec: Numerical method}) that uses the MPFA method as the subdomain discretization in a flux-mortar domain decomposition setting. The \emph{a priori} analysis shows that an additional quadrature error arises from the multipoint flux approximation. However, this term decays linearly with the mesh size, for simplicial (\Cref{sec: Analysis simplices}), $h^2$-parallelogram, and $h^2$-parallelepiped grids (\Cref{sec: Analysis quads}). In turn, the flux-mortar MFMFE method converges linearly in both flux and pressure. In \Cref{sec: DD}, we showed how the system can be reduced to an interface problem and proposed a Dirichlet-to-Neumann operator as preconditioner.
The numerical experiments of \Cref{sec: numerical results} verify the theoretical results and moreover show that the method provides the flexibility to handle general grids and complex porous media flow problems containing, for example, discontinuous and highly heterogeneous permeability fields.

As noted in \Cref{rem:polytopes}, the analysis can be further extended to general quadrilaterals and hexahedra using a non-symmetric quadrature rule \cite{klausen2006robust,WheXueYot}, as well as to general polytopes through the MFD interpretation of the MPFA method \cite{LSY}. These topics will be considered in future research.

\section*{Acknowledgments}
This project has received funding from the European Union's Horizon 2020 research and innovation programme under the Marie Skłodowska-Curie grant agreement No. 101031434 -- MiDiROM, from the Deutsche Forschungsgemeinschaft (DFG, German Research Foundation) under SFB 1313, Project Number 327154368, and from the U.S. National Science Foundation under grant DMS 2111129.

\bibliographystyle{abbrv}
\bibliography{bibliography}

\begin{thebibliography}{10}

\bibitem{Aarnes-etal}
J.~E. Aarnes, S.~Krogstad, and K.-A. Lie.
\newblock Multiscale mixed/mimetic methods on corner-point grids.
\newblock {\em Comput. Geosci.}, 12(3):297--315, 2008.

\bibitem{aavatsmark2002introduction}
I.~Aavatsmark.
\newblock An introduction to multipoint flux approximations for quadrilateral
  grids.
\newblock {\em Comput. Geosci.}, 6(3-4):405--432, 2002.
\newblock Locally conservative numerical methods for flow in porous media.

\bibitem{aavatsmark1998discretization}
I.~Aavatsmark, T.~Barkve, O.~B{\o}e, and T.~Mannseth.
\newblock Discretization on unstructured grids for inhomogeneous, anisotropic
  media. {{I}}. {D}erivation of the methods.
\newblock {\em SIAM J. Sci. Comput.}, 19(5):1700--1716, 1998.

\bibitem{aavatsmark2007convergence}
I.~Aavatsmark, G.~T. Eigestad, R.~A. Klausen, M.~F. Wheeler, and I.~Yotov.
\newblock Convergence of a symmetric {MPFA} method on quadrilateral grids.
\newblock {\em Comput. Geosci.}, 11(4):333--345, 2007.

\bibitem{AhmFumBud}
E.~Ahmed, A.~Fumagalli, and A.~Budi\v{s}a.
\newblock A multiscale flux basis for mortar mixed discretizations of reduced
  {D}arcy-{F}orchheimer fracture models.
\newblock {\em Comput. Methods Appl. Mech. Engrg.}, 354:16--36, 2019.

\bibitem{AKYZ}
I.~Ambartsumyan, E.~Khattatov, I.~Yotov, and P.~Zunino.
\newblock A {L}agrange multiplier method for a {S}tokes-{B}iot
  fluid-poroelastic structure interaction model.
\newblock {\em Numer. Math.}, 140(2):513--553, 2018.

\bibitem{arbogast2004analysis}
T.~Arbogast.
\newblock Analysis of a two-scale, locally conservative subgrid upscaling for
  elliptic problems.
\newblock {\em SIAM Journal on Numerical Analysis}, 42(2):576--598, 2004.

\bibitem{ACWY}
T.~Arbogast, L.~C. Cowsar, M.~F. Wheeler, and I.~Yotov.
\newblock Mixed finite element methods on nonmatching multiblock grids.
\newblock {\em SIAM J. Numer. Anal.}, 37(4):1295--1315, 2000.

\bibitem{APWY}
T.~Arbogast, G.~Pencheva, M.~F. Wheeler, and I.~Yotov.
\newblock A multiscale mortar mixed finite element method.
\newblock {\em Multiscale Model. Simul.}, 6(1):319--346, 2007.

\bibitem{ArshadParkShin-parabolic}
M.~Arshad, E.-J. Park, and D.~Shin.
\newblock Multiscale mortar mixed domain decomposition approximations of
  nonlinear parabolic equations.
\newblock {\em Comput. Math. Appl.}, 97:375--385, 2021.

\bibitem{ArshadParkShin-elliptic}
M.~Arshad, E.-J. Park, and D.-w. Shin.
\newblock Analysis of multiscale mortar mixed approximation of nonlinear
  elliptic equations.
\newblock {\em Comput. Math. Appl.}, 75(2):401--418, 2018.

\bibitem{BLSWY}
M.~Berndt, K.~Lipnikov, M.~Shashkov, M.~F. Wheeler, and I.~Yotov.
\newblock A mortar mimetic finite difference method on non-matching grids.
\newblock {\em Numer. Math.}, 102(2):203--230, 2005.

\bibitem{boffi2013mixed}
D.~Boffi, F.~Brezzi, and M.~Fortin.
\newblock {\em Mixed finite element methods and applications}, volume~44.
\newblock Springer, 2013.

\bibitem{Boon2019StokesDarcy}
W.~M. Boon.
\newblock A parameter-robust iterative method for {S}tokes-{D}arcy problems
  retaining local mass conservation.
\newblock {\em ESAIM Math. Model. Numer. Anal.}, 54(6):2045--2067, 2020.

\bibitem{boon2020flux}
W.~M. Boon, D.~Gl\"{a}ser, R.~Helmig, and I.~Yotov.
\newblock Flux-mortar mixed finite element methods on nonmatching grids.
\newblock {\em SIAM J. Numer. Anal.}, 60(3):1193--1225, 2022.

\bibitem{boon2018robust}
W.~M. Boon, J.~M. Nordbotten, and I.~Yotov.
\newblock Robust discretization of flow in fractured porous media.
\newblock {\em SIAM J. Numer. Anal.}, 56(4):2203--2233, 2018.

\bibitem{BDDF}
F.~Brezzi, J.~Douglas, Jr., R.~Dur\'an, and M.~Fortin.
\newblock Mixed finite elements for second order elliptic problems in three
  variables.
\newblock {\em Numer. Math.}, 51(2):237--250, 1987.

\bibitem{BDM}
F.~Brezzi, J.~Douglas, Jr., and L.~D. Marini.
\newblock Two families of mixed finite elements for second order elliptic
  problems.
\newblock {\em Numer. Math.}, 47(2):217--235, 1985.

\bibitem{ChungEfeLee}
E.~T. Chung, Y.~Efendiev, and C.~S. Lee.
\newblock Mixed generalized multiscale finite element methods and applications.
\newblock {\em Multiscale Model. Simul.}, 13(1):338--366, 2015.

\bibitem{duran2019multiscale}
O.~Duran, P.~R. Devloo, S.~M. Gomes, and F.~Valentin.
\newblock A multiscale hybrid method for darcy’s problems using mixed finite
  element local solvers.
\newblock {\em Computer methods in applied mechanics and engineering},
  354:213--244, 2019.

\bibitem{MHM-two-phase}
O.~Duran, P.~R.~B. Devloo, S.~M. Gomes, and J.~Villegas.
\newblock A multiscale mixed finite element method applied to the simulation of
  two-phase flows.
\newblock {\em Comput. Methods Appl. Mech. Engrg.}, 383:Paper No. 113870, 23,
  2021.

\bibitem{edwards2002unstructured}
M.~G. Edwards.
\newblock Unstructured, control-volume distributed, full-tensor finite-volume
  schemes with flow based grids.
\newblock {\em Comput. Geosci.}, 6(3-4):433--452, 2002.
\newblock Locally conservative numerical methods for flow in porous media.

\bibitem{edwards1998finite}
M.~G. Edwards and C.~F. Rogers.
\newblock Finite volume discretization with imposed flux continuity for the
  general tensor pressure equation.
\newblock {\em Comput. Geosci.}, 2(4):259--290 (1999), 1998.

\bibitem{GanYot}
B.~Ganis and I.~Yotov.
\newblock Implementation of a mortar mixed finite element method using a
  multiscale flux basis.
\newblock {\em Comput. Methods Appl. Mech. Engrg.}, 198(49-52):3989--3998,
  2009.

\bibitem{GSWY}
V.~Girault, S.~Sun, M.~F. Wheeler, and I.~Yotov.
\newblock Coupling discontinuous {G}alerkin and mixed finite element
  discretizations using mortar finite elements.
\newblock {\em SIAM J. Numer. Anal.}, 46(2):949--979, 2008.

\bibitem{GVY}
V.~Girault, D.~Vassilev, and I.~Yotov.
\newblock Mortar multiscale finite element methods for {S}tokes-{D}arcy flows.
\newblock {\em Numer. Math.}, 127(1):93--165, 2014.

\bibitem{GW}
R.~Glowinski and M.~F. Wheeler.
\newblock Domain decomposition and mixed finite element methods for elliptic
  problems.
\newblock In R.~Glowinski, G.~H. Golub, G.~A. Meurant, and J.~Periaux, editors,
  {\em First International Symposium on Domain Decomposition Methods for
  Partial Differential Equations}, pages 144--172. SIAM, Philadelphia, 1988.

\bibitem{IngWheYot}
R.~Ingram, M.~F. Wheeler, and I.~Yotov.
\newblock A multipoint flux mixed finite element method on hexahedra.
\newblock {\em SIAM J. Numer. Anal.}, 48(4):1281--1312, 2010.

\bibitem{JKY}
M.~Jayadharan, E.~Khattatov, and I.~Yotov.
\newblock Domain decomposition and partitioning methods for mixed finite
  element discretizations of the {B}iot system of poroelasticity.
\newblock {\em Comput. Geosci.}, 25(6):1919--1938, 2021.

\bibitem{KhatYot}
E.~Khattatov and I.~Yotov.
\newblock Domain decomposition and multiscale mortar mixed finite element
  methods for linear elasticity with weak stress symmetry.
\newblock {\em ESAIM Math. Model. Numer. Anal.}, 53(6):2081--2108, 2019.

\bibitem{klausen2006robust}
R.~A. Klausen and R.~Winther.
\newblock Robust convergence of multi point flux approximation on rough grids.
\newblock {\em Numer. Math.}, 104(3):317--337, 2006.

\bibitem{Kochetal2020Dumux}
T.~Koch, D.~Gl\"aser, K.~Weishaupt, S.~Ackermann, M.~Beck, B.~Becker,
  S.~Burbulla, H.~Class, E.~Coltman, S.~Emmert, T.~Fetzer, C.~Gr\"uninger,
  K.~Heck, J.~Hommel, T.~Kurz, M.~Lipp, F.~Mohammadi, S.~Scherrer,
  M.~Schneider, G.~Seitz, L.~Stadler, M.~Utz, F.~Weinhardt, and B.~Flemisch.
\newblock {DuMu\textsuperscript{x} 3 - an open-source simulator for solving
  flow and transport problems in porous media with a focus on model coupling}.
\newblock {\em Comput. Math. with Appl.}, 2020.

\bibitem{LSY}
K.~Lipnikov, M.~Shashkov, and I.~Yotov.
\newblock Local flux mimetic finite difference methods.
\newblock {\em Numer. Math.}, 112(1):115--152, 2009.

\bibitem{CVD-MPFA}
E.~Parramore, M.~G. Edwards, M.~Pal, and S.~Lamine.
\newblock Multiscale finite-volume {CVD}-{MPFA} formulations on structured and
  unstructured grids.
\newblock {\em Multiscale Model. Simul.}, 14(2):559--594, 2016.

\bibitem{PenYot}
G.~Pencheva and I.~Yotov.
\newblock Balancing domain decomposition for mortar mixed finite element
  methods.
\newblock {\em Numer. Linear Algebra Appl.}, 10(1-2):159--180, 2003.

\bibitem{PWY}
M.~Peszy\'{n}ska, M.~F. Wheeler, and I.~Yotov.
\newblock Mortar upscaling for multiphase flow in porous media.
\newblock {\em Comput. Geosci.}, 6(1):73--100, 2002.

\bibitem{Scott-Zhang}
L.~R. Scott and S.~Zhang.
\newblock Finite element interpolation of nonsmooth functions satisfying
  boundary conditions.
\newblock {\em Math. Comput.}, 54(190):483--493, 1990.

\bibitem{Toselli-Widlund}
A.~Toselli and O.~Widlund.
\newblock {\em Domain decomposition methods---algorithms and theory}, volume~34
  of {\em Springer Series in Computational Mathematics}.
\newblock Springer-Verlag, Berlin, 2005.

\bibitem{VWY}
D.~Vassilev, C.~Wang, and I.~Yotov.
\newblock Domain decomposition for coupled {S}tokes and {D}arcy flows.
\newblock {\em Comput. Methods Appl. Mech. Engrg.}, 268:264--283, 2014.

\bibitem{WheXueYot}
M.~F. Wheeler, G.~Xue, and I.~Yotov.
\newblock A multipoint flux mixed finite element method on distorted
  quadrilaterals and hexahedra.
\newblock {\em Numer. Math.}, 121(1):165--204, 2012.

\bibitem{wheeler2012multiscale}
M.~F. Wheeler, G.~Xue, and I.~Yotov.
\newblock A multiscale mortar multipoint flux mixed finite element method.
\newblock {\em ESAIM Math. Model. Numer. Anal.}, 46(4):759--796, 2012.

\bibitem{WheYot}
M.~F. Wheeler and I.~Yotov.
\newblock A multipoint flux mixed finite element method.
\newblock {\em SIAM J. Numer. Anal.}, 44(5):2082--2106, 2006.

\end{thebibliography}

\end{document}